\documentclass[11pt]{article}
\usepackage{inputenc,amsfonts,amsmath,amssymb}
\usepackage[all]{xy}
\usepackage{amsthm}
\usepackage[colorlinks=true, linkcolor=blue, citecolor=black!50!green, urlcolor=blue]{hyperref}
\usepackage[left=2.8cm, right=2.8cm, top=2cm]{geometry}
\usepackage{extarrows}
\usepackage{graphicx,subcaption}
\usepackage{wrapfig}
\usepackage{multirow}
\usepackage{tikz-cd}
\usepackage{sectsty}
\usepackage{authblk}
\usepackage{aliascnt}

\numberwithin{equation}{section}
\sectionfont{\normalfont\fontsize{15}{15}\selectfont}

\usepackage{aliascnt}
\newcommand{\thdef}[2]{
	\newaliascnt{#1}{theorem}  
	\newtheorem{#1}[#1]{#2}
	\aliascntresetthe{#1}  
	\newtheorem*{#1*}{#2}
	
    \expandafter\newcommand\expandafter{\csname #1autorefname\endcsname}{#2}
    
}

\def \CC {{\mathbb C}}
\def \ZZ {{\mathbb Z}}

\def \TT {{\mathbb T}}

\def \QQ {{\mathbb Q}}

\def \calA  {{\mathcal{A}}}           

\def \calL  {{\mathcal{L}}}
\def \calQ  {{\mathcal{Q}}}
\def \calP {{\mathcal{P}}}

\def \calS {{\mathcal{S}}}
\def \calT {{\mathcal{T}}}

\def \sL {{\scriptscriptstyle L}}

\def \t  {\mathfrak{t}}

\def \hs {\hspace{.2in}}
\def \hand {\hs \mbox{and} \hs}

\def \btheta {{\boldsymbol{\theta}}}

\def \TT {{\mathbb{T}}}

\def \mfX {\mathfrak{X}}

\def \ep {\varepsilon}
\def \bfu {{\bf u}}
\def \mH {{\mathsf{H}}}

\def \bnu {\boldsymbol{\nu}}

\def \bsigma {{\boldsymbol{\sigma}}}

\def \bsigma {{\boldsymbol{\sigma}}}

\def \blam {{\boldsymbol{\lambda}}}

\def \bth {{\btheta}}

\def \by {{\bf y}}

\def \bfk {{\bf k}}
\def \bfy {{\bf y}}
\def \GY {{\scriptscriptstyle {\rm GY}}}
\def \T {{\mathbb{T}}}
\def \TT {{\mathbb{T}}}
\def \bk {{\bf k}}
\def \calU {{\mathcal{U}}}
\def \FF {{\mathbb{F}}}

\usepackage{tikz-cd,pgfplots}
\usetikzlibrary{angles,quotes}
\usetikzlibrary{calc}

\def \taub {{\tau_\bullet}}
\def \taubi {{\tau_\bullet^{-1}}}
\def \sA {{\scriptstyle A}}
\def \bida {{\bf i}^\dagger}

\def \bfi {{\bf i}}
\def \TTA {{\TT_\sA}}
\def \ida {i^\dagger}
\def \Srn {{\mathfrak{S}}_{r, n}}
\def \Trn {{\mathfrak{T}}_{r, n}}
\def \mH {{\mathsf{H}}}

\def \kto {{k_\tau^{(1)}}}
\def \ktt {{k_\tau^{(2)}}}

\def \jto {{j_\tau^{(1)}}}
\def \jtt {{j_\tau^{(2)}}}

\def \whm {\widehat{m}}

\newtheorem{theorem}{Theorem}[subsection]

\newtheorem{theoremalpha}{Theorem}[subsection]

\thdef{lemma}{Lemma}
\thdef{proposition}{Proposition}
\thdef{corollary}{Corollary}
\thdef{conjecture}{Conjecture}
\theoremstyle{definition}
\thdef{definition}{Definition}
\thdef{example}{Example}
\thdef{remark}{Remark}
\thdef{lemma-definition}{Lemma-Definition}
\thdef{lemma-notation}{Lemma-Notation}
\thdef{notation}{Notation}
\thdef{assumption-notation}{Assumption-Notation}
\thdef{notation-lemma}{Notation-Lemma}
\thdef{definition-lemma}{Definition-Lemma}
\thdef{notation-remark}{Notation-Remark}
\thdef{definition-remark}{Definition-Remark}
\thdef{definition-notation}{Definition-Notation}
\thdef{remark-definition}{Remark-Definition}
\thdef{notation-definition}{Notation-Definition}

\title{Mutation matrices from Poisson CGL extensions}

\author{Zihang Liu \, and \, Jiang-Hua Lu \, and \, Yipeng Mi}

\AtEndDocument{%
  \par
  \medskip
  \begin{tabular}{@{}l@{}}%
    \small
    \textsc{Department of Mathematics, the Hong Kong University, Pokfulam Road, Hong Kong.}\\
    \textit{E-mail address}: \texttt{\href{zhliu22@connect.hku.hk}{zhliu22@connect.hku.hk}}
  \end{tabular}
  \par
  \medskip
  \begin{tabular}{@{}l@{}}%
    \small
    \textsc{Department of Mathematics, the Hong Kong University, Pokfulam Road, Hong Kong.}\\
    \textit{E-mail address}: \texttt{\href{jhlu@maths.hku.hk}{jhlu@maths.hku.hk}}
  \end{tabular}
    \par
  \medskip
  \begin{tabular}{@{}l@{}}%
    \small
    \textsc{Department of Mathematics, the Hong Kong University, Pokfulam Road, Hong Kong}\\ 
    {(address while research on this paper was carried out).}\\
    \textit{E-mail address}: \texttt{\href{mypkyle@gmail.com}{mypkyle@gmail.com}}
  \end{tabular}
  }
\date{}

\begin{document}

\maketitle
\begin{abstract}
Symmetric Poisson CGL extensions form a particular class of 
polynomial Poisson algebras that are shown by K. Goodearl and M. Yakimov to admit  compatible cluster structures.
In this paper, we give explicit formulas for a family of mutation matrices
in the Goodearl-Yakimov theory via matrix products as well as by entry-wise description.
\end{abstract}

\tableofcontents
\addtocontents{toc}{\protect\setcounter{tocdepth}{1}}

\section{Introduction and statements of main results}
\subsection{Introduction} Among the systematical examples of commutative and associative algebras admitting (meaningful) 
cluster structures \cite{FZ:I} are a class of polynomial Poisson algebras studied by K. Goodearl and M. Yakimov
in \cite{GY:PNAS, GY:Poi-CGL}. 
 
Let ${\bf k}$ be a field of characteristic $0$. A length $n$ iterated Poisson-Ore extension (of ${\bf k}$)
is a polynomial Poisson algebra $(\bk[x_1, \ldots, x_n], \{\, , \, \})$
such that 
\begin{equation}\label{eq:xxx-1}
\{\bk[x_{1}, \ldots, x_{j-1}], \; x_j\} \subset x_j \bk[x_{1}, \ldots, x_{j-1}] + \bk[x_{1}, \ldots, x_{j-1}], \hs 
\forall \; 2 \leq j \leq n.
\end{equation}
Such an iterated Poisson-Ore extension is said to be {\it symmetric} if it also satisfies
\begin{equation}\label{eq:zzz-0}
\{x_j, \; \bk[x_{j+1}, \ldots, x_n]\} \subset x_j \bk[x_{j+1}, \ldots, x_n] + \bk[x_{j+1}, \ldots, x_n], \hs \forall \; 1 \leq j \leq n-1.
\end{equation}
Given a split ${\bf k}$-torus $\TT$, a {\it $\TT$-Poisson CGL extension} (reps. {\it symmetric 
$\TT$-Poisson CGL extension}), as defined in \cite{GY:PNAS, GY:Poi-CGL} and named after 
G. Cauchon, K. Goodearl, and E. Letzter, is an iterated Poisson-Ore extension 
(resp.  symmetric iterated Poisson-Ore extension)
together with a compatible $\TT$-action  and certain nilpotency condition. See $\S$\ref{ss:main-intro}
for details.

In a remarkable theory developed in \cite{GY:PNAS, GY:Poi-CGL}, K. Goodearl and M. Yakimov showed that 
a presentation  of a Poisson algebra $(R, \{\, ,\, \})$ 
as a {\it symmetric} $\TT$-Poisson CGL extension 
gives rise to (under some mild conditions on scalars and Poisson CGL generators, see \autoref{thm:GY-Mtau} for detail)
a seed\footnote{Mutation matrices are denoted as 
$\widetilde{B}$ in \cite{GY:Poi-CGL}.}  $(\bfy, M)$ in the fraction field ${\rm Frac}(R)$ of $R$ such that
\begin{equation}\label{eq:RAU}
R = \overline{\calA}(\bfy,  M) = \overline{\calU}(\bfy,  M) \subset {\rm Frac}(R),
\end{equation}
where  $\overline{\calA}(\bfy,  M)$ and $\overline{{\calU}}(\bfy,  M)$ are respectively the 
cluster ${\bf k}$-algebra and the upper cluster ${\bf k}$-algebra defined
by $(\bfy,  M)$ with no frozen variables inverted (see \autoref{de:cluster-A}).
The seed $(\bfy,  M)$ is  a {\it $\TT$-Poisson} in 
 the sense of 
M. Gekhtman, M. Shapiro, and A. Vainstein
\cite{GSV:book}, i.e., 
the extended cluster $\bfy^\prime$ in any seed $(\bfy^\prime, M^\prime)$
mutation equivalent to $(\bfy,  M)$ has log-canonical Poisson bracket 
with respect to $\{\, , \, \}$ and consists of $\TT$-weight vectors. 
We denote by $\Sigma_{\GY}$ the mutation equivalence class of seeds in ${\rm Frac}(R)$ containing
$(\bfy, M)$.

In the theory of cluster algebras, it is often desirable to have a good and easily 
workable description of the mutation matrices,
as each one of them  completely determines the cluster structure up to isomorphisms.
For a  {\it symmetric} $\TT$-Poisson CGL extension 
$(R,\, \{\, , \, \})$ of length $n$,
Goodearl and Yakimov constructed in 
\cite{GY:Poi-CGL} (again under some mild conditions on scalars and Poisson CGL generators) 
a family of seeds $({\bf y}_\tau, M_\tau)$ in $\Sigma_{\GY}$,
 where
$\tau$ lies in a certain subset $\Xi_n$ of the symmetric group $S_n$.
The mutation matrix  $M_\tau$ for each $\tau \in \Xi_n$ 
is characterized in \cite[Theorem 11.1]{GY:Poi-CGL} as the unique solution to a certain system of linear equations, which we refer to as {\it GSV Equations}
(after M. Gekhtman, M. Shapiro, and A. Vainstein),
that are defined by
the Poisson coefficient matrix of $\bfy_\tau$ and the $\TT$-characters of the elements in $\bfy_\tau$. 
The existence and uniqueness of $M_\tau$ are proved in \cite[Theorem 11.1]{GY:Poi-CGL} by a rather involved induction procedure. 

In this paper, for any length $n$ symmetric Poisson CGL extension $R$ and for each $\tau \in \Xi_n$, we 
solve the GSV Equations on $M_\tau$ explicitly, and we do so by first proving 
some general results on Poisson CGL extensions that are not necessarily symmetric.

More precisely, for  an arbitrary $\TT$-Poisson CGL extension $(R, \{\,, \, \})$ of length $n$, {\it not necessarily symmetric}, 
we use the sequence ${\bf y} = (y_1, \ldots, y_n)$ of {\it homogeneous Poisson prime elements} of $R$ 
defined in \cite{GY:Poi-CGL} to  
formulate the GSV Equations on a matrix $M$ (see \eqref{eq:GSV-intro}).  
Then

1) by some elementary linear algebra arguments, we show that 
a solution $M$ to the GSV Equations, if exists, must be given by a certain explicit matrix product, implying in particular the 
uniqueness of such an $M$;

2) by a computation on the Poisson bracket $\{\, ,\, \}$, we describe 
an explicit integer solution $M$ to the GSV Equations via the expansion of certain elements of $R$ as Laurent polynomials of $\bfy$;

3)  under some mild conditions on scalars (to ensure that $M$ is skew-symmetrizable) and normality of 
$R$ (\autoref{de:normal}), we show that  the upper cluster algebra 
defined by $(\bfy, M)$ (with no frozen variables inverted) coincides with the polynomial ring $R$.

When the $\TT$-Poisson CGL $R$ is symmetric, it is shown in \cite{GY:Poi-CGL} that 
the same Poisson algebra $(R, \{\, , \, \})$ becomes a (in general no longer symmetric) 
$\TT$-Poisson CGL extension $R_\tau$ in the coordinates $(x_{\tau(1)}, \ldots, x_{\tau(n)})$
for each $\tau\in\Xi_n$.
Applying our aforementioned results on (the not necessarily symmetric) $\TT$-Poisson CGL extension
$R_\tau$, we arrive at  
not only the existence and uniqueness but also explicit formulas of $M_\tau$ as matrix products.

A classification of symmetric $\TT$-Poisson CGL extensions is recently given in  
\cite{Lu-Mykola:deformation} in terms of their
log-canonical terms and the second $\TT$-invariant Poisson 
cohomology of the log-canonical terms. In particular, it is shown in \cite{Lu-Mykola:deformation} that  there is a collection of 
non-positive integers, called {\it Cartan integers}, associated to the log-canonical term of 
any symmetric $\TT$-Poisson CGL extension $R$. As an application of our explicit formulas, we show that 
the non-zero entries of $M_\tau$, for every $\tau \in \Xi_n$, are either $\pm 1$
or $\pm a$, where $a$ is a negative Cartan integer associated to the log-canonical term of $R$. 

The results in this paper are in particular applicable to 
examples of symmetric Poisson CGL extensions from Lie theory, namely those constructed
\cite{Lu-Mykola:deformation} 
from any symmetrizable generalized Cartan matrix $A$ and any finite sequence ${\bf i}$ of
indices in $\{1, \ldots, r\}$, where $r$ is the size of $A$. In such cases, we give the precise relation 
between the mutation matrices $M_\tau$ and 
the Berenstein-Fomin-Zelevinsky mutation matrices associated to 
signed words in $\{1, \ldots, r\}$, which, by the work of A. Contu, F. Qin and Q. Wei
\cite{CQW:i-boxes}, are the same as  the Kashiwara-Kim exchange matrices from ${\bf i}$-boxes 
\cite{KKOP:24, KK:i-boxes}.

We remark that in this paper we are only concerned with the mutation matrices 
$M_\tau$ for $\tau \in \Xi_n$ in the Goodearl-Yakimov theory, and we make use of many results from
\cite{GY:Poi-CGL}. In particular, we refer to \cite{GY:Poi-CGL} for the proofs 
of the mutation equivalence of the seeds $(\bfy_\tau, M_\tau)$ and the equalities in \eqref{eq:RAU}. 
We now give more details of our results.

\subsection{Statements of main results}\label{ss:main-intro}
Throughout the paper, we fix a field $\bk$ of characteristic $0$, and let  $\bk^\times= {\bf k}\backslash\{0\}$.
Let $\TT$ be a split $\bfk$-torus with Lie algebra $\t$ and 
character lattice  $X(\T)$. We regard $X(\T)$ as a sub-lattice in 
$\t^*$ by identifying $\chi \in X(\T)$ with its differential at the identity element of $\TT$. 
For any integers $a \leq b$, let $[a, b]$ be the set of all integers $j$ such that $a \leq j \leq b$.

We first recall the definition of $\TT$-Poisson CGL extensions from 
\cite{GY:Poi-CGL}.

\begin{definition}\label{de:CGL-intro}\cite[$\S$5.1]{GY:Poi-CGL}
{\rm
A $\TT$-Poisson CGL extension of length $n$ is the polynomial algebra
${\bf k}[x_1, \ldots, x_n]$ together with a $\TT$-action by $\bfk$-algebra automorphisms
and a $\TT$-invariant Poisson bracket $\{\, ,\, \}$, 
such that

1) each $x_j$, for $j \in [1, n]$, is a $\TT$-weight vector with $\TT$-weight $\chi_j \in X(\TT)$;

2) there exist $h_1, \ldots, h_n \in \t$ with 
$\chi_j(h_j) \in \bfk^\times$ for each $j \in [1, n]$ such that 
\begin{equation}\label{eq:xj-xk-intro}
\{x_i, x_j\} = -\chi_i(h_j) x_i x_j -\delta_j(x_i),\hs 1 \leq i < j \leq n,
\end{equation}
where $\delta_j$ is a locally nilpotent derivation of the algebra $\bfk[x_1, \ldots, x_{j-1}]$.
We sometimes write
\begin{equation}\label{eq:R-x-chi-h}
R = (\bk[x_1, \ldots, x_n], \, \{\, , \,\}) \hs \mbox{or} \hs
R = (\bk[x_1, \ldots, x_n], \, \{\, , \,\})_{(\chi_1, \ldots, \chi_n; h_1, \ldots, h_n)}
\end{equation}
to indicate that $R$ is a $\TT$-Poisson CGL extension in the ordered set  $(x_1, \ldots, x_n)$ 
of polynomial generators, also called {\it CGL generators},  of $R$, and to indicate the
$\TT$-action  on $R$ via the $\TT$-weights of the generators $x_1, \ldots, x_n$ 
and the choice of $(h_1, \ldots, h_n)$.
\hfill $\diamond$
}
\end{definition}

Let $R$ be a  $\TT$-Poisson CGL extension of length $n$ as in  
\autoref{de:CGL-intro}. Set 
\begin{equation}\label{eq:lambda-j}
\lambda_j = \chi_j(h_j) \in \bk^\times, \hs j \in [1, n].
\end{equation}
Goodearl and Yakimov showed in \cite[Theorem 5.5 and Corollary 5.11]{GY:Poi-CGL} 
(see \autoref{thm:5.5} for details) 
that there exists a
{\it successor map} 
\[
s: \;\; [1, n] \longrightarrow [2, n] \sqcup \{+\infty\}
\]
and a unique sequence
$\bfy = (y_1, \ldots,y_n)$, recursively determined by
$y_{j} = x_j$ if $j \notin {\rm Im} (s)$, and 
\begin{equation}\label{eq:y-k-intro}
y_{s(j)} = x_{s(j)} y_{j} - \frac{\delta_{s(j)}(y_{j})}{\lambda_{s(j)}} 
\in \bk[x_1, \ldots, x_{s(j)}], \hs j \in [1, n], \; s(j) \neq +\infty.
\end{equation}
The sequence $\bfy$ is called the {\it sequence of homogeneous 
Poisson prime elements} of the Poisson CGL extension $R$. It follows from \eqref{eq:y-k-intro} that 
$\bfy$ is a set of free transcendental generators of ${\rm Frac}(R)
=\bk(x_1, \ldots, x_n)$ and that
$R \subset \bk [y_1^{\pm 1}, \ldots, y_n^{\pm 1}]$. 
Moreover, $\bfy$ is log-canonical with respect to the Poisson bracket 
$\{\, , \, \}$, and each $y_j$ is a $\TT$-weight vector 
with $\TT$-weight $\chi_{y_j}$. Set
\[
{\bf q} = (q_{i, j})_{i, j \in [1, n]} \hs \mbox{and} \hs
\chi_\bfy = (\chi_{y_1}, \,\ldots, \,\chi_{y_n}),
\]
where $\{y_i, y_j\} =q_{i, j} y_iy_j$ for all $i, j \in [1, n]$. Let 
\[
{\rm ex} = \{j \in [1, n]: s(j) \neq +\infty\} \subset [1, n],
\]
and denote by
${\rm Mat}_{n \times {\rm ex}}(\bk)$, resp. 
${\rm Mat}_{n \times {\rm ex}}(\ZZ)$, the set of all 
matrices with entries in $\bk$, resp. in $\ZZ$, whose rows are labeled by $i \in [1, n]$ and 
columns by $j \in {\rm ex}$.
Introduce the  {\it diagonal matrix} (see convention in $\S$\ref{ss:nota-intro})
\begin{equation}\label{eq:Lambda-intro}
\Lambda =(\lambda_{s(j)}e_j)_{j \in {\rm ex}} 
\in {\rm Mat}_{n \times {\rm ex}}(\bk),
\end{equation}
where $\{e_1, \ldots, e_n\}$ is the standard $\ZZ$-basis of
$\ZZ^n$ of {\it column} vectors. 

\begin{definition}\label{de:GSV-intro}
For a $\TT$-Poisson CGL extension $R$ of length $n$, the  {\it GSV Equations} for $R$ are the 
equations for 
$M \in {\rm Mat}_{n \times {\rm ex}}(\bk)$ given by
(see \eqref{eq:GSV} for a slightly more general definition) 
\begin{equation}\label{eq:GSV-intro}
{\bf q} M = -\Lambda \hs \mbox{and} \hs \chi_\bfy M = 0.
\end{equation}
\end{definition}

A solution $M \in {\rm Mat}_{n \times {\rm ex}}(\bk)$,  to the GSV Equations, if exists,
is not necessarily integral nor skew-symmetrizable. 
Introduce the $n \times n$ lower triangular matrix
\[
E =(e_1-e_{s(1)}, \, e_2-e_{s(2)}, \, \ldots, \, e_n-e_{s(n)}),
\]
where $e_{+\infty} := 0$. Setting
$\nu_{i, j} = \chi_i(h_j) + \chi_j(h_i)$ for $1\leq i < j \leq n$, introduce also 
\[
\bnu = \left(\begin{array}{ccccc} \lambda_1 & \nu_{1,2} & \nu_{1,3} & \cdots &\nu_{1, n}\\
0 & \lambda_2 & \nu_{2,3} & \cdots & \nu_{2, n}\\
0 & 0 & \lambda_3 & \cdots & \nu_{3, n}\\
\cdots & \cdots & \cdots & \cdots &\cdots\\
0 & 0 & 0 & \cdots & \lambda_n\end{array}\right).
\]
Using elementary linear algebra arguments, we show in \autoref{cor:M-unique} that a
solution $M \in {\rm Mat}_{n \times {\rm ex}}(\bk)$ to the GSV Equations, if exists,
is unique and is necessarily given by
\begin{equation}\label{eq:M-intro-000}
M = E^t \bnu^{-1} E \Lambda.
\end{equation}
Our main results on arbitrary (not necessarily symmetric) $\TT$-Poisson CGL extensions are
now summarized as follows.

\begin{theoremalpha}\label{thm:A}
Let $R$ be any  $\TT$-Poisson CGL extension of length $n$  as in  
\autoref{de:CGL-intro}.

1) (\autoref{pr:bsj-1} and \autoref{thm:M-1}) The matrix 
$M$ in \eqref{eq:M-intro-000} 
has integer entries and is a solution, and the only solution, to the GSV Equations \eqref{eq:GSV-intro} in 
${\rm Mat}_{n \times {\rm ex}}(\bk)$. Moreover,
for each $j \in {\rm ex}$ the element 
\[
c_{s(j)}:=\frac{\delta_{s(j)}(y_j)}{\lambda_{s(j)}} \in \bk[x_1, \ldots, x_{s(j)-1}]\subset \bk[y_1^{\pm 1}, \ldots, y_{s(j)-1}^{\pm 1}]
\]
has, in its expansion as a Laurent polynomial in $\bfy$, a unique Laurent monomial term $b_{s(j)}$ in $\bfy$
that does not contain any power of $y_j$, and one has
\[
\frac{b_{s(j)}}{y_{s(j)}} = \iota_{s(j)} \bfy^{M_j},
\]
for some $\iota_{s(j)} \in \bk^\times$, 
where $M_j$ is the $j^{\rm th}$-column of  $M =E^t \bnu^{-1} E \Lambda
\in {\rm Mat}_{n \times {\rm ex}}(\ZZ)$;

2) (\autoref{thm:M-1}) If $\lambda_{s(j)}/\lambda_{s(k)} \in \QQ$ for all $j, k \in {\rm ex}$, 
then there exists a diagonal
$\varepsilon \in {\rm Mat}_{{\rm ex} \times {\rm ex}}(\ZZ)$
with diagonal entries $\pm 1$  such that $M\ep$ is skew-symmetrizable, and  
$(\bfy, M\varepsilon)$ is a $\TT$-Poisson seed in ${\rm Frac}(R)$. 

3) (\autoref{thm:M-2}) Assume that $M\ep$ is skew-symmetrizable as in 2).
If, in addition, $R$ is normal in the sense that $\iota_{s(j)} = 1$ 
for every $j \in {\rm ex}$, 
then 
\[
R = \overline{\calU}(\bfy, M\varepsilon),
\]
where 
$\overline{\calU}(\bfy, M\varepsilon)$ is the upper cluster algebra defined by $(\bfy, M\ep)$ with no frozen variables
inverted.
\end{theoremalpha}

We now turn to symmetric Poisson CGL extensions.

\begin{definition}\label{de:sym-CGL-intro}\cite[$\S$6.1]{GY:Poi-CGL}
{\rm
A $\TT$-Poisson CGL extension $R = ({\bf k}[x_1, \ldots, x_n], \{\, , \, \})$ as in 
\autoref{de:CGL-intro} is said to be {\it symmetric} if, in addition, there exists
$(h_1^\ast, \ldots, h_n^\ast)\in \t^n$ satisfying
\[
\chi_j(h_j^\ast) \in \bk^\times, \;\;\; \forall\; j \in [1, n], 
\hs
\mbox{and}\hs \chi_i(h_j) = -\chi_j(h_i^\ast), \;\;\; \forall\; 1 \leq i< j \leq n,
\]
and if $\delta_j(x_i) \in \bk[x_{i+1}, \ldots, x_{j-1}]$ for all $1 \leq i< j \leq n$.
}
\end{definition}

For a $\TT$-Poisson CGL extension $R$ of length $n$ that is symmetric, 
we extend the diagonal matrix $\Lambda \in 
{\rm Mat}_{n \times {\rm ex}}(\bk)$  in \eqref{eq:Lambda-intro} to the $n \times n$ diagonal 
matrix $\overline{\Lambda}$ by setting the entry of $\overline{\Lambda}$ at $(j, j)$ 
to be $\lambda_j$ for $j \notin {\rm ex}$, and we introduce
\begin{equation}\label{eq:Q-intro}
Q = \bnu^{-1} \overline{\Lambda} \in {\rm Mat}_{n \times n}(\bk).
\end{equation}
The matrix $M = E^t \bnu^{-1} E\Lambda \in {\rm Mat}_{n \times {\rm ex}}(\ZZ)$ then
takes the more symmetric form (see $\S$\ref{ss:nota-sym}) 
\begin{equation}\label{eq:M-EQE-intro}
M = E^t Q E_{n \times{\rm ex}},
\end{equation}
where $E_{n \times {\rm ex}}$ is the sub-matrix of $E$ formed by the columns of $E$ indexed by
$j \in {\rm ex}$.
The formula for $M$ in \eqref{eq:M-EQE-intro}  has recently been proved 
in \cite{Lu-Mykola:deformation}, in which\footnote{All the results in \cite{Lu-Mykola:deformation}, proved 
over the complex numbers, hold over any field $\bk$ of characteristic $0$.}
symmetric $\TT$-Poisson CGL extensions 
are classified by their log-canonical terms and the second $\TT$-invariant Poisson cohomology
of their log-canonical terms. In particular, 
for each $j \in {\rm ex}$, the $j^{\rm th}$ column $QEe_j$
of the matrix $QE_{n \times {\rm ex}}$ is shown to be 
a $(\bk^\times)^n$-weight in the second $\TT$-invariant
Poisson cohomology space of the log-canonical term of  $\{\, , \, \}$, explicitly given as 
\begin{equation}\label{eq:QEj-intro}
QEe_j = -e_j-e_{s(j)} + {\rm deg}_{\bf x} (\{x_j, x_{s(j)}\}_{\rm tail}),
\end{equation}
where $\{x_j, x_{s(j)}\}_{\rm tail} = -\delta_{s(j)}(x_j)$, and 
${\rm deg}_{\bf x}(\phi)$ for a non-zero 
$\phi \in R$ is the  ${\bf x} = (x_1, \ldots, x_n)$-exponent of the leading monomial term of $\phi$ 
with respect to the reverse lexicographic order on $\ZZ^n$ (see $\S$\ref{ss:two-gradings}). 
It is also shown in \cite{Lu-Mykola:deformation} that the entries of 
the vectors ${\rm deg}_{\bf x} (\{x_j, x_{s(j)}\}_{\rm tail})$ for
$j \in {\rm ex}$ are all expressed in terms of 
a well-defined collection of non-positive
{\it Cartan integers} associated to the log-canonical term of $\{\, , \, \}$
(see $\S$\ref{ss:bth} for more detail).
We review the identity \eqref{eq:QEj-intro} in \autoref{thm:sym-M-1} and give a proof using results 
from \cite{GY:Poi-CGL}.

Using the symmetric property of $R$, Goodearl and Yakimov introduced  
a subset $\Xi_n$ of the permutation group $S_n$ and constructed,
for each $\tau \in \Xi_n$, 
first a presentation $R_\tau$ of the (same) Poisson algebra $R$ as 
a (not necessarily symmetric) $\TT$-Poisson CGL extension
 (see $\S$\ref{ss:reordering} for detail) and then a $\TT$-Poisson pre-seed 
$ ({\bf y}_\tau, M_\tau)$ (see \autoref{de:pre-seed})
in ${\rm Frac}(R)$ associated to $R_\tau$. 
Under certain conditions on the entries of $\Lambda$ and certain normality condition on $R$, 
see $\S$\ref{ss:y-tau-M-tau} and in particular \autoref{thm:GY-Mtau}, it is shown in 
\cite[Theorem 11.1]{GY:Poi-CGL} that 
\[
\{({\bf y}_\tau, M_\tau): \tau \in \Xi_n\}
\]
is a family of mutation equivalent 
 $\TT$-Poisson seeds in ${\rm Frac}(R)$.  
In particular, $(\bfy, M) = (\bfy_{\rm id}, M_{\rm id})$, where ${\rm id}\in S_n$ is the identity element.

We now state our  main results on 
the mutation matrices $M$ and $M_\tau$ for $\tau \in \Xi_n$, 
obtained by 
applying \autoref{thm:A} to $\TT$-Poisson CGL extension $R_\tau$.

\begin{theoremalpha}\label{thm:B}
Let $R=(\bk[x_1, \ldots, x_n],\, \{\, , \, \})$ be any symmetric $\TT$-Poisson CGL extension.

1) (\autoref{thm:sym-M-1}) For each $j \in {\rm ex}$, the $j^{\rm th}$ column $M_j$ of $M$ is given by 
\[
M_j = -e_{s(j)} + \deg_{\bf y}(c_{s(j)}) = \deg_{\bf y} \left(\frac{c_{s(j)}}{y_{s(j)}}\right),
\]
where $\deg_{\bf y}(f)$ for a Laurent polynomial $f$ in $\bfy$ is defined in $\S$\ref{ss:two-gradings};

2) (\autoref{thm:M-3}) For each $\tau \in \Xi_n$, the matrix $M_\tau \in {\rm Mat}_{n \times {\rm ex}}(\ZZ)$ 
is given by
\[
M_\tau = E^t \tau_\bullet Q \tau_\bullet^t E
=(E^{-1} \tau_\bullet^{-1}E)^t M 
(E^{-1} \tau_\bullet^{-1}E)_{{\rm ex} \times {\rm ex}},
\]
where $\tau_\bullet$ is a permutation matrix associated to $\tau$ and $R$ (see $\S$\ref{ss:y-tau-M-tau}), 
and $(E^{-1} \tau_\bullet^{-1}E)_{{\rm ex} \times {\rm ex}}$ denotes the 
${\rm ex} \times {\rm ex}$ sub-matrix of
$E^{-1} \tau_\bullet^{-1}E$;

3) (\autoref{thm:M-tau-cases}) For each $\tau \in \Xi_n$, 
the non-zero entries of $M_\tau$ are either $\pm 1$ or $\pm a$ for a negative Cartan integer $a$
associated to the log-canonical part of $\{\, , \, \}$. See \autoref{thm:M-tau-cases} for detailed description of
the entries of $M_\tau$.
\end{theoremalpha}

An important class of symmetric Poisson CGL extensions comes from Lie theory: 
given any symmetrizable generalized Cartan matrix $A = (a_{i, i'})_{i, i' \in [1, r]}$ 
with a symmetrizer $(d_i)_{i \in [1, r]}$ and
any sequence ${\bf i} =(i_1, \ldots, i_n)$ in $[1, r]$, it is shown in 
\cite[$\S$6.2]{Lu-Mykola:deformation} that one has a unique 
algebraic Poisson structure $\pi^{(A, {\bf i})}$ on $\CC^n$, giving rise to a
normalized symmetric $\TTA$-Poisson CGL extension
\[
R^{(A, \bfi)} = (\CC[x_1, \ldots, x_n], \; \{\, ,\, \}_{(A, {\bf i})}),
\]
where $\TTA$ is the split complex torus with character lattice the root lattice associated to $A$. 
When $A$ if of finite type, it is shown in \cite[$\S$6.3]{Lu-Mykola:deformation} that 
$\pi^{(A, {\bf i})}$ coincides with the so-called standard Poisson structure on the Bott-Samelson cells 
associated to $(A, \bfi)$ (see $\S$\ref{ss:R-A-i}).

For arbitrary symmetrizable generalized Cartan matrix $A$ and any sequence 
$\bfi = (i_1, \ldots, i_n)$ in $[1, r]$, the 
matrix $Q$ in \eqref{eq:Q-intro} for $R^{(A, \bfi)}$ is shown in \cite[$\S$6.2]{Lu-Mykola:deformation}
to be given by
\[
Q = \left(\begin{array}{ccccc} 1 & a_{i_1, i_2} & \cdots & a_{i_1, i_{n-1}}  & a_{i_1,i_n}\\
0 & 1 &\cdots & a_{i_2, i_{n-1}}  & a_{i_2, i_n}\\
\cdots & \cdots & \cdots & \cdots & \cdots \\
0 & 0 & \cdots & 1 & a_{i_{n-1}, i_n}\\
0 & 0 & \cdots & 0 & 1\end{array}\right).
\]
Applying \autoref{thm:B} to the symmetric Poisson CGL extension $R^{(A, \bfi)}$ and any $\tau \in \Xi_n$, we thus obtain a family of mutation matrices 
depending on $(A, \bfi, \tau)$.

On the other hand, define a length $n$ {\it signed word} \cite{CQW:i-boxes}, also called a {\it double word}
\cite{BZ:quantum}, in $[1, r]$, 
to be any sequence 
\[
\bida = (\ida_1, \, \ldots, \, \ida_n),
\]
where $\ida_j \in \{\pm1, \ldots, \pm r\}$. Associated to such a signed word  
$\bida$ one has \cite{BFZ:III, BZ:quantum, Qin:dual, CQW:i-boxes} the
Berenstein-Fomin-Zelevinsky mutation matrix $\widetilde{B}(\bida)$ (see $\S$\ref{ss:GY-BFZ}), which appears as 
a mutation matrix in the BFZ cluster structure on reduced double Bruhat cells 
and Bott-Samelson cells \cite{BFZ:III,Shen-Weng:dBS}, and, by \cite{CQW:i-boxes}, also as 
a Kashiwara-Kim mutation matrix in the monoidal categorification of cluster algebras 
associated to representation theory of quantum affine algebras \cite{KK:i-boxes, KKOP:24}. 

For any integers $r \geq1$ and $n \geq 2$, we establish in $\S$\ref{ss:Srn-Trn} a bijection 
\begin{equation}\label{eq:Srn-Trn}
\Srn \;\; \longleftrightarrow\;\; \Trn, \hs \bida \;\;
\longleftrightarrow \;\; (\bfi, \, \tau, \, \ep_1),
\end{equation}
where $\Srn$ is the set of all length $n$ signed words $\bida$ in $[1, r]$, 
and $\Trn$ is the set of all
triples $(\bfi, \tau, \ep_1)$ with $\bfi$  a length $n$ word in $[1, r]$, $\tau \in \Xi_n$, and 
$\ep_1 =\pm 1$. For any symmetrizable generalized Cartan matrix $A$, we prove in \autoref{thm:same-BFZ} that under the bijection in \eqref{eq:Srn-Trn} one has
\begin{equation}\label{eq:BFZ-M}
\widetilde{B}(\bida) = \widetilde{M}(\bida), 
\end{equation}
where $\widetilde{M}(\bida)$ is a {\it permutation} of the matrix $M_\tau$ in \autoref{thm:B} 
(corresponding to a re-ordering of the 
variables in the extended cluster $\bfy_\tau$) for the symmetric Poisson CGL extension
$R^{(A, {\bf i})}_\tau$ (see $\S$\ref{ss:seeds-prime} and \autoref{rk:tilde-M-bida} for detail). As 
a consequence of \autoref{thm:hat-M-bida}, we have an explicit matrix product
\begin{equation}\label{eq:B-intro}
\widetilde{B}(\bida) = E(\bida)^t Q(\bida) E(\bida)_{n \times {\rm ex}(\bida)},
\end{equation}
where the matrices $E(\bida)$ and $Q(\bida)$ are defined in \autoref{nota:hat-B-bida}.

We remark that 
while we obtained the formula in \eqref{eq:B-intro}  
 as a special case of \autoref{thm:B},  it
can be proved directly without using any Poisson CGL theory, and we are not aware of such a 
formula in the literature. On the other hand, the identity in 
\eqref{eq:BFZ-M} says that 
the family of  mutation matrices  in the seeds
$(\bfy_\tau, M_\tau)$ from the Goodearl-Yakimov theory applied the symmetric Poisson CGL 
extensions $R^{(A, \bfi)}$ is, up to re-ordering of indices, the same as that of all the BFZ mutation matrices 
associated to signed words. 

In addition, for any generalized symmetrizable Cartan matrix $A$ and any signed word $\bida$, we show 
in \autoref{thm:same-ensemble} that the 
full square matrix $E(\bida)^t Q(\bida) E(\bida)$ coincides with the  nondegenerate cluster ensemble matrix 
$\widehat{B}(\bida)$, defined by H. Willaims  \cite{Wil:ensemble} for the case of double reduced words and generalized to 
any signed words (see  $\S$\ref{ss:ensemble}), resulting in  
a matrix product formula not only for $\widehat{B}(\bida)$ but also for its the skew-symmetrizable part 
and symmetrizable part (see \autoref{rk:decomp}).


\subsection{Notation and conventions}\label{ss:nota-intro} 
Throughout the paper, we fix a field $\bk$ of characteristic $0$, and all algebras are understood to be over $\bk$.
For $n \in \ZZ_{> 0}$, elements in $\ZZ^n$ and in $\bk^n$
are understood as column vectors unless otherwise indicated, and 
the standard basis of ${\bf k}^n$ is denoted as $\{e_1, \ldots, e_n\}$, where 
$e_k$ for $k \in [1, n]$ has $1$ at the $k^{\rm th}$ entry and $0$ everywhere else. 
For integers $a < b$, let $[a, b]$ be the set of all integers $j$ such that $a \leq j \leq b$.
The transpose of a matrix $M$ is denoted as $M^t$.

If $A$ is a commutative $\bk$-algebra, the space of all $\bk$-derivations of $A$ is denoted as ${\rm Der}_\bk(A)$.
If $A$ is also an integral domain with fraction field ${\rm Frac}(A)$, and if
${\bf a} = (a_1, \ldots,a_n)$ a sequence of non-zero elements in $A$, for $m = (m_1, \ldots, m_n)^t \in \ZZ^n$
we write 
\begin{equation}\label{eq:a-m}
{\bf a}^m = a_1^{m_1} a_2^{m_2} \cdots a_n^{m_n} \in {\rm Frac}(A).
\end{equation}

If $(A, \{\, , \, \})$ is a Poisson algebra, two elements $a, b \in A$ are said to have {\it log-canonical Poisson
bracket} if $\{a, b\} \in \bk ab$. A sequence 
$a = (a_1, \ldots, a_n)$ of elements in $A$ is said to be {\it log-canonical} if its elements have pairwise log-canonical Poisson brackets, and in this case, if
\[
\{a_i,  a_j\} = q_{i,j} a_ia_j, \hs i, j \in [1, n], 
\]
 we call the skew-symmetric matrix 
${\bf q} = (q_{i,j})_{i,j \in [1, n]}$ 
the {\it Poisson coefficient matrix} of $a$.

Let $n \in \ZZ_{>0}$ and $I, J \subset [1, n]$, and let 
$A = \ZZ$ or $\bk$.
We denote by
${{\rm Mat}}_{I \times J}(A)$ the set of all matrices with entries in $A$ whose rows are indexed by $i \in I$ and columns by
$j \in J$. 
For $M = (m_{ij}) \in {{\rm Mat}}_{I \times J}(A)$ and $I' \subset I$ and $J' \subset J$, we have the sub-matrix
\[
M_{I' \times J'} = (m_{ij})_{i \in I', j \in J'} \in  {{\rm Mat}}_{I' \times J'}(A),
\]
and for $j \in J$, we denote the $j^{\rm th}$ column of $M$ by 
$M_j= Me_j$.  
 We also set 
\[
{{\rm Mat}}_{I \times n}(A) = {{\rm Mat}}_{I \times [1, n]}(A), \hs
{{\rm Mat}}_{n \times J}(A) = {{\rm Mat}}_{[1, n] \times J}(A),
\]
and ${{\rm Mat}}_{n \times n}(A) = {{\rm Mat}}_{[1,n]\times [1,n]}(A)$.
A matrix
\[
D =(d_{ij})\in {{\rm Mat}}_{n \times J}(A) \hs \mbox{or} \hs D =(d_{ij})\in {{\rm Mat}}_{I \times n}(A)
\]
is said 
to be {\it diagonal} if $d_{ij} = 0$ for all $i \neq j$. A matrix $M \in {{\rm Mat}}_{ n\times J}(\ZZ)$ is said to be 
{\it skew-symmetrizable} if there exists a diagonal $D = (d_{ij}) \in {{\rm Mat}}_{J \times n}(\ZZ)$ with 
$d_{ii} > 0$ for all $i \in J$, called a {\it skew-symmetrizer of $M$}, such that 
$(DM)^t = -DM \in {\rm Mat}_{J \times J}(\ZZ)$.

For $\sigma$ in the permutation group $S_n$, we will use $\sigma$ to also
denote the $n \times n$  matrix 
\begin{equation}\label{eq:sigma-perm}
\sigma = (e_{\sigma(1)}, \; e_{\sigma(2)}, \; \ldots, \; e_{\sigma(n}),
\end{equation}
so that $(\xi_1, \ldots, \xi_n)\sigma  = (\xi_{\sigma(1)}, \ldots, \xi_{\sigma(n)})$
for $(\xi, \ldots, \xi_n)^t \in A^n$. Note that $\sigma^{-1} = \sigma^t$ as matrices.
If $C = (c_{i,j}) \in {\rm Mat}_{n\times n}(A)$
and $\sigma \in S_n$, then for $i,j \in [1, n]$, the $(i,j)$-entry of $\sigma^{-1} C \sigma$ is 
\begin{equation}\label{eq:tau-lam}
(\sigma^{-1} C \sigma)_{i, j} =c_{\sigma(i), \sigma(j)}.
\end{equation}


\subsection{Acknowledgments} Yipeng Mi's contribution to the paper is entirely based on the work he carried out  while at the University of Hong Kong, during which period his research was partially supported by the Research Grants Council (RGC) of the Hong Kong SAR, China (GRF 17304415). 
The research of Zihang Liu and Jiang-Hua Lu has been supported by the RGC of the Hong Kong SAR, 
China (GRF 17304415 and GRF 17306724).

\section{Cluster algebra preliminaries}
\subsection{Cluster algebras and upper cluster algebras}\label{ss:cluster-alg}
We only consider 
cluster algebras of geometric type, and we follow the notation and convention used  in \cite{FZ:I, BFZ:III, GY:Poi-CGL}.

Let $n \in \ZZ_{>0}$, and let $\FF$ be a  field extension of ${\bf k}$ of pure  transcendental
degree $n$. Let ${\rm ex}$ be any subset of $[1, n]$.
A {\it seed in $\FF$ of type ${\rm ex}$} is a pair $(\bfu, M)$, where $\bfu = (u_1, \ldots, u_n)$ 
is a free transcendental basis of $\FF$ over ${\bf k}$, also called an {\it extended cluster} in $\FF$, 
and $M \in {\rm Mat}_{n \times {\rm ex}}(\ZZ)$ and is skew-symmetrizable (see $\S$\ref{ss:nota-intro}).
The elements $u_j \in \FF$ for $j \in {\rm ex}$, resp. for $j \in [1, n] \backslash {\rm ex}$, are called
{\it cluster variables}, resp. {\it frozen variables}, of the seed $(\bfu, M)$.  

Given a seed $(\bfu = (u_1, \ldots, u_n), M=(m_{i,k}))$ in $\FF$ of type ${\rm ex}$, the {\it mutation of $(\bfu, M)$ in direction $j \in {\rm ex}$} 
is the seed 
\[
\mu_j({\bf u}, \,M) = (\mu_j^M({\bf u}), \;M^\prime),
\]
where
$\mu_j^M(\bfu) = (u_1, \ldots, u_{j-1}, u_j^\prime, u_{j+1}, \ldots, u_n)$
and $M^\prime = (m_{ik}^\prime)$ are given by
\begin{align*}
u_j^\prime &= \frac{1}{u_j} \left(\prod_{m_{i,j} > 0} u_i^{m_{i,j}} + \prod_{m_{i,j} < 0} 
u_i^{-m_{i,j}}\right),\\
m_{i,k}^\prime &= \begin{cases} -m_{i,k}, & \hs i = j \;\; \mbox{or} \;\; k=j, \\
m_{i,k} + \frac{1}{2}(|m_{i,j}|m_{j,k} + m_{i,j} |m_{j,k}|), & \hs \mbox{otherwise}.\end{cases}
\end{align*}
A seed $(\bfu^\prime, M^\prime)$ in $\FF$ that can be obtained from $(\bfu, M)$ by a sequence of mutations 
is said to be {\it mutation equivalent} to
$(\bfu, M)$, and we denote by   $[(\bfu, M)]$  the mutation equivalent class of seeds of $(\bfu, M)$. 
 Note that all seeds in $[(\bfu, M)]$ have the same frozen variables which are
called the frozen variables of $[(\bfu, M)]$.
A cluster variable of any seed 
$(\bfu^\prime, M^\prime) \in [(\bfu, M)]$ is called a cluster variable of $[(\bfu, M)]$.

For any ${\rm inv} \subset [1, n]\backslash {\rm ex}$, and for
any $(\bfu^\prime, M^\prime) \in [(\bfu, M)]$ with $\bfu^\prime = (u_1^\prime, \ldots, u_n^\prime)$, let
\[
\calL(\bfu^\prime; {\rm inv}) = {\bf k}[u_1^\prime, \ldots, u_n^\prime][(u_j^\prime)^{-1}: j \in {\rm ex}\sqcup
{\rm inv}].
\]

\begin{definition}\label{de:cluster-A}
{\rm Let $(\bfu, M)$ be a seed  $\FF$ of type ${\rm ex} \subset [1, n]$, and let 
${\rm inv} \subset [1, n]\backslash {\rm ex}$.

1) The upper cluster algebra $\calU(\bfu, M; {\rm inv})$ is defined to be
\[
{\calU}(\bfu, M; {\rm inv}) = \bigcap_{(\bfu^\prime, M^\prime) 
\in [(\bfu, M)]} \calL(\bfu^\prime; {\rm inv}) \subset \FF.
\]

2) The cluster algebra ${\calA}(\bfu, M; {\rm inv})$ 
is defined to be the ${\bf k}$-sub-algebra of $\FF$ generated by 
all the cluster variables, all the frozen variables of $[(\bfu, M)]$, and all $u_j^{-1}$ for $j \in {\rm inv}$;

3) We write $\overline{\calA}(\bfu, M)  = {\calA}(\bfu, M; \emptyset)$
and $\overline{\calU}(\bfu, M) = {\calU}(\bfu, M; \emptyset)$.
\hfill $\diamond$
}
\end{definition}

\begin{remark}\label{rk:inv}
{\rm 
For any seed $(\bfu, M)$ in $\FF$ of type ${\rm ex}$ and any ${\rm inv} \subset [1, n]\backslash {\rm ex}$, 
the Laurent phenomenon \cite{FZ:I} says that 
${\calA}(\bfu, M; {\rm inv})  \subset {\calU}(\bfu, M; {\rm inv})$.
\hfill $\diamond$
}
\end{remark}

\subsection{$\TT$-Poisson seeds and $\TT$-Poisson pre-seeds}\label{ss:T-Poi-seeds}
Suppose now that the field $\FF$ is equipped with a Poisson structure and an action by
a split ${\bf k}$-torus $\TT$ via Poisson isomorphisms. Let again $X(\TT)$ be the character lattice of $\TT$, and 
denote the $\T$-weight of a $\T$-weight vector $u \in \FF$ by $\chi_u\in X(\T)$.
Define an {\it extended log-canonical $\TT$-cluster} in $\FF$ to be an extended cluster in $\FF$ which is 
log-canonical with respect to the Poisson structure and consists of $\T$-weight vectors. 
A seed $(\bfu, M)$ in $\FF$ is said to be {\it $\TT$-Poisson}, or \cite{GSV:book} {\it compatible with the Poisson 
structure and the $\T$-action}, if $\bfu'$ is an extended log-canonical $\TT$-cluster for every 
seed $(\bfu', M')$ in $[(\bfu, M)]$.
If $\bfu = (u_1, \ldots, u_n)$ is an extended  log-canonical $\TT$-cluster in $\FF$, we write 
${\bf q}_{\bf u} \in {\rm Mat}_{n \times n}({\bf k})$ for the
Poisson coefficient matrix of $\bfu$, and let
\[
\chi_{\bfu} = (\chi_{u_1}, \ldots, \chi_{u_n}) \in X(\T)^n.
\]
The following \autoref{le:compatible-pair} is well-known, but we present a version that is suitable for our purpose.
 See $\S$\ref{ss:nota-intro} for our convention on diagonal and skew-symmetrizable 
matrices which are not necessarily square matrices.

\begin{lemma}\label{le:compatible-pair}
Let $\bfu = (u_1, \ldots, u_n)$ be an extended log-canonical $\TT$-cluster in $\FF$. 
Let ${\rm ex} \subset [1, n]$ and suppose that  
$M \in {\rm Mat}_{n \times {\rm ex}}(\ZZ)$ satisfies
\begin{equation}\label{eq:qM-1}
{\bf q}_{\bf u} M= {\Lambda} \hs \mbox{and} \hs \chi_{\bfu}M = 0,
\end{equation}
where ${\Lambda} \in {\rm Mat}_{n \times {\rm ex}}({\bf k})$ 
is diagonal with diagonal entry 
${\Lambda}_{j,j}\neq 0$ for every $j \in {\rm ex}$.
If
\begin{equation}\label{eq:epsilon-Lambda}
\frac{{\Lambda}_{j,j}}{{\Lambda}_{k,k}} 
\in \QQ, \hs \forall \; j, k \in {\rm ex},
\end{equation}
then there exists a diagonal $\ep \in {\rm Mat}_{{\rm ex} \times {\rm ex}}(\ZZ)$ with $\pm 1$ on the diagonals
such that 
$M\ep \in {\rm Mat}_{n \times {\rm ex}}(\ZZ)$ is skew-symmetrizable, 
and $(\bfu, M\ep)$ is a $\TT$-Poisson seed in $\FF$. 
\end{lemma}

\begin{proof}
Fix any $k_0 \in {\rm ex}$ and choose an integer $z$ and $\varepsilon(j) \in \{\pm 1\}$ for each $j \in {\rm ex}$ such that 
\begin{equation}\label{eq:z-r0}
z \;\frac{ \varepsilon(j){\Lambda}_{j,j}}{ {\Lambda}_{k_0,k_0}}
\in \ZZ_{>0}, \hs 
\forall \; j \in {\rm ex}.
\end{equation}
Let  $r_ 0 = z/{\Lambda}_{k_0, k_0} \in \bk^\times$ and let
$\ep \in {\rm Mat}_{{\rm ex}\times \rm ex}(\ZZ)$ be the diagonal matrix with $(j, j)$-entry 
$\ep(j)$ for $j \in {\rm ex}$, Let
$D \in {\rm Mat}_{{\rm ex} \times n}(\ZZ)$ be such that 
\begin{equation}\label{eq:D-Lambda}
D^t = r_0 {\Lambda}\ep
\in {\rm Mat}_{n \times {\rm ex}}(\ZZ).
\end{equation}
Then $D$ is diagonal with positive integers on the diagonal, and it
follows from  
${\bf q}_{\bf u} M\ep= {\Lambda}\ep$ that $r_0 {\bf q}_{\bf u} M\ep = D^t$.
Taking transpose and using
${\bf q}_{\bf u}^t = -{\bf q}_{\bf u}$, we get 
$(M\ep)^t(r_0 {\bf q}_{\bf u}) = -D$, which, by $r_0 {\bf q}_{\bf u} M\ep = D^t$, gives 
\[
(DM\ep)^t =(M\ep)^t D^t = (M\ep)^tr_0 {\bf q}_{\bf u}M\ep= -DM\ep.
\]
Thus $M\ep  \in {\rm Mat}_{n \times {\rm ex}}(\ZZ)$ is skew-symmetrizable with $D$ as a skew-symmetrizer.

A direct calculation using the mutation rule of cluster variables shows that 
(see, for example, \cite[Proposition 3.3]{BZ:quantum}, \cite[(4.4)]{GSV:book}, \cite[Lemma 3.8]{GY:Poi-CGL})
the identity ${\bf q}_{\bf u} M\ep  = {\Lambda}\ep$ implies that 
every extended cluster in $[({\bf u}, M\ep)]$
is log-canonical, and the identity $\chi_{\bfu} M\ep  = 0$ implies that every cluster variable of 
$[({\bf u}, M\ep)]$ is
a $\T$-weight vector. 
\end{proof}

The following terminology is convenient for the discussions in this paper.

\begin{definition}\label{de:pre-seed}
{\rm 
Given ${\rm ex} \subset [1, n]$, by a {\it $\TT$-Poisson pre-seed in $\FF$ of type ${\rm ex}$} we mean a pair 
$({\bf u}, M)$, where ${\bf u}$ is an extended log-canonical $\TT$-cluster in $\FF$ and 
$M \in {\rm Mat}_{n \times {\rm ex}}(\ZZ)$ satisfying
\begin{equation}\label{eq:GSV}
{\bf q}_{\bf u} M = {\Lambda} \hs \mbox{and} \hs
\chi_{{\bf u}} M = 0
\end{equation}
for some diagonal 
${\Lambda} = ({\Lambda}_{i,j}) \in {\rm Mat}_{n \times {\rm ex}}({\bf k})$ with
${\Lambda}_{j,j}\neq 0$ for every $j \in {\rm ex}$. We refer to the equations in \eqref{eq:GSV} as
{\it GSV Equations}.
\hfill $\diamond$
}
\end{definition}

\begin{remark}\label{rk:pre-seed}
{\rm
1) A $\TT$-Poisson seed in $\FF$ is thus a $\TT$-Poisson pre-seed $({\bf u}, M)$ in $\FF$ such that $M$ is skew-symmetrizable;

2) Suppose that $({\bf u}, M)$ is $\TT$-Poisson pre-seed in $\FF$
 with $M$ satisfying \eqref{eq:GSV}. Then for any diagonal matrix
$\varepsilon \in {\rm Mat}_{{\rm ex} \times {\rm ex}}(\ZZ)$ with diagonal entries
$\varepsilon(j) \in \{\pm 1\}$ for $j \in {\rm ex}$, the pair $(\bfu, M\varepsilon)$ is also
a $\TT$-Poisson pre-seed in $\FF$. \autoref{le:compatible-pair} shows that the condition on $\Lambda$ in 
\eqref{eq:epsilon-Lambda} is 
sufficient for such an $\ep \in {\rm Mat}_{{\rm ex} \times {\rm ex}}(\ZZ)$ to exist. 
If 
\begin{equation}\label{eq:Q-pos}
\frac{{\Lambda}_{j,j}}{{\Lambda}_{k,k}} 
\in \QQ_{>0}, \hs \forall \; j, k \in {\rm ex},
\end{equation}
then one can choose $z$ to be a positive integer and $\ep(j) = 1$ for every $j \in {\rm ex}$ in \eqref{eq:z-r0}, 
so $M$ is skew-symmetrizable, and thus $(\bfu, M)$ is a $\TT$-Poisson seed in $\FF$. It also follows from
\eqref{eq:D-Lambda} that \eqref{eq:Q-pos} is equivalent to 
the existence of positive integers $\{d_j: j \in {\rm ex}\}$ such that 
\[
\frac{d_j}{d_k} = \frac{{\Lambda}_{j,j}}{{\Lambda}_{k,k}}, \hs j, k \in {\rm ex}.
\]
\hfill $\diamond$
}
\end{remark}

Recall that for $\sigma\in S_n$, the same
symbol $\sigma$ also denote the $n \times n$ matrix given in \eqref{eq:sigma-perm}.

\begin{definition}\label{de:seed-reordering}
{\rm
Given a $\TT$-Poisson seed (resp. pre-seed) 
$({\bf u} =(u_1, \ldots, u_n), M)$ in $\FF$ of type ${\rm ex} \subset [1, n]$ 
and an element
$\sigma \in S_n$, let $({\rm ex})^\sigma = \sigma^{-1}({\rm ex}) \subset [1, n]$, and let 
\[
{\bf u}^\sigma = (u_{\sigma(1)}, \, u_{\sigma(2)}, \, \ldots, \, u_{\sigma(n)}) \hs \mbox{and} \hs
M^\sigma = \sigma^{-1} M \sigma_{{\rm ex} \times ({\rm ex})^\sigma}.
\]
Then $({\bf u}^\sigma, M^\sigma)$ is a $\TT$-Poisson seed (resp. pre-seed) in $\FF$ of type $({\rm ex})^\sigma$,
which we will 
refer to as the {\it re-ordering} of $({\bf u}, M)$ by $\sigma \in S_n$.
}
\end{definition}

\section{Proof of \autoref{thm:A} on Poisson CGL extensions}\label{s:proof-A}
In this section, we prove \autoref{pr:bsj-1},  \autoref{thm:M-1} and \autoref{thm:M-2}, which, when combined together, give
\autoref{thm:A} stated in $\S$\ref{ss:main-intro}.

Throughout $\S$\ref{s:proof-A}, we fix a split $\bk$-torus $\TT$ with Lie algebra $\t$ and character 
lattice $X(\TT) \subset \t^*$. When $\TT$ acts on a vector space $V$, a $\TT$-weight vector
in $V$ is also called a $\TT$-homogeneous element in $V$. Recall again that by a 
$\TT$-Poisson algebra we mean a Poisson $\bk$-algebra $A$ with a $\TT$-action by Poisson 
algebra automorphisms. An element $y$ in a Poisson $\bk$-algebra $A$ is said to be {\it Poisson} if  $\{y, A\} \subset yA$. When $A$ is an integral domain, an element $y \in A$ that is both Poisson and prime is called a
{\it Poisson prime element}.

\subsection{The sequence of homogeneous Poisson prime elements}\label{ss:bfy}
Let $R = (\bk[x_1, \ldots,x_n], \{\, , \, \})_{(\chi_1, \ldots, \chi_n; h_1, \ldots, h_n)}$ be a $\T$-Poisson CGL extension 
as in \autoref{de:CGL-intro}. Set $R_0 = \bk$ and 
$R_k = \bk[x_1, \ldots, x_k]$ for $k \in [1, n]$, and note that $R_k$ is
a $\TT$-Poisson sub-algebra of $R$ by \eqref{eq:xj-xk-intro}. 
For $k \in [1, n]$, let $\partial_{h_k} \in {\rm Der}_\bk(R)$ be such that
$\partial_{h_k}(x_j) = \chi_j(h_k)$ for all $j \in [1, n]$, and let 
$\theta_k = \partial_{h_k}|_{R_{k-1}} \in {\rm Der}_\bk(R_{k-1})$.
Then  \eqref{eq:xj-xk-intro} becomes
\begin{equation}\label{eq:bracket-j}
\{x_k, \;a\} = \theta_k(a) x_k + \delta_k(a), \hs k \in [1, n], \; a \in R_{k-1}.
\end{equation}
 In the terminology and notation in \cite[$\S$4.2]{GY:Poi-CGL}, 
the $\TT$-Poisson algebra $R_k$ for each $k \in [1, n]$ is a {\it $\TT$-Poisson-Cauchon extension} of
$R_{k-1}$, and one writes
$R_k = R_{k-1}[x; \theta_k, \delta_k]$. The $\TT$-Poisson CGL extension $R$ is thus  an iterated 
$\TT$-Poisson-Cauchon extension, and one writes
\[
R = \mathbf{k}[x_1;\,\theta_1,\,\delta_1][x_2;\,\theta_2,\,\delta_2]\cdots
[x_n;\,\theta_n,\,\delta_n].
\]

We now recall a fundamental result from \cite{GY:Poi-CGL} on 
the nested
sequence
\[
{\bf k} = R_0 \subset R_1 \subset \cdots \subset R_n = R
\]
of $\T$-Poisson algebras. 
For $k \in [1, n]$,
let $\calP_k \subset R_k$ be the set of all homogeneous Poisson prime elements of $R_k$, and
let $\calP^\prime_k\subset \calP_k$ be the set of elements in $\calP_k$ that
are not in $R_{k-1}$. Note that both $\calP_k$ and $\calP_k^\prime$ are invariant under multiplication by scalars in
${\bf k}^\times$. For $k \in [1, n]$, let $\calP_k^\prime/{\bf k}^\times$ be the quotient set.
The cardinality of a finite set $X$ is denoted by $|X|$. Recall  we have set
\[
\lambda_k = \chi_k(h_k)\in \bk^\times, \hs k \in [1, n].
\]

\begin{theorem}\label{thm:5.5} \cite[Theorem 5.5, Corollary 5.11]{GY:Poi-CGL}.
Let $R = (\bk[x_1, \ldots,x_n], \{\, , \, \})$ 
be a $\T$-Poisson CGL extension as in \autoref{de:CGL-intro}. For each $k \in [1, n]$, one has 
$|\calP_k^\prime/{\bf k}^\times|=1$, and there is a unique $p(k) \in \{-\infty\}
\cup [1, k-1]$ and a unique $y_k \in \calP_k^\prime$, determined recursively 
as follows: $y_1 = x_1$, and for $k \geq 2$,

1) if $\delta_k = 0$, then $p(k) = -\infty$ and $y_k = x_k$;

2) if $\delta_k \neq 0$, then $p(k)$ is the unique integer in $[1, k-1]$ such that $y_{p(k)} \in \calP_{k-1}$ and
$\delta_k(y_{p(k)}) \neq 0$, 
and in such a case, $\delta_k^2 (y_{p(k)}) = 0$, and 
\begin{equation}\label{eq:yk-ypk}
y_k= y_{p(k)}x_k- \frac{\delta_k(y_{p(k)})}{\lambda_k}.
\end{equation}

\noindent
Moreover, for $k \in [1, n]$, define $s(k) =+ \infty$ if $k \neq p(k')$ for any $k' \in [k+1, n]$, and define $s(k) =
k'$ if $p(k') = k$ (such a $k'$ is necessarily unique). Then
\begin{equation}\label{eq:calP-k}
\calP_k = \bigcup_{j \in [1, k],\, s(j) > k} {\bf k}^\times y_j.
\end{equation}
\end{theorem}

The sequence
$\bfy = (y_1, \; \ldots, \; y_n)$
in \autoref{thm:5.5} is called \cite[$\S$5.2]{GY:Poi-CGL} the {\it sequence of homogeneous Poisson prime elements} associated to the 
$\T$-Poisson CGL extension $R$. 

\begin{notation}\label{nota-p-s} 
{\rm In the context of \autoref{thm:5.5},
setting $y_{-\infty} = 1 \in {\bf k} = R_0$, then
\begin{equation}\label{eq:ykk-00}
y_k = y_{p(k)}x_k- \frac{\delta_k(y_{p(k)})}{\lambda_k}
\end{equation}
holds for every $k \in [1, n]$.
Following \cite{GY:Poi-CGL}, the two maps
\[
p:\; [1, n] \longrightarrow \{-\infty\} \sqcup [1, n-1], \hs \mbox{and} \hs
s: \; [1, n] \longrightarrow [2, n] \sqcup \{+\infty\},
\]
are respectively called the  {\it predecessor map} and the  {\it successor map} associated to 
the Poisson CGL extension $R$. 
For $k \in [1, n]$, the two integers 
\begin{align}\label{eq:om-k}
o_-(k) &:= {\rm max}\{m \in \ZZ_{\geq 0}: p^m(k) \neq -\infty\},\\
\label{eq:op-k}
o_+(k) &:= {\rm max}\{m \in \ZZ_{\geq 0}: s^m(k) \neq +\infty\},
\end{align}
are  respectively
called the {\rm $p$-order} and the {\rm $s$-order} of $k$, and the set
\begin{equation}\label{eq:Lk}
L(k) = \{p^{o_-(k)}(k), \; \ldots, \; p(k), \; k, \; s(k), \;\ldots, \;  s^{o_+(k)}(k)\},
\end{equation}
is called  the {\it level set of $k$ associated to $R$}.
Every $j \in [1, n]$ then belongs to a unique level set.
Let $\calL$ be the set of all level sets.
Following \cite{GY:Poi-CGL}, the number of level sets, i.e., the integer
\[
|\calL| = |\{k \in [1, n]: p(k) = -\infty\}| = |\{j \in [1, n]: s(j) = +\infty\}|
\]
is called the
{\it rank} of the Poisson CGL extension $R$. Set again
\begin{equation}\label{eq:ex}
{\rm ex} = \{j \in [1, n]: s(j) \neq +\infty\}.
\end{equation}
For a  $\TT$-weight vector $\phi \in R$, we denote by $\chi_\phi \in X(\TT)$ its $\TT$-weight. In particular, we have $\chi_j = \chi_{x_j}$ for $j \in [1, n]$, and it follows from \eqref{eq:ykk-00} that 
\begin{equation}\label{eq:chi-kk}
\chi_{y_k} = \chi_{y_{p(k)}} + \chi_k = \chi_{p^{o_-(k)}} + \cdots + \chi_{p(k)} + \chi_k, \hs k \in [1, n].
\end{equation}
}
\end{notation}

\begin{lemma}\label{le:y-recursice-Poi}
The sequence ${\bf y} = (y_1, \ldots, y_n)$ is also recursively given by 
\[
\lambda_k y_k = (\chi_{y_{p(k)}}(h_k)+\lambda_k)  y_{p(k)}x_k + \{y_{p(k)}, \, x_k\}, \hs k \in [1, n].
\]
\end{lemma}

\begin{proof}
Let $k \in [1, n]$. Combining \eqref{eq:bracket-j} and \eqref{eq:ykk-00}, we get 
\begin{align*}
\lambda_k y_k &= \lambda_k y_{p(k)} x_k - \delta_k(y_{p(k)}) = \lambda_k y_{p(k)} x_k-\{x_k, \, y_{p(k)}\} + \chi_{y_{p(k)}}(h_k) y_{p(k)} x_k\\
& = (\chi_{y_{p(k)}}(h_k)+\lambda_k) y_{p(k)}x_k + \{y_{p(k)}, \, x_k\}.
\end{align*}
\end{proof}

By \cite[Proposition 5.8]{GY:Poi-CGL}, $\bfy = (y_1, \ldots, y_n)$ is log-canonical with respect to 
the Poisson bracket $\{\, ,  \, \}$.
To recall the Poisson coefficient matrix
for $\bfy$,  we first set up some notation. Recall that $\{e_1, \ldots, e_n\}$ denotes the standard 
$\ZZ$-basis of $\ZZ^n$ (of column vectors). Set $e_{+\infty} = e_{-\infty} = 0$.

\begin{notation}\label{nota:blam-E}
{\rm 
Recall from $\S$\ref{ss:main-intro}  the  lower triangular matrix $E$ and its transpose 
\begin{equation}\label{eq:E-Et} 
E= (e_1-e_{s(1)}, \; \ldots, \; e_n - e_{s(n)}), \hs
E^t = (e_1-e_{p(1)}, \; \ldots, \; e_n - e_{p(n)}).
\end{equation} 
Let $F = E^{-1}$. Setting, for $j \in [1, n]$, 
\begin{equation}\label{eq:ee}
\widetilde{e}_j = \sum_{k \in L(j) \cap [j, n]}e_k \hs \mbox{and} \hs 
\overline{e}_j  = \sum_{k \in L(j) \cap [1, j]}e_k,
\end{equation}
we also have 
\begin{equation}\label{eq:F}
F  = (\widetilde{e}_1, \, \widetilde{e}_2, \, \ldots, \, \widetilde{e}_n), \hs \hs 
F^t = (\overline{e}_1, \, \overline{e}_2, \, \ldots, \, \overline{e}_n).
\end{equation}
Set $\chi_{\bf y} = (\chi_{y_1}, \ldots, \chi_{y_n}) \in X(\TT)^n$ and $\chi_{\bf x} = (\chi_1, \ldots, \chi_n)
\in X(\TT)^n$.  Rewriting  
\eqref{eq:chi-kk}, one has 
\begin{equation}\label{eq:chi-y-chi-x}
\chi_{\bf y} = \chi_{\bf x} F^t.
\end{equation}
Introduce also the skew-symmetric matrix
\begin{equation}\label{eq:blam}
\blam= \left(\displaystyle \begin{array}{cccc} 0 & -\chi_1(h_2) &  \cdots & -\chi_1(h_n)\\
\chi_1(h_2) & 0  & \cdots & -\chi_2(h_n)\\
\cdots & \cdots & \cdots& \cdots \\
\chi_1(h_n) & \chi_2(h_n)  & \cdots & 0\end{array}\right).
\end{equation}
\hfill $\diamond$
}
\end{notation}

\begin{lemma}\label{le:bfq}\cite[Proposition 5.8]{GY:Poi-CGL}
The sequence $\bfy = (y_1, \ldots, y_n)$ is 
log-canonical with respect to $\{\, , \, \}$, and the  
Poisson coefficient matrix ${\bf q} = (q_{k, l})_{k, l \in [1, n]}$ of $\bfy$ is given by
\begin{equation}\label{eq:bfq}
{\bf q} = F \blam F^t.
\end{equation}
\end{lemma}

The following \autoref{le:q-jpj} will be used in $\S$\ref{ss:bfy-M}.

\begin{lemma}\label{le:q-jpj}
Set $q_{-\infty, l} =0$ for any $l \in [1, n]$. For any $k, l \in [1, n]$, one has
\[
\{x_k, \, y_l\} = \begin{cases} (q_{k, l} - q_{p(k), l}) x_ky_l, & \hs k \leq l,\\
(q_{k, l} - q_{p(k), l}) x_ky_l + \delta_k(y_l), & \hs k > l.\end{cases}
\]
\end{lemma}

\begin{proof} The case when $k \leq l$ is proved in \cite[Corollary 5.10]{GY:Poi-CGL}. 
Assume now that $k > l$. 
Then by \eqref{eq:bracket-j}, $\{x_k, y_l\} = \chi_{y_l}(h_k)x_ky_l + \delta_k(y_l)$. On the other hand, 
\[
q_{k, l} -q_{p(k),l}  = (E^te_k)^t {\bf q} e_l = e_k^t \blam F^t e_l
=e_k^t \blam  \overline{e}_l =
\sum_{j \in [1, l]\cap L(l)} \chi_j(h_k) =\chi_{y_l}(h_k).
\]
Thus $\{x_k, y_l\} = (q_{k, l} -q_{p(k),l})x_ky_l + \delta_k(y_l)$.
\end{proof}

Let $\calT=\bk[\bfy^{\pm 1}] = {\bf k}[y_1^{\pm 1}, \ldots, y_n^{\pm 1}]$. 
Note then that $R \subset \calT$ by \eqref{eq:ykk-00}. It follows from 
\autoref{le:bfq} that $\calT$  is a 
Poisson sub-algebra of ${\rm Frac}(R)$, where
${\rm Frac}(R)$ has the unique Poisson bracket extending that on $R$. More precisely, for 
(column vectors) $f, g\in \ZZ^n$, one has
\begin{equation}\label{eq:bra-bfy-f-g}
\{\bfy^f, \, \bfy^g\} = f^t {\bf q}\, g.
\end{equation}

\subsection{Two $\ZZ^n$-gradings on $R$}\label{ss:two-gradings}
Following \cite[$\S$5.2]{GY:Poi-CGL}, consider the 
reverse lexicographic order $\prec$ on $\ZZ^n$, i.e., 
\[
(f_1, \ldots, f_n)^t \prec (f_1^\prime, \ldots, f_n^\prime)^t
\]
if there exists $j \in [1, n]$ such that $f_j < f_j^\prime$ and $f_k = f^\prime_k$ for all $k \in [j+1, n]$.
One then has
\begin{equation}\label{eq:ffgg}
f \prec f' \;\;\; \mbox{and} \;\;\; g \prec g' \;\;\; \Longrightarrow f + f' \prec g + g', \hs \hs f, f', g, g' \in \ZZ^n.
\end{equation}
As in \cite[$\S$5.2]{GY:Poi-CGL}, writing a non-zero $b \in \calT$
as
\[
b = \xi_f \bfy^f + \sum_{g \in \ZZ^n, \, g \prec f} \xi_g \bfy^g,
\]
where $\xi_g \in {\bf k}$ for $g \in \ZZ^n$ and $\xi_f \in {\bf k}^*$,
we set
\[
{\rm lt}_\bfy(b) =  \xi_f \bfy^f \hs \mbox{and} \hs {\rm deg}_\bfy(b) = f,
\]
and call them, respectively, the {\it $\bfy$-leading term} and the {\it $\bfy$-degree} of $b$.
We also call $\xi_f$ the {\it $\bfy$-leading coefficient}
of $b$. By \eqref{eq:ffgg}, for all non-zero $b,b' \in \calT$ one has
\begin{equation}\label{eq:degree-bb}
{\rm lt}_\bfy(bb') = {\rm lt}_\bfy(b) \, {\rm lt}_\bfy(b') \hs \mbox{and} \hs \deg_\bfy(bb') = \deg_\bfy(b) + \deg_\bfy(b').
\end{equation}

Similarly (see again \cite[$\S$5.2]{GY:Poi-CGL}), using the reverse lexicographic order on $\ZZ^n$ and the sequence
${\bf x} = (x_1, \ldots, x_n)$ in place of the sequence $\bfy$, one has the {\it ${\bf x}$-leading term} 
${\rm lt}_{\bf x}(b) \in R$,
the {\it ${\bf x}$-leading coefficient},  and the {\it ${\bf x}$-degree} ${\rm deg}_{\bf x}(b) \in \ZZ^n_{\geq 0}$
for every non-zero $b \in R$ . Again for all non-zero $b,b' \in R$ one has
\begin{equation}\label{eq:degree-bb-x}
{\rm lt}_{\bf x}(bb') = {\rm lt}_{\bf x}(b) \, {\rm lt}_{\bf x}(b') \hs \mbox{and} \hs \deg_{\bf x}(bb') = 
\deg_{\bf x}(b) + \deg_{\bf x}(b').
\end{equation}
With the lower-triangular matrices $E$  in \eqref{eq:E-Et} and $F = E^{-1}$, we also note that
\begin{equation}\label{eq:E-degree}
f \prec g \hs \Leftrightarrow \hs E^{t}f \prec E^{t}g \hs \Leftrightarrow \hs F^{t}f \prec F^{t}g, \hs \hs f, g \in \ZZ^n.
\end{equation}

\begin{lemma}\label{le:leading}
Let $b \in R \subset \calT$ be non-zero. Then for $\xi \in {\bf k}^*$ and $f \in \ZZ^n_{\geq 0}$,  one has 
${\rm lt}_{\bf x}(b) = \xi {\bf x}^f$ if and only if ${\rm lt}_{\bfy}(b) = \xi {\bfy}^{E^{t}f}$. In particular, 
\[
\deg_{\bf x}(b) = F^t {\rm deg}_{\bf y}(b) \hs \mbox{and} \hs \deg_{\bf y}(b) = E^t {\rm deg}_{\bf x}(b).
\]
\end{lemma}

\begin{proof}
 For  $k \in [1, n]$, it follows from $x_k = (y_k + c_k)/y_{p(k)}$ and $c_k \in R_{k-1}$
that
${\rm lt}_\bfy(x_k) = y_k/y_{p(k)}$. By \eqref{eq:degree-bb},  for any $g = (g_1, \ldots, g_n)^t \in \ZZ^n_{\geq 0}$, one has
\[
{\rm lt}_\bfy({\bf x}^g) = \left(\frac{y_1}{y_{p(1)}}\right)^{g_1} \, \left(\frac{y_2}{y_{p(2)}}\right)^{g_2} \, \cdots\,
\left(\frac{y_n}{y_{p(n)}}\right)^{g_n} =
\bfy^{E^{t}g}.
\]
\autoref{le:leading} now follows from \eqref{eq:E-degree}.
\end{proof}

\subsection{Uniqueness of solutions to the GSV Equations}\label{ss:linear-alg}
Let $R$ be a length $n$ $\T$-Poisson CGL extension 
as in \autoref{de:CGL-intro}, and we  
continue with the notation in $\S$\ref{ss:bfy}.
Recall that $\lambda_k = \chi_k(h_k) \neq 0$ for $k \in [1, n]$, and recall from 
$\S$\ref{ss:main-intro} the matrix
\begin{equation}\label{eq:bnu}
\bnu = \left(\begin{array}{ccccc} \lambda_1 & \nu_{1,2} & \nu_{1,3} & \cdots &\nu_{1, n}\\
0 & \lambda_2 & \nu_{2,3} & \cdots & \nu_{2, n}\\
0 & 0 & \lambda_3 & \cdots & \nu_{3, n}\\
\cdots & \cdots & \cdots & \cdots &\cdots\\
0 & 0 & 0 & \cdots & \lambda_n\end{array}\right),
\end{equation}
where 
 $\nu_{j, k} = \chi_j(h_k) + \chi_k(h_j)$ for $j, k \in [1, n]$. For $\kappa \in X(\TT)$, let
$\overrightarrow{\kappa} = (\kappa(h_1), \ldots, \kappa(h_n))^t  \in {\bf k}^n$.

\begin{lemma}\label{:linear-alg-1}
For any $g \in {\bf k}^n$ and $\kappa \in X(\TT)$, the linear system 
\begin{equation}\label{eq:q-chi-f}
{\bf q} f = -g \hand \chi_{\bf y} f = \kappa
\end{equation}
has a solution $f \in {\bf k}^n$ if and only if 
$\chi_{\bf x} \bnu^{-1}(\overrightarrow{\kappa}+Eg)=\kappa$, and in this case
the solution $f \in {\bf k}^n$ to \eqref{eq:q-chi-f} is unique and  is given by
$f = E^t \bnu^{-1}(\overrightarrow{\kappa}+ Eg)$.
\end{lemma}

\begin{proof} As ${\bf q} = F \blam F^t, \chi_{\bf y} = \chi_{\bf x} F^t$, and $E = F^{-1}$,
the equations in 
\eqref{eq:q-chi-f} are equivalent to 
\begin{equation}\label{eq:q-chi-tf}
\blam \widetilde{f} = -Eg \hand \chi_{\bf x}  \widetilde{f} = \kappa.
\end{equation}
where $\widetilde{f} = F^tf \in \bk^n$. Let $\bsigma = \blam + \bnu$, i.e., 
\[
\bsigma = \left(\displaystyle \begin{array}{cccc} \chi_1(h_1) & \chi_2(h_1) &  \cdots & \chi_n(h_1)\\
\chi_1(h_2) & \chi_2(h_2)  & \cdots & \chi_n(h_2)\\
\cdots & \cdots & \cdots& \cdots \\
\chi_1(h_n) & \chi_2(h_n)  & \cdots & \chi_n(h_n)\end{array}\right).
\]
Suppose that $\widetilde{f} \in {\bf k}^n$ satisfies
$\chi_{\bf x} \widetilde{f} = \kappa$. Evaluating both sides of $\chi_{\bf x} \widetilde{f} = \kappa$ at 
$h_j$ for every $j \in [1, n]$ gives $\bsigma \widetilde{f} = \overrightarrow{\kappa}$. As $\bnu = \bsigma - \blam$, 
the equations in \eqref{eq:q-chi-tf} are now equivalent to 
\[
\bnu \widetilde{f} = \overrightarrow{\kappa}+Eg \hand \chi_{\bf x} \widetilde{f} = \kappa.
\]
As $\bnu$ is invertible,   \autoref{:linear-alg-1} now follows.
\end{proof}

Recall now that we have introduced in $\S$\ref{ss:main-intro}
the diagonal matrix (see convention in $\S$\ref{ss:nota-intro})
\begin{equation}\label{eq:Lambda}
\Lambda = (\Lambda_{i, j}) \in {\rm Mat}_{n \times {\rm ex}}({\bf k}) \hs \mbox{with} \hs
\Lambda_{j, j} = \lambda_{s(j)}, \hs j \in {\rm ex}.
\end{equation}
and recall  the GSV Equations for $R$ given in \eqref{eq:GSV-intro}.

\begin{corollary}\label{cor:M-unique}
For any $\TT$-Poisson  CGL extension $R$ of length $n$,  
the GSV Equations 
\[
{\bf q} M = -\Lambda \hs \mbox{and} \hs \chi_{\bfy} M = 0.
\]
for $R$ have a solution 
 $M \in {\rm Mat}_{n \times {\rm ex}}(\bk)$ if and only if $\chi_{\bf x} \bnu^{-1} E \Lambda =0$, 
and in such a case the solution  is unique and is given by 
$M = E^t \bnu^{-1} E \Lambda$.
\end{corollary}

\begin{proof}
This is a direct consequence of \autoref{:linear-alg-1}.
\end{proof}

The following \autoref{le:b-1} will be used in $\S$\ref{ss:bfy-M} to show the existence of 
a solution to the GSV Equations in \eqref{eq:GSV-intro}.

\begin{lemma}\label{le:b-1}
Given $g = (g_1, \ldots, g_n)^t \in {\bf k}^n$ and $\kappa \in X(\T)$, if a non-zero 
$b \in \calT = \bk[\bfy^{\pm 1}]$ is $\TT$-homogeneous with $\T$-weight $\kappa$ and satisfies
\[
\{b, \, y_l\} = g_l b y_l, \hs \forall \; l \in [1, n],
\]
then $b$ is a non-zero scalar multiple of the Laurent monomial ${\bf y}^f$, where $f \in \ZZ^n$
is a unique solution (c.f. \autoref{:linear-alg-1}) to the
system of linear equations 
\begin{equation}\label{eq:qfg}
{\bf q} f = -g \hand \chi_{\bf y} f = \kappa.
\end{equation}
\end{lemma}

\begin{proof}
Write  $b = \sum_{f \in {\rm supp}(b)} b_f \bfy^f$, where 
$b_f \in {\bf k}^\times$ for $f \in {\rm supp}(b) \subset  \ZZ^n$. 
Let $l \in [1, n]$. It follows from  $\{y_l, \, b\} = -g_lby_l$ and
\eqref{eq:bra-bfy-f-g} that
\[
\sum_{f \in {\rm supp}(b)} b_f (e_l^t {\bf q} f) \bfy^{e_l+f}
= -\sum_{f \in {\rm supp}(b)} g_lb_f \bfy^{e_l+f},
\]
so $e_l^t {\bf q} f= -g_l$ for every $f \in {\rm supp}(b)$. Thus ${\bf q}f = -g$
for every $f \in {\rm supp}(b)$.
Similar arguments show that 
every monomial term of $b$ is a $\TT$-weight vector with $\TT$-weight $\kappa$, i.e.,
$\chi_{\bf y} f = \kappa$ for every $f \in {\rm supp}(b)$. By \autoref{:linear-alg-1}, such
an $f \in \ZZ^n$ is necessarily unique. In particular, $b$ is  non-zero multiple of a monomial in $\bfy$.
\end{proof}

\begin{remark}\label{re:terms-in-b}
{\rm
The proof of \autoref{le:b-1} also shows that if a  non-zero $b \in \calT$ and  $l \in [1, n]$ are such that
$\{b, y_l\} = g_lby_l$ for some $g_l \in {\bf k}$, then $\{b', y_l\} = g_l b^\prime y_l$ for every
$b' = \sum_{f \in {\rm supp}^\prime(b)} b_f\bfy^f$ with non-empty ${\rm supp}^\prime(b)\subset {\rm supp}(b)$. 
\hfill $\diamond$
}
\end{remark}

\subsection{The initial $\TT$-Poisson pre-seed $({\bf y}, M)$}\label{ss:bfy-M}
Let again  $R = (\bk[x_1, \ldots,x_n], \{\, , \, \})$ be a $\T$-Poisson CGL extension 
as in \autoref{de:CGL-intro}, and let 
$\bfy = (y_1, \ldots, y_n)$ be the sequence of homogeneous Poisson prime elements associated to $R$.
 Recall from \autoref{thm:5.5} 
the successor map $s: [1, n] \to [2, n] \sqcup\{+\infty\}$. 
By  \eqref{eq:ykk-00} we have 
\begin{equation}\label{eq:Rk-T}
R_k \subset {\bf k}[y_1, \ldots, y_k][y_i^{-1}: j \in [1, k], s(j)\leq k\}, \hs k\in [1, n].
\end{equation}
For $k \in [1, n]$, set 
\begin{equation}\label{eq:ck}
c_k =\frac{\delta_k(y_{p(k)})}{\lambda_k} \in R_{k-1} \backslash\{0\},
\end{equation}
so that (see \eqref{eq:yk-ypk})  $\delta_k(c_k) = 0$, and 
\begin{equation}\label{eq:yk-ck-11}
y_k = y_{p(k)}x_k - c_k.
\end{equation}
Recall now that 
${\rm ex} = \{j \in [1, n]: \; s(j) \neq +\infty\}$.
For $j \in {\rm ex}$, setting 
\begin{equation}\label{eq:calT-prime}
\calT^\prime_{s(j)-1} = {\bf k}[y_1^{\pm 1}, \ldots, y_{j-1}^{\pm 1}, \, y_{j+1}^{\pm 1}, \ldots, y_{s(j)-1}^{\pm 1}],
\end{equation}
by \eqref{eq:Rk-T} we then have $c_{s(j)} \in R_{s(j)-1} \subset \calT^\prime_{s(j)-1}[y_j]$. 

\begin{notation}\label{nota:b-sj}
{\rm
For $j \in {\rm ex}$, let
$a_{s(j)} \in \calT^\prime_{s(j)-1}[y_j]$ and  $b_{s(j)} \in \calT^\prime_{s(j)-1}$ be such that
\begin{equation}\label{eq:csj-bsj}
c_{s(j)} = y_j a_{s(j)} + b_{s(j)}.
\end{equation}
In other words, $b_{s(j)}$ is the constant term of  $c_{s(j)}$ when 
expressed as
a polynomial in $y_j$ with coefficients in $\calT^\prime_{s(j)-1}$.
}
\end{notation}

Recall that ${\bf q} = (q_{k, j})_{k, l \in [1, n]}$ is
the Poisson coefficient matrix of ${\bfy}$. The following
\autoref{le:Mi-bsj} is 
proved in \cite{Mi:thesis, Mi:archive}. 
We include a proof 
for the convenience of the reader.

\begin{lemma}\label{le:Mi-bsj}
For every $j \in {\rm ex}$, the element $b_{s(j)} \in \calT^\prime_{s(j)-1}$ is non-zero, and one has 
\begin{equation}\label{eq:b-yl-1}
\{b_{s(j)}, \, y_l\} = \begin{cases} q_{s(j), l}b_{s(j)} y_l, & \hs l \in [1, n], \, l \neq j,\\
(q_{s(j), j} + \lambda_{s(j)})b_{s(j)} y_j, & \hs l = j,\end{cases}
\end{equation}
\end{lemma}

\begin{proof}
Let $j \in {\rm ex}$. 
 We first prove that $b_{s(j)} \neq 0$. By \autoref{thm:5.5}, $c_{s(j)} \neq 0$.
Suppose that $b_{s(j)}=0$. Then $c_{s(j)}=y_{j}\frac{a_1}{a_2}$, where $a_1\in R_{s(j)-1}$
and $a_2$ is a  monomial in $\{y_i: i \in [1, s(j)-1]\backslash\{j\}\}$ with non-negative exponents. It then follows from
$c_{s(j)}a_2 = y_{j}a_1$ that $c_{s(j)} \in y_{j}R_{s(j)-1}$,
contradicting $y_{s(j)}=y_{j}x_{s(j)}-c_{s(j)}$ being prime in $R_{s(j)}$. Thus, $b_{s(j)} \neq 0$.

To prove \eqref{eq:b-yl-1},
note first that since $y_{s(j)} = y_j x_{s(j)}-c_{s(j)}$, for every  $l \in [1, n]$ we have
\begin{align}\nonumber
\{c_{s(j)}, \, y_l\} & = \{y_j x_{s(j)} - y_{s(j)}, \, y_l\}=y_j\{x_{s(j)},y_l\} + x_{s(j)}\{y_j, y_l\} - \{y_{s(j)}, \, y_l\}\\
\label{eq:csj-yl}
& = \{x_{s(j)},y_l\}y_j + (q_{j, l}-q_{s(j),l})x_{s(j)}y_jy_l + q_{s(j), l}c_{s(j)} y_l.
\end{align}
Assume first that $l \geq s(j)$. Then $\{x_{s(j)},y_l\}= (q_{s(j), l} - q_{j, l}) x_{s(j)} y_l$
by  \autoref{le:q-jpj}, 
Thus
\[
\{c_{s(j)}, \, y_l\}  = q_{s(j), l} y_j x_{s(j)} y_l - q_{s(j), l}y_{s(j)} y_l 
=q_{s(j), l} c_{s(j)}y_l.
\]
By \autoref{re:terms-in-b}, one has
$\{b_{s(j)},  y_l\}  = q_{s(j), l} b_{s(j)}y_l$. 
Assume now that $l< s(j)$. On the one hand, 
\begin{align*}
\{c_{s(j)}, \, y_l\} &= \{y_ja_{s(j)} + b_{s(j)}, \, y_l\} = y_j \{a_{s(j)}, \, y_l\}
+ a_{s(j)} \{y_j, \,y_l\} + \{b_{s(j)}, \,y_l\} \\
& =\{a_{s(j)}, \, y_l\}y_j + q_{j, l}a_{s(j)} y_ly_j  + \{b_{s(j)}, \,y_l\}.
\end{align*}
On the other hand, $\{x_{s(j)}, y_l\} = 
(q_{s(j), l} - q_{j, l}) x_{s(j)} y_l + \delta_{s(j)} (y_l)$ by
 \autoref{le:q-jpj}.  Thus by \eqref{eq:csj-yl},
\[
\{c_{s(j)}, \, y_l\} = q_{s(j), l} c_{s(j)} y_l + \delta_{s(j)}(y_l)y_j = 
q_{s(j), l}a_{s(j)}y_ly_j + \delta_{s(j)}(y_l)y_j + q_{s(j), l} b_{s(j)}y_l.
\]
When $l \neq j$, since both $\{a_{s(j)}, \, y_l\}$ and 
$a_{s(j)} y_l$ are in 
$\calT^\prime_{s(j)-1}[y_j]$, and since
\[
\delta_{s(j)}(y_l) \in R_{s(j)-1} \subset \calT^\prime_{s(j)-1}[y_j] \hs \mbox{and} \hs
\{b_{s(j)}, \, y_l\} \in \calT^\prime_{s(j)-1},
\]
by  comparing the constant
terms of the above two expressions of $\{c_{s(j)}, y_l\}$ as a polynomial in $y_j$ with coefficient in $\calT^\prime_{s(j)-1}$,
we get $\{b_{s(j)}, \,y_l\} = q_{s(j), l} b_{s(j)}y_l$.
Let now  $l = j$. Since 
\[
\{\calT^\prime_{s(j)-1}, y_j\} \subset y_j \calT^\prime_{s(j)-1},
\]
the above two expressions of $\{c_{s(j)}, y_j\}$ are both  in $y_j \calT^\prime_{s(j)-1}[y_j]$.
Since 
\[
\delta_{s(j)}(y_j) = \lambda_{s(j)} c_{s(j)} = \lambda_{s(j)} y_ja_{s(j)} + \lambda_{s(j)} b_{s(j)},
\]
comparing the linear terms in $y_j$ in the two expressions of $\{c_{s(j)}, y_j\}$, we get
\[
\{b_{s(j)}, \,y_j\} = (q_{s(j), j} + \lambda_{s(j)}) b_{s(j)}y_j.
\]
This finishes the proof of \eqref{eq:b-yl-1}.
\end{proof}

\begin{proposition}\label{pr:bsj-1}
For every $j \in {\rm ex}$, the element $b_{s(j)} \in \calT_{s(j)-1}^\prime$ in 
\eqref{eq:csj-bsj} is a non-zero scalar multiple of a
Laurent 
monomial in $(y_1, \ldots, y_{j-1}, y_{j+1}, \ldots, y_{s(j)-1})$. Writing, for $j \in {\rm ex}$,
\begin{equation}\label{eq:iota-bsj}
\frac{b_{s(j)}}{y_{s(j)}} = \iota_{s(j)} \bfy^{M_j},
\end{equation}
where $\iota_{s(j)} \in {\bf k}^\times$ and
$M_j = \left(M_{1, j}, \ldots, M_{j-1, j}, 0, M_{j+1, j}, \ldots, 
M_{s(j)-1, j}, -1, 0, \ldots, 0\right)^t \in \ZZ^n$, then
\begin{equation}\label{eq:bfq-Mj}
{\bf q} M_j = -\lambda_{s(j)}e_j \hs \mbox{and} \hs \chi_{\bfy} M_j = 0.
\end{equation}
\end{proposition}

\begin{proof} 
Let $j \in {\rm ex}$. By \autoref{le:Mi-bsj} and \autoref{le:b-1}, $b_{s(j)}$ is a non-zero scalar
multiple of a Laurent monomial in $\bfy$. As $c_{s(j)} \in R_{s(j)-1}$ is a $\T$-weight vector with the same $\T$-weight as $y_{s(j)}$,
the element $\frac{b_{s(j)}}{y_{s(j)}} \in \calT$ is a $\TT$-weight vector 
with  $\T$ weight $0$. Thus $\chi_{\bfy} M_j = 0$. 
On the other hand, a direct calculation shows that \eqref{eq:b-yl-1} is equivalent to 
\begin{equation}\label{eq:b-yl}
\left\{\frac{b_{s(j)}}{y_{s(j)}}, \; y_l\right\} = \begin{cases} 0, & \hs l \in [1, n], \, l\neq j,\\
\lambda_{s(j)} \frac{b_{s(j)}}{y_{s(j)}} y_l, & \hs l = j,\end{cases}
\end{equation}
which, by \autoref{le:b-1} again, gives ${\bf q} M_j = -\lambda_{s(j)}e_j$.
\end{proof}

\begin{remark}\label{rk:m-pos}
{\rm
For $j \in {\rm ex}$, it follows from \eqref{eq:Rk-T} that $b_{s(j)}$ contains no negative power
of $y_i$ for any $i \in [1, s(j)-1]$ such that $s(i) \geq s(j)$. In the notation of
\autoref{pr:bsj-1}, we thus have $M_{i, j} \geq 0$ for all $i \in [1, s(j)-1]$ such that $s(i) \geq s(j)$.
\hfill $\diamond$
}
\end{remark}

Recall now from $\S$\ref{ss:main-intro} and \eqref{eq:Lambda} the
diagonal matrix $\Lambda$. 
We can now prove our first result on arbitrary $\TT$-Poisson CGL extensions.

\begin{theorem}\label{thm:M-1}
Let $R$ be a length $n$ Poisson CGL extension as in \autoref{de:CGL-intro},
and let ${\rm ex} \subset [1, n]$
be as in \eqref{eq:ex}.
Let $M \in {\rm Mat}_{n \times {\rm ex}}(\ZZ)$ whose $j^{\rm th}$ column for $j \in {\rm ex}$
is $M_j$ in \eqref{eq:iota-bsj}.  Then 
the integer matrix 
$M$ is a unique solution  to the GSV Equations
\begin{equation}\label{eq:M-linear}
{\bf q} M = -\Lambda \hand \chi_{\bf y} M = 0
\end{equation}
in ${\rm Mat}_{n \times {\rm ex}}({\bf k})$. Moreover, with 
$E$ and $\bnu$ respectively given in \eqref{eq:E-Et} and \eqref{eq:bnu}, one has
\begin{equation}\label{eq:M-explicit-0}
M = E^t\bnu^{-1} E \Lambda.
\end{equation}
Furthermore, if 
\begin{equation}\label{eq:QQ-11}
\frac{\lambda_{s(j)}}{\lambda_{s(k)}} \in \QQ, \hs  \forall j, k \in {\rm ex},
\end{equation}
then there exists 
a diagonal $\varepsilon \in {\rm Mat}_{{\rm ex} \times {\rm ex}}(\ZZ)$ with diagonal entries
$\pm 1$ such that $M\ep \in {\rm Mat}_{n \times {\rm ex}}(\ZZ)$ is skew-symmetric, and 
$(\bfy, M)$ is a $\TT$-Poisson seed in ${\rm Frac}(R)$.
\end{theorem}

\begin{proof} 
By \autoref{pr:bsj-1}, $M$ satisfies \eqref{eq:M-linear}, which, by
\autoref{cor:M-unique}, is the only solution of  \eqref{eq:M-linear} in ${\rm Mat}_{n \times {\rm ex}}(\bk)$
and must be given by $M = E^t\bnu^{-1} E \Lambda$.
\end{proof}

\begin{definition}\label{de:bfy-M-pre-seed}
{\rm
1)  Without assuming that $M$ is skew-symmetrizable, we call the pair 
$(\bfy, M)$ in \autoref{thm:M-1} 
the {\it initial $\TT$-Poisson pre-seed} in ${\rm Frac}(R)$ associated to the
$\TT$-Poisson CGL extension $R$ (see \autoref{de:pre-seed});

2) If 
$\varepsilon \in {\rm Mat}_{{\rm ex} \times {\rm ex}}(\ZZ)$  is diagonal with diagonal entries
$\pm 1$ such that $M\ep \in {\rm Mat}_{n \times {\rm ex}}(\ZZ)$ is skew-symmetrizable, we
call $(\bfy, M\ep)$ an {\it initial $\TT$-Poisson seed associated to $R$}.
\hfill $\diamond$
}
\end{definition}

\subsection{Re-scaling of the CGL generators}\label{ss:normalization}
Let $R = (\bk[x_1, \ldots, x_n], \{\, , \, \})$ be a $\TT$-Poisson CGL extension $R$ as in \autoref{de:CGL-intro}, and we
continue with the notation from $\S$\ref{ss:bfy-M}.

\begin{definition}\label{de:normal}
{\rm
The  $\TT$-Poisson CGL extension $R$
is said to be {\it normal in the CGL generators $(x_1, \ldots, x_n)$ } if $\iota_{s(j)}=1$ for all $j \in {\rm ex}$, 
where $\iota_{s(j)} \in {\bf k}^\times$ is as in \autoref{pr:bsj-1}.
}
\end{definition}

\begin{remark}\label{rk:normal}
{\rm
Goodearl and Yakimov define in \cite[$\S$9.2]{GY:Poi-CGL} the notion of a symmetric CGL being normal. 
We will see in \autoref{rk:GY-normal-1} that when a Poisson CGL extension $R$  is symmetric 
it is normal in the sense of 
\autoref{de:normal} if and only if it is normal in the sense of \cite[$\S$9.2]{GY:Poi-CGL}. 
\hfill $\diamond$
}
\end{remark}

For $\gamma = (\gamma_1, \ldots, \gamma_n) \in
({\bf k}^\times)^n$, consider the new CGL generators $(\tilde{x}_1, \ldots, \tilde{x}_n)$ of
$R$ given by $\tilde{x}_k = \gamma_kx_k$ for $k \in [1, n]$. Then
\begin{equation}\label{eq:x-jk-tilde}
\{\widetilde{x}_j, \, \widetilde{x}_k\} = -\chi_j(h_k) \widetilde{x}_j\widetilde{x}_k-\gamma_k
\delta_k(\widetilde{x}_j), \hs 1 \leq j < k \leq n.
\end{equation}

Recall the matrix 
$F = E^{-1}$ in \eqref{eq:F}.

\begin{lemma}\label{le:rescaling}
For  $\gamma =(\gamma_1, \ldots, \gamma_n) \in ({\bf k}^\times)^n$, the 
$\TT$-Poisson CGL extension $R$ is normal in the CGL generators
$(\widetilde{x}_1, \ldots, \widetilde{x}_n) =(\gamma_1x_1, \ldots, \gamma_nx_n)$ if 
\begin{equation}\label{eq:gamma-iota}
\iota_{s(j)}=\gamma^{F^tM_j}, \hs \forall \;\; j \in {\rm ex}.
\end{equation}
\end{lemma}

\begin{proof} Let 
$\widetilde{\bfy} = (\widetilde{y}_1, \ldots, \widetilde{y}_n)$ be the sequence of homogeneous Poisson 
prime elements of $R$ with respect to the CGL generators $(\widetilde{x}_1, \ldots, \widetilde{x}_n)$. 
By \autoref{thm:5.5}, 
$\widetilde{y}_k = \nu_k y_k$ for some $\nu_k \in {\bf k}^\times$ for each $k \in [1, n]$. 
By \eqref{eq:x-jk-tilde}, 
the corresponding derivation $\widetilde{\delta}_k$ of $R_{k-1}$ is 
$\widetilde{\delta}_k = \gamma_k \delta_k$, and 
\[
\widetilde{y}_k= \widetilde{y}_{p(k)}\widetilde{x}_k- 
\frac{\widetilde{\delta}_k(\widetilde{y}_{p(k)})}{\lambda_k} = \gamma_k \nu_{p(k)} y_{p(k)} x_k -
\frac{\gamma_k \nu_{p(k)}}{\lambda_k} \delta_k(y_{p(k)}) = \gamma_k \nu_{p(k)} y_k.
\]
Thus $\nu_k = \gamma^{\overline{e}_k}$ and  $\widetilde{y}_k = \gamma^{\overline{e}_k} y_k$ for all $k \in [1, n]$, where recall from \eqref{eq:F} that
$\overline{e}_k = F^te_k = \sum_{j \in L(k) \cap [1, k]}e_j
\in \ZZ^n$. In other words, we have $\widetilde{\bfy} = \gamma^{F^t}{\bf y}$.

Let $j \in {\rm ex}$ and
let $\widetilde{b}_{s(j)}$ be defined as in \eqref{eq:csj-bsj} using the CGL generators
$(\widetilde{x}_1, \ldots, \widetilde{x}_n)$.
It then follows from 
$\widetilde{y}_{s(j)} = \widetilde{x}_{s(j)} \widetilde{y}_j - \widetilde{a}_{s(j)}\widetilde{y}_j -\widetilde{b}_{s(j)}$ that $\widetilde{b}_{s(j)} = \gamma^{\bar{e}_{s(j)}} b_{s(j)}$.
Write $\widetilde{b}_{s(j)} = \widetilde{\iota}_{s(j)} \widetilde{\bfy}^{e_{s(j)}+M_j}$
with $\widetilde{\iota}_{s(j)}\in {\bf k}^\times$. Then 
\[
\gamma^{\bar{e}_{s(j)}} \iota_{s(j)} \bfy^{e_{s(j)}+M_j} =\widetilde{b}_{s(j)} = \widetilde{\iota}_{s(j)} \widetilde{\bfy}^{e_{s(j)}+M_j}
= \widetilde{\iota}_{s(j)} \gamma^{\overline{e}_{s(j)}+F^tM_j} {\bfy}^{e_{s(j)}+M_j}.
\]
It follows that $\iota_{s(j)}  = \widetilde{\iota}_{s(j)} \gamma^{F^tM_j}$. Hence 
$\widetilde{\iota}_{s(j)}=1$ if and only if  $\iota_{s(j)}=\gamma^{F^tM_j}$.
\end{proof}

\begin{corollary}\label{cor:normaliztion}
Every $\TT$-Poisson CGL extension can be normalized by rescaling its CGL generators.
\end{corollary}

\begin{proof}
Let $R$ be a $\TT$-Poisson CGL extension with CGL generators $(x_1, \ldots, x_n)$. 
By the formula for $M_j$ for $j \in {\rm ex}$ in \autoref{pr:bsj-1}, 
the $k^{\rm th}$ entry of $F^tM_j \in \ZZ^n$ is $-1$ for $k = s(j)$ and 
zero for $k > s(j)$. Thus \eqref{eq:gamma-iota} 
expresses $\gamma_{s(j)}$ in terms of $\gamma_1, \ldots, \gamma_{s(j)-1}$ and $\iota_{s(j)}$.
Setting $\gamma_k = 1$ if $p(k) = -\infty$ and solving for $\gamma_k$ recursively from 
\eqref{eq:gamma-iota}, one sees that $R$ is normal in the rescaled CGL generators
$(\gamma_1x_1, \ldots, \gamma_nx_n)$.
\end{proof}

\subsection{Upper cluster structures associated to Poisson CGL extensions}\label{ss:proof-M-2}
We now prove our second main result on arbitrary $\TT$- Poisson CGL extensions.

\begin{theorem}\label{thm:M-2}
Let $R = (\bk[x_1, \ldots, x_n], \{\, , \, \})$ be any $\TT$-Poisson CGL extension, and let
$({\bf y}, \; M\ep)$ 
be any initial $\TT$-Poisson seed associated to $R$ (see \autoref{de:bfy-M-pre-seed}). 
Assume that $R$ is normal.
Then for every ${\rm inv} \subset [1, n]\backslash {\rm ex}$, one has
\[
{\calU}(\bfy, M\varepsilon;\, {\rm inv}) = R[y_j^{-1}: j \in {\rm inv}].
\]
\end{theorem}

\begin{proof} 
Recall that we have set 
$\overline{\calU}(\bfy, M\ep)= {\calU}(\bfy, M\ep; \emptyset)$.
We first prove that $\overline{\calU}(\bfy, M\ep) = R$.

For any extended cluster ${\bf y}^\prime = (y_1^\prime, \ldots, y_n^\prime)$ in $[({\bf y}, M)]$, set
\[
\calL({\bf y}^\prime) = {\bf k}[y_1^\prime, \ldots, y_n^\prime][(y_j^\prime)^{-1}: j \in {\rm ex}].
\]
It follows from ${\bf q} M\ep = -\Lambda \ep$ that $M\ep$ has full rank. By \cite[Theorem 3.11]{GSV:2018} one has 
\begin{equation}\label{eq:upper-bound}
\overline{\calU}(\bfy, M\ep) = \calL({\bf y}) \cap \bigcap_{j \in {\rm ex}} \calL(\bfy[j]),
\end{equation}
where for $j \in {\rm ex}$, $\bfy[j] :=\mu_j^{M\ep}({\bf y})$  is the extended cluster of the mutation of the 
seed $({\bf y}, M\ep)$ in the direction $j$. Using \eqref{eq:upper-bound}, we 
now show that $R \subset \overline{\calU}(\bfy, M\ep) \subset R$.

Let $j \in {\rm ex}$, and let 
$\bfy[j] = (y_1, \ldots, y_{j-1}, y_j^\prime, y_{j+1}, \ldots, y_n)$. We first show that $y_j^\prime \in R_{s(j)}$.
Let $b_{s(j)}$ be as in \eqref{eq:csj-bsj}.  Since $R$ is normal in $(x_1, \ldots, x_n)$, we have
$b_{s(j)} = b_{s(j)}^+/b_{s(j)}^-$, where
\[
b_{s(j)}^+ = \prod_{i \in [1, s(j)-1], \,m_{i,j} > 0} y_i^{m_{i,j}} \hs \mbox{and} \hs
b_{s(j)}^- = \prod_{i \in [1, s(j)-1], \,m_{i,j} < 0} y_i^{-m_{i,j}}.
\]
By the definition of $y_j^\prime$ and using $y_{s(j)} = y_j x_{s(j)} - y_j a_{s(j)}-b_{s(j)}$, we have
\begin{equation}\label{eq:yj-prime}
y_j^\prime = \frac{y_{s(j)} b_{s(j)}^- + b_{s(j)}^+}{y_j} = b_{s(j)}^-(x_{s(j)} - a_{s(j)})
= b_{s(j)}^-x_{s(j)} - b_{s(j)}^-a_{s(j)}.
\end{equation}
To show that $y_j^\prime \in R_{s(j)}$, it suffices to show that $b_{s(j)}^-a_{s(j)} \in R_{s(j)-1}$. 
Recall from \eqref{eq:calT-prime} that 
\[
\calT^\prime_{s(j)-1} = {\bf k}[y_1^{\pm 1}, \ldots, y_{j-1}^{\pm 1}, \, y_{j+1}^{\pm 1}, \ldots, y_{s(j)-1}^{\pm 1}].
\]
As $b_{s(j)}^-$ contains no power of $y_j$ and $a_{s(j)}\in \calT^\prime_{s(j)-1}[y_j]$, 
we have $b_{s(j)}^-a_{s(j)}\in \calT^\prime_{s(j)-1}[y_j]$.
 Since all the $y_i$'s for $i \in [1, n]$ are prime elements in $R$, by first writing
\[
b_{s(j)}^-a_{s(j)} = \frac{\phi_1(y_1, \ldots, y_{s(j)-1})}{{\bf y}^{g}}\in \calT^\prime_{s(j)-1}[y_j]
\]
for some $\phi_1 \in {\bf k}[ y_1, \ldots, y_{s(j)-1}]$ and 
$g =(g_1, \ldots, g_{s(j)-1})^t \in (\ZZ_{\geq 0})^{s(j)-1}$ with $g_j = 0$ and further taking the prime factorization of
$\phi_1(y_1, \ldots, y_{s(j)-1})$ as an element in $R_{s(j)-1}$, we can 
 write
\[
b_{s(j)}^-a_{s(j)} =\frac{\phi}{\bfy^f},
\]
where
$\phi \in R_{s(j)-1}$, $f =(f_1, \ldots, f_{s(j)-1})^t \in (\ZZ_{\geq 0})^{s(j)-1}$ with $f_j = 0$,
and $\phi$ and $\bfy^f$ are co-prime
in $R_{s(j)-1}$. On the other hand, setting
$\psi = y_j b_{s(j)}^-a_{s(j)}$, we have
\[
\psi =  b_{s(j)}^-y_j a_{s(j)} = b_{s(j)}^-(c_{s(j)}-b_{s(j)}) = b_{s(j)}^-c_{s(j)}-b_{s(j)}^+  \in R_{s(j)-1}.
\] 
It then follows from $y_j\phi = \bfy^f \psi$ and the assumption that $y_j\phi$ and ${\bf y}^f$ are co-prime
in $R_{s(j)-1}$ that $f = 0$. Thus 
$b_{s(j)}^-a_{s(j)}=\phi \in R_{s(j)-1}$. By \eqref{eq:yj-prime}, $y_j^\prime\in R_{s(j)}$.

Let again $j \in {\rm ex}$. By \eqref{eq:yj-prime} and the fact  that
$b_{s(j)}^-a_{s(j)}\in R_{s(j)-1}$, we have (see $\S$\ref{ss:two-gradings})
\[
{\rm deg}_{\bf x}(y_j^\prime) = {\rm deg}_{\bf x} (b_{s(j)}^-x_{s(j)}) = e_{s(j)} + {\rm deg}_{\bf x} (b_{s(j)}^-)
= e_{s(j)} + F^t {\rm deg}_{\bf y}(b_{s(j)}^-).
\]
Since $b_{s(j)}^-$ is a monomial of $\{y_i: i \in [1, s(j)-1]\backslash\{j\}]$, the $j^{\rm th}$ entry 
of ${\rm deg}_{\bf x}(y_j^\prime)$ is $0$.  On the other hand, the $j^{\rm th}$ entry of 
${\rm deg}_{\bf x} y_j$ is $1$. It then  follows from \eqref{eq:degree-bb-x} that $y_j^\prime$ is 
not divisible
by $y_j$ in $R$, i.e., $y_j^\prime \in R$ is co-prime with $y_j$. By Starfish Lemma \cite[Remark 6.4.4]{FWZ:introduction}, 
$\overline{\calU}(\bfy, M\ep) \subset R$.

To show that $R \subset \overline{\calU}(\bfy, M\ep)$, note that we 
already know from \eqref{eq:ykk-00} that $R \subset \calL(\bfy)$. Fix $j \in {\rm ex}$.
We now prove that $x_i \in  \calL(\bfy[j])$ for every $i \in [1, n]$. 
Note from
\eqref{eq:ykk-00} that $R_i \subset \widetilde{\calT}_i$ for every $i\in [1, n]$, where
\[
\widetilde{\calT}_i= {\bf k}[y_1, \ldots, y_i][y_l^{-1}: \, l \in [1, i], \, s(l)\leq i] 
\subset {\bf k}[y_1^{\pm 1}, \ldots, y_i^{\pm 1}].
\]
If $i \in [1, j]$, then $\widetilde{\calT}_i \subset \calL(\bfy[j])$, so $x_i \in R_i \subset \calL(\bfy[j])$.
For $i \in [j+1, s(j)-1]$, since 
\begin{equation}\label{eq:xi-yi-ci}
x_i =\frac{y_i + c_i}{y_{p(i)}},
\end{equation}
and since $y_i \neq y_j, y_{p(i)} \neq y_j$ and $c_i \in R_{i-1}$, it follows by induction that 
$x_i \in \calL(\bfy[j])$. 
For $i = s(j)$, note from \eqref{eq:yj-prime} that one has
\[
x_{s(j)} = \frac{y_j^\prime}{b^-_{s(j)}} + a_{s(j)}.
\]
As $b^-_{s(j)}$ does not contain any power of $y_j$, we have 
$\frac{y_j^\prime}{b^-_{s(j)}} \in \calL(\bfy[j])$. 
As $y_j \in \calL(\bfy[j])$, and as ${\calT}^\prime_{s(j)-1} \subset \calL(\bfy[j])$, we have 
$a_{s(j)} \in {\calT}_{s(j)-1}^\prime[y_j] \subset \calL(\bfy[j])$.
Thus  $x_{s(j)} \in \calL(\bfy[j])$. Finally, for $i \in [s(j)+1, n]$, again by \eqref{eq:xi-yi-ci} and the facts that $y_i \neq y_j, y_{p(i)} \neq y_j$ and $R_{i-1} \in \calL(\bfy[j])$, we see by induction that
$x_i \in \calL(\bfy[j])$. This finises the proof that $R \subset \calL(\bfy[j])$.
As $j \in {\rm ex}$ is arbitrary, we have
$R \subset \overline{\calU}(\bfy, M\ep)$.

Let now ${\rm inv}$ be any subset of $[1, n]\backslash {\rm ex}$. 
Since $M\ep$ has full rank,  by \cite[Theorem 3.11]{GSV:2018} again (and by \cite[Corollary 1.9]{BFZ:III}
when ${\rm inv} = [1, n]\backslash {\rm ex}$), we have
\begin{equation}\label{eq:upper-bound-inv}
{\calU}(\bfy, M\ep; {\rm inv}) = \calL({\bf y}; {\rm inv}) \cap \bigcap_{j \in {\rm ex}} \calL(\bfy[j]; {\rm inv}),
\end{equation}
where $\calL({\bf y}^\prime; {\rm inv}) = \calL({\bf y}^\prime)[y_i^{-1}: i \in {\rm inv}]$
for any extended cluster ${\bf y}^\prime$  in $[({\bf y}, M\ep)]$. 
It follows from 
$R = \calL({\bf y}) \cap \bigcap_{j \in {\rm ex}} \calL(\bfy[j])$ 
that 
\[
R[y_i^{-1}: i \in {\rm inv}] \subset 
\calL({\bf y}; {\rm inv}) \cap \bigcap_{j \in {\rm ex}} \calL(\bfy[j]; {\rm inv})=
{\calU}(\bfy, M\ep; {\rm inv}).
\]
Conversely, given any $\varphi \in {\calU}(\bfy, M\ep; {\rm inv})$, by \eqref{eq:upper-bound-inv} one has
$\varphi = \psi \prod_{i \in {\rm inv}} y_i^{-n_i}$ for some 
\[
\psi \in \calL({\bf y}) \cap \bigcap_{j \in {\rm ex}} \calL(\bfy[j]) = R
\]
and some positive integer $n_i$ for each $i \in {\rm inv}$. Thus $\varphi \in 
R[y_j^{-1}: j \in {\rm inv}]$. We have thus proved that ${\calU}(\bfy, M\ep; {\rm inv}) = 
R[y_j^{-1}: j \in {\rm inv}]$.
This finishes the proof of \autoref{thm:M-2}.
\end{proof}

\section{Proof of \autoref{thm:B} on symmetric Poisson CGL extensions}\label{s:proof-B}
Throughout $\S$\ref{s:proof-B}, we assume that $R = (\bk[x_1, \ldots, x_n], \{\, , \, \})$ is a
$\TT$-Poisson CGL extension that is, in addition,  {\it symmetric} as in 
\autoref{de:sym-CGL-intro}. Let $\bfy$ be the sequence of homogeneous Poisson prime elements
associated to $R$ given in \autoref{thm:5.5}, 
and let 
\[
M = (M_j)_{j \in {\rm ex}}=E^t \bnu^{-1} E \Lambda  \in {\rm Mat}_{n \times {\rm ex}}(\ZZ)
\]
be given as in 
\autoref{thm:M-1}. 
Making use of the assumption that $R$ is symmetric,  
we first give in 
\autoref{thm:sym-M-1} alternative descriptions of each column $M_j$ 
of $M$ in terms of the $\bfy$-degree of
the element $c_{s(j)}$ as well as in terms of the ${\bf x}$-degree of the {\it tail term}
of the Poisson bracket $\{x_j, x_{s(j)}\}$. 
As reviewing  the $\TT$-Poisson CGL extension
$R_\tau$ for each $\tau \in \Xi_n \subset S_n$ and the Goodearl-Yakimov 
$\TT$-Poisson pre-seed $(\bfy_\tau, M_\tau)$, we prove 
\autoref{thm:sym-M-1}, \autoref{thm:M-3}, and \autoref{thm:M-tau-cases}, the combination of which  gives \autoref{thm:B} in $\S$\ref{ss:main-intro}.

\subsection{Notation and the matrices $Q$ and $\Theta$}\label{ss:nota-sym}
Regarding $R$ as a $\TT$-Poisson CGL extension in the ordered set $(x_1, \ldots, x_n)$ of 
CGL generators, we have the respective predecessor and successor maps of $R$ 
denoted as  (see \autoref{nota-p-s})
\[
p: \; [1, n] \longrightarrow \{-\infty\} \sqcup [1, n-1], 
\hs \mbox{and} \hs 
s: \; [1, n] \longrightarrow [2, n] \sqcup \{+\infty\}.
\]
Recall that each $j \in [1, n]$ belongs to a 
unique level set $L(j) \subset [1, n]$ associated to $R$ defined using $p$ and 
$s$ (see again \autoref{nota-p-s}).
With $h_1, \ldots, h_n, h_1^\ast, \ldots, h_n^\ast \in \t$ as in \autoref{de:sym-CGL-intro},
recall that we have set $\lambda_j = \chi_j(h_j)\in \bk^\times$ for $j \in [1, n]$. Set also
\[
\lambda_j^\ast = \chi_j(h_j^\ast)\in \bk^\times,
\hs j \in [1, n].
\]
If $L =\{l, s(l), \ldots, s^m(l)\}$ 
is a level set associated to $R$ and $m \geq 1$, by \cite[Proposition 8.8]{GY:Poi-CGL},
\begin{equation}\label{eq:lam-L}
\lambda_l^\ast = \lambda_{s(l)}^\ast = \cdots = \lambda_{s^{m-1}(l)}^\ast = -\lambda_{s(l)} = -\lambda_{s^2(l)} = \cdots
=-\lambda_{s^m(l)}.
\end{equation}
For each level set $L$ associated to $R$, we set  
\begin{equation}\label{eq:lam-L-1}
\lambda_{\sL} = \lambda_{{\rm max}(\sL)} \in \bk^\times.
\end{equation}
Then for every level set $L$ with at least two elements, we have 
\begin{equation}\label{eq:Lam-max-min}
\lambda_\sL = \lambda_{{\rm max}(\sL)}=- \lambda^\ast_{{\rm min}(\sL)}.
\end{equation}
Recall the matrices
$E$, $\bnu$, and  $\Lambda \in {\rm Mat}_{n \times {\rm ex}}$, respectively given in 
\eqref{eq:E-Et}, \eqref{eq:bnu} and \eqref{eq:Lambda}, and recall from 
\autoref{thm:M-1} that $M = E^t \bnu^{-1} E \Lambda$. 
Introduce the diagonal matrix 
\begin{equation}\label{eq:overline-Lambda}
\overline{\Lambda} = (\overline{\Lambda}_{i, j})\in {\rm Mat}_{n \times n}({\bf k}) \hs
\mbox{with} \hs \overline{\Lambda}_{j, j} = \lambda_{\sL(j)}, \hs j \in [1, n].
\end{equation}
By \eqref{eq:lam-L}, we have 
\begin{equation}\label{eq:E-Lambda}
(\overline{\Lambda})_{n \times {\rm ex}} = \Lambda \hs \mbox{and} \hs
E\Lambda = (\overline{\Lambda} E)_{n \times {\rm ex}}
=\overline{\Lambda}\, E_{n \times ex}.
\end{equation}
Introduce also the matrices
\begin{equation}\label{eq:def-Q-calH}
Q = \bnu^{-1} \overline{\Lambda} \in {\rm Mat}_{n \times n}(\bk) \hs \mbox{and} \hs
\Theta = QE_{n \times {\rm ex}} \in {\rm Mat}_{n \times n}(\bk).
\end{equation}
With $\overline{\Lambda}$, $Q$, and $\Theta$ thus defined, we have 
\begin{equation}\label{eq:M-EQE}
M =  E^t \bnu^{-1} \overline{\Lambda} \,E_{n \times {\rm ex}} = 
E^t Q E_{n \times {\rm ex}} = E^t \Theta.
\end{equation}
In particular, $\Theta = F^t M \in {\rm Mat}_{n \times n}(\ZZ)$. We will come back to the matrix
$\Theta$ in $\S$\ref{ss:bth}.

\subsection{The interval variables and some almost cluster mutation relations}\label{ss:interval-y}
For $1 \leq j < k \leq n$, set 
$R_{[j, k]} = \bk[x_j, \ldots, x_k] \subset R$.
The symmetric condition on $R$ implies that each $R_{[j, k]}$ is a 
$\TT$-invariant Poisson sub-algebra of $R$ with respect to $\{\, , \, \}$.
Set $\delta_n^\ast = 0$, and for $j \in [1, n-1]$, set 
$\delta_j^\ast \in {\rm Der}_\bk (R_{[j+1, n]})$ by 
\[
\delta_j^\ast (x_k) = -\delta_k(x_j) \in R_{[j+1, k-1]}\subset R_{[j+1, n]}, \hs k \in [j+1, n].
\]
Then each $R_{[j, k]}$ is a 
$\TT$-Poisson CGL extension in both the ordered set $(x_j, \ldots, x_k)$ and the ordered set 
$(x_k, \ldots, x_j)$. Applying \autoref{thm:5.5} to these two presentations of $R_{[j, k]}$
as $\TT$-Poisson CGL extensions, one has the following result proved in 
\cite[$\S$8]{GY:Poi-CGL}.

\begin{lemma}\label{le-y-isi} \cite[Theorem 8.1]{GY:Poi-CGL} 
For any $i \in [1, n]$ and $m \in \ZZ_{\geq 0}$ such that $s^m(i) \neq +\infty$,
the Poisson algebra $R_{[i, s^m(i)]}$ has a unique homogeneous Poisson prime element, denoted by $y_{[i, s^m(i)]}$,
which is not in $R_{[i+1, s^m(i)]}$ nor in $R_{[i, s^m(i)-1]}$ and is inductively determined by
$y_\emptyset = 1$,  
$y_{[i, i]}= x_i$, $y_{[s^m(i), s^m(i)]}= x_{s^m(i)}$, and 
\begin{equation}\label{eq:y-i-c-c}
y_{[i, s^m(i)]} = y_{[i, s^{m-1}(i)]}x_{s^m(i)} - c_{[i, s^m(i)]} = x_i y_{[s(i), s^m(i)]} 
- c^\ast_{[i, s^m(i)]},
\end{equation}
where\footnote{In \cite[$\S$8.1]{GY:Poi-CGL}, the elements $c_{[i, s^m(i)]}$ and 
$c^\ast_{[i, s^m(i)]}$
are respectively denoted as $c_{[i, s^m(i)-1]}$ and $c'_{[i+1, s^m(i)]}$.} 
$c_{[i, s^m(i)]} \in R_{[i, s^m(i)-1]}$ and $c^\ast_{[i, s^m(i)]} \in R_{[i+1, s^m(i)]}$ and are respectively given by
\begin{equation}\label{eq:cc}
c_{[i, s^m(i)]} =\frac{\delta_{s^m(i)}(y_{[i, s^{m-1}(i)]})}{\lambda_{s^m(i)}} 
\hs \mbox{and} \hs
c^\ast_{[i, s^m(i)]} = \frac{\delta_{i}^\ast(y_{[s(i), s^{m}(i)]})}{\lambda_i^\ast}. 
\end{equation}
Moreover, recalling from 
$\S$\ref{ss:two-gradings} the definition of ${\rm lt}_{\bf x}(b)$ for $b \in R\backslash\{0\}$, one has
\begin{equation}\label{eq:lt-x-y-interval}
{\rm lt}_{\bf x}y_{[i, s^m(i)]} = x_i x_{s(i)} \cdots x_{s^m(i)} = \prod_{j \in L(i) \cap [i, s^m(i)]} x_j.
\end{equation}
\end{lemma}


Let $i \in [1, n]$ and $m \in \ZZ_{>0}$ such that $s^m(i) \in [1, n]$. Let
\[
J_{[i, s^m(i)]} = \{j \in [i, s^m(i)]\backslash L(i): s(j) > s^m(i)\},
\]
and for $j \in J_{[i, s^m(i)]}$, let
$j^{\rm min}([i, s^m(i)]) ={\rm min}\{L(j) \cap [i, s^m(i)]\}$. By \autoref{thm:5.5}, 
\[
\{{\bf k}^\times y_{[j^{\rm min}([i, s^m(i)]), j]}:\; j \in J_{[i, s^m(i)]}\}
\]
is the set of all homogeneous Poisson prime 
element of $R_{[i, s^m(i)]}$ that are not scalar multiples of $y_{[i, s^m(i)]}$
(see also \cite[(8.20)]{GY:Poi-CGL}). The next \autoref{lm:mut-yy} is proved in 
\cite[Corollary 8.11]{GY:Poi-CGL}, and the identity \eqref{eq:yy-y} is refereed to in
\cite[$\S$7.2]{GY:Poi-CGL} as an {\it almost cluster mutation
relation}.

\begin{lemma}\label{lm:mut-yy}
For any $i \in [1, n]$ and $m \in \ZZ_{>0}$ such that $s^m(i) \in [1, n]$, there exist 
$\zeta_{[i, s^m(i)]} \in \bk^\times$ and $n_j \in \ZZ_{\geq 0}$ for each $j \in J_{[i, s^m(i)]}$ such that
\begin{equation}\label{eq:yy-y}
y_{[i, s^{m-1}(i)]}y_{[s(i), s^m(i)]} = y_{[s(i), s^{m-1}(i)]}y_{[i, s^m(i)]} + 
\zeta_{[i, s^m(i)]}\prod_{j \in J_{[i, s^m(i)]}} 
y_{[j^{\rm min}([i, s^m(i)]), j]}^{n_j}.
\end{equation}
\end{lemma}

For $i \in [1, n]$ and $m \in \ZZ_{>0}$ such that $s^m(i) \in [1, n]$, set, as in
\cite[Corollary 8.11]{GY:Poi-CGL}, 
\begin{align}\label{eq:de-uii-0}
u_{[i, s^m(i)]}& = y_{[i, s^{m-1}(i)]}y_{[s(i), s^m(i)]} - y_{[s(i), s^{m-1}(i)]}y_{[i, s^m(i)]}\\
\label{eq:de-uii-1}
& =\zeta_{[i, s^m(i)]}\prod_{j \in J_{[i, s^m(i)]}} 
y_{[j^{\rm min}([i, s^m(i)]), j]}^{n_j}
\in R_{[i, s^m(i)]}.
\end{align}
Setting $c_{[s(i), s(i)]} = 0$, note that \eqref{eq:y-i-c-c} gives
\begin{align*}
&y_{[s(i), s^m(i)]}  = y_{[s(i), s^{m-1}(i)]}x_{s^m(i)} - c_{[s(i), s^m(i)]}\hs \mbox{and} \hs
y_{[i, s^m(i)]} = y_{[i, s^{m-1}(i)]}x_{s^m(i)}-c_{[i, s^m(i)]}.
\end{align*}
It follows that one also has
\begin{equation}\label{eq:u-yy-cc}
u_{[i, s^m(i)]}=y_{[s(i), s^{m-1}(i)]}c_{[i, s^m(i)]} - y_{[i, s^{m-1}(i)]}c_{[s(i), s^m(i)]}. 
\end{equation}
The next \autoref{le-cor-811} is part of \cite[Corollary 8.11]{GY:Poi-CGL} and follows from 
\eqref{eq:lt-x-y-interval}.

\begin{lemma}\label{le-cor-811}
For $i \in [1, n]$ and $m \in \ZZ_{\geq 0}$ such that $s^m(i) 
\in [1, n]$, writing
\begin{equation}\label{eq:uii-lt-x}
{\rm lt}_{\bf x}(u_{[i, s^m(i)]}) = \zeta_{[i, s^m(i)]} {\bf x}^{f_{[i, s^m(i)]}},
\end{equation}
where $\zeta_{[i, s^m(i)]} \in {\bf k}\backslash\{0\}$ and
$f_{[i, s^m(i)]} = (0, \ldots, 0, f_{i+1}, \ldots, f_{s^m(i)-1}, 0, \ldots, 0)^t \in \ZZ^n_{\geq 0}$,
 one has $f_j = f_{j'}$ for all $j, j' \in  [i+1, s^m(i)-1]$ that are on the same level, and
$f_j =0$ for all $j \in L(i)$.
\end{lemma}

\subsection{The matrix $M$ via ${\bf x}$ and ${\bf y}$-degrees and the Poisson bracket 
$\{\, , \}$}\label{ss:M-deg-csk}
We continue to assume that $R = (\bk[x_1, \ldots, x_n], \{\, , \, \})$ is a symmetric 
$\TT$-Poisson CGL extension.
Let the notation be as in $\S$\ref{ss:interval-y}.
Recall that ${\rm ex} = \{j \in [1, n]: s(j) \neq +\infty\}$. 
Let $j \in{\bf ex}$ and write 
\[
\{x_j, x_{s(j)}\} = \{x_j, x_{s(j)}\}_{\rm log-can} + 
\{x_j, x_{s(j)}\}_{\rm tail},
\]
where $\{x_j, x_{s(j)}\}_{\rm log-can} = -\chi_j(h_{s(j)}) x_jx_{s(j)}$ and
$\{x_j, x_{s(j)}\}_{\rm tail}  =-\delta_{s(j)}(x_j)$.
By \eqref{eq:cc} and \eqref{eq:u-yy-cc},
\begin{equation}\label{eq:u-j-sj-c}
u_{[j, s(j)]} = c_{[j,s(j)]}  =\frac{\delta_{s(j)}(x_j)}{\lambda_{s(j)}} = 
-\frac{1}{\lambda_{s(j)}}\{x_j, x_{s(j)}\}_{\rm tail}.
\end{equation}
In particular, $\{x_j, x_{s(j)}\}_{\rm tail}\neq 0$.
Let $b_{s(j)} \in \calT^\prime_{s(j)-1}$ be as in \eqref{eq:csj-bsj}.

\begin{lemma}\label{le:deg-csj}
For any symmetric $\TT$-Poisson CGL extension $R$ and for $j \in {\rm ex}$, one has
\[
{\rm lt}_{\bf x}(c_{s(j)}) = -\frac{{\bf x}^{\bar{e}_{p(j)}}}{\lambda_{s(j)}}{\rm lt}_{\bf x} \{x_j, x_{s(j)}\}_{\rm tail},  \hs \mbox{and} \hs 
{\rm lt}_{\bf y} (c_{s(j)}) = b_{s(j)},
\]
where  $\overline{e}_j\in\ZZ^n$ for $j \in [1, n]$ is defined in \eqref{eq:ee}.
\end{lemma}

\begin{proof}
Let $j \in {\rm ex}$. Let 
$j_0 = p^{o_-(j)}(j)$ so that $p(j_0) = -\infty$. Setting $i = j_0$ and $m \in \ZZ_{>0}$
such that $s^{m-1}(j_0) = j$ in \eqref{eq:u-yy-cc}, one has 
\[ 
u_{[j_0, s(j)]} = y_{[s(j_0), j]} c_{s(j)} - y_j c_{[s(j_0), s(j)]}. 
\]
By \autoref{le-cor-811}, ${\rm lt}_{\bf x}(u_{[j_0, s(j)]})$ does not contain any non-zero power of $x_j$.
By \eqref{eq:lt-x-y-interval}, 
\[
{\rm lt}_{\bf x}(y_{[s(j_0)), j]}) = x_{s(j_0)}x_{s^2(j_0)}\cdots x_j \hs
\mbox{and} \hs {\rm lt}_{\bf x}(y_{j})=x_{j_0}x_{s(j_0)}\cdots x_j
\]
We must then have ${\rm lt}_{\bf x}(y_{[s(j_0), j]} c_{s(j)})={\rm lt}_{\bf x}(y_j c_{[s(j_0), s(j)]})$, and 
\[ 
{\rm lt}_{\bf x}(c_{s(j)})=x_{j_0}{\rm lt}_{\bf x}(c_{[s(j_0), s(j)]}).
 \]
By induction, ${\rm lt}_{\bf x}(c_{s(j)})={\bf x}^{\bar{e}_{p(j)}}{\rm lt}_{\bf x}(c_{[j, s(j)]})$, and 
by \eqref{eq:u-j-sj-c}, 
\[
{\rm lt}_{\bf x}(c_{s(j)})={\bf x}^{\bar{e}_{p(j)}}{\rm lt}_{\bf x}(c_{[j, s(j)]})
= -\frac{1}{\lambda_{s(j)}}{\bf x}^{\bar{e}_{p(j)}}{\rm lt}_{\bf x}(\{x_j, x_{s(j)}\}_{\rm tail}).
\]
By \autoref{le:leading} and using the notation in \autoref{le-cor-811}, one has
\begin{equation}\label{eq:lt-y-csj-1}
{\rm lt}_{\bf y} (c_{s(j)})= \zeta_{[j, s(j)]} {\bf y}^{E^t(\bar{e}_{p(j)} + f_{[j, s(j)]})}
=\zeta_{[j, s(j)]} {\bf y}^{e_{p(j)} + E^tf_{[j, s(j)]}}.
\end{equation}
In particular, ${\rm lt}_{\bf y} (c_{s(j)})$ contains no power of $y_j$. Thus 
${\rm lt}_{\bf y} (c_{s(j)}) = b_{s(j)}$.
\end{proof}

\begin{remark}\label{rk:GY-normal-1} ({\bf Goodearl-Yakimov normality})
The symmetric $\T$-Poisson CGL extension $R$ is said to be {\it normal} 
in \cite[$\S$9.2]{GY:Poi-CGL} if
$\zeta_{[j, s(j)]} = 1$ for 
every $j \in {\bf ex}$. By \autoref{le:deg-csj}, \eqref{eq:iota-bsj}, and 
 \eqref{eq:lt-y-csj-1}, one has
\begin{equation}\label{eq:iota-zeta}
\iota_{s(j)} = \zeta_{[j, s(j)]}, \hs \forall \;\; j \in {\rm ex}.
\end{equation}
Thus $R$ is normal in the sense of 
\cite[$\S$9.2]{GY:Poi-CGL} if and only if it normal in the sense of   \autoref{de:normal}.
By {\rm \eqref{eq:u-j-sj-c}, for $j \in {\rm ex}$ one also has
\begin{equation}\label{eq:lt-x-xjj}
{\rm lt}_{\bf x} \{x_j, x_{s(j)}\}_{\rm tail} = -\lambda_{s(j)}
{\rm lt}_{\bf x} u_{[j, s(j)]} = -\lambda_{s(j)} \zeta_{[j, s(j)]} 
{\bf x}^{f_{[j, s(j)]}}.
\end{equation}
Thus normality of $R$ is, in turn, equivalent to the ${\bf x}$-leading coefficient (see 
notation in $\S$\ref{ss:two-gradings}) 
of $\{x_j, x_{s(j)}\}_{\rm tail}$ equal to $-\lambda_{s(j)}$ for every $j \in {\rm ex}$.
\hfill $\diamond$
}
\end{remark}

We can now give the two alternative descriptions of the columns of the matrix $M$.

\begin{theorem}\label{thm:sym-M-1}
For a symmetric $\T$-Poisson CGL extension $R = (\bk[x_1, \ldots, x_n], \{\, , \, \})$  and for 
$j \in {\rm ex}$, the $j$'th column $M_j$ of the matrix $M \in {\rm Mat}_{n \times {\rm ex}}(\ZZ)$
in \autoref{thm:M-1}
is given by 
\begin{equation}\label{eq:Mj-csj}
M_j =-e_{s(j)} +  {\rm deg}_{\bf y}(c_{s(j)}) = \deg_{\bf y}\left(\frac{c_{s(j)}}{y_{s(j)}}\right).
\end{equation}
Alternatively, with the matrix $E \in {\rm Mat}_{n \times n}(\ZZ)$ given in \eqref{eq:E-Et} and for 
$j \in {\rm ex}$, one has
\begin{equation}\label{eq:Mj-xjj}
M_j = E^t\left(-e_j - e_{s(j)} + {\rm deg}_{\bf x} \left(\{x_j, x_{s(j)}\}_{\rm tail}\right)\right).
\end{equation}
\end{theorem}

\begin{proof}
Let $j \in{\rm ex}$. By 
 \autoref{thm:M-1}, $b_{s(j)} = \iota_{s(j)} {\bf y}^{e_{s(j)} + M_j}$.  
 By \autoref{le:deg-csj}, ${\rm deg}_{\bf y}
c_{s(j)} = \deg_{\bf y} b_{s(j)}$. Thus \eqref{eq:Mj-csj} holds. 
By  \eqref{eq:lt-y-csj-1}, 
${\rm deg}_{\bf y} c_{s(k)} = e_{p(k)} + E^tf_{[k, s(k)]}$. 
Thus \eqref{eq:Mj-xjj} holds.
\end{proof}

\begin{remark}\label{rk:quantum}
{\rm
The special case of \eqref{eq:Mj-xjj} for $M_1$ when $s(1) = n$ is given in \cite[$\S$8.10]{GY:Quantum-CGL}.
\hfill $\diamond$
}
\end{remark}

\subsection{Cartan integers via Poisson cohomology of the log-canonical term}\label{ss:bth}
We continue with the notation in $\S$\ref{ss:M-deg-csk}, and recall from \eqref{eq:M-EQE} that 
$M = E^t \Theta$, where 
\[
\Theta = F^t M=(\bnu^{-1} \overline{\Lambda}) E_{n \times {\rm ex}} = 
QE_{n \times {\rm ex}} \in 
{\rm Mat}_{n \times {\rm ex}}(\ZZ).
\]
For $j \in {\rm ex}$, denote the $j^{\rm th}$-column of $\Theta$ by 
\begin{equation}\label{eq:QE-bth}
\bth^{(j, s(j))} =\Theta e_j = QEe_j \in \ZZ^n.
\end{equation}
By \eqref{eq:Mj-xjj}, one then has
\begin{equation}\label{eq:bth-j}
\bth^{(j, s(j))} = -e_j - e_{s(j)} + {\rm deg}_{\bf x} \left(\{x_j, x_{s(j)}\}_{\rm tail}\right),
\hs j \in {\rm ex}.
\end{equation}
The fact that the columns of $\Theta$ are given as in \eqref{eq:bth-j}
has also been proved in \cite[$\S$4.6]{Lu-Mykola:deformation}, where
a classification of symmetric Poisson CGL extensions is given, and the vectors 
$\bth^{(j, s(j))}$ are interpreted in terms of Poisson cohomology.
To review the relevant results from 
\cite[$\S$4.6]{Lu-Mykola:deformation}, for a given symmetric $\TT$-Poisson CGL extension
$R = (\bk[x_1, \ldots, x_n], \{\, , \, \})$, consider 
the 
Poisson bi-vector field
\[
\pi = \sum_{1 \leq j < k \leq n} \{x_j, x_k\} \frac{\partial}{\partial x_j}
\wedge \frac{\partial}{\partial x_k}
\]
on $\bk^n$ defining $\{\, , \, \}$. 
Write $\pi = \pi_0 + \pi_{\rm tail}$, where 
\begin{equation}\label{eq:pi-0-tail}
\pi_0 = -\sum_{1 \leq j < k \leq n} \chi_j(h_k) x_jx_k \frac{\partial}{\partial x_j}
\wedge \frac{\partial}{\partial x_k} \hs \mbox{\rm and}\hs
\pi_{\rm tail} = -\sum_{1 \leq j < k \leq n} \delta_k(x_j) 
\frac{\partial}{\partial x_j}\wedge \frac{\partial}{\partial x_k}
\end{equation}
are respectively the {\it log-canonical term} and the {\it tail term} of $\pi$. 
Consider the standard action of the torus $(\bk^\times)^n$ on $\bk^n$ and the induced 
$(\bk^\times)^n$-action on the space $\mfX^2(\bk^n)$ of all the bi-vector fields on $\bk^n$. 
Identify the character lattice of $(\bk^\times)^n$ with $\ZZ^n$ (of column vectors). 
For $j \in {\rm ex}$, the vector $\bth^{(j, s(j))} \in \ZZ^n$ is then the $(\bk^\times)^n$-weight of
\[
V_{\bth^{(j, s(j))}} := {\bf x}^{{\rm deg}_{\bf x} \left(\{x_j, x_{s(j)}\}_{\rm tail}\right)}
\frac{\partial}{\partial x_j}\wedge \frac{\partial}{\partial x_{s(j)}}.
\]
Let $\calS_\pi = \{\bth^{(j, s(j))}: j \in {\rm ex}\}$ and let
\[
\pi_1 = -\sum_{j \in {\rm ex}}\lambda_{s(j)} \iota_{s(j)} V_{\bth^{(j, s(j))}}.
\]
By \eqref{eq:lt-x-xjj}, $\pi_1$ is a summand of $\pi_{\rm tail}$. Let $[\, , \, ]_{\rm Sch}$ be the
Schouten bracket on the space of poly-vector fields on $\bk^n$. 
By the classification result  on symmetric $\TT$-Poisson CGL extensions
stated in \cite[Theorem C]{Lu-Mykola:deformation}, the identity $\Theta  =
\bnu^{-1}\overline{\Lambda} E_{n \times {\rm ex}}$ 
implies that 
\[
[\pi_0, \;\pi_1]_{\rm Sch} = 0,
\]
and that 
$\pi_1$ defines a non-zero element in $\mH^2_{\pi_0}(\CC^n)^\TT$, the second $\TT$-invariant
Poisson cohomology space of $\pi_0$. Moreover, 
 $\pi$ is the unique
{\it $\calS_\pi$-admissible algebraic Poisson deformation of $\pi_0$ along $\pi_1$}, in the sense that 
$\pi$ is a unique finite sum
\[
\pi = \pi_0 + \pi_1 + \pi_2+ \cdots ,
\]
where for each $m \geq 1$, every monomial term in $\pi_m$ has 
a $(\bk^\times)^n$-weight that is a sum of exactly $m$ elements of $\calS_\pi$. 

Denote, as in  \cite{Lu-Mykola:deformation}, by $\calS(\pi_0) \subset \ZZ^n$ the set of all
non-zero $(\bk^\times)^n$-weights in $\mH^2_{\pi_0}(\CC^n)^\TT$. Then $\calS_\pi \subset \calS(\pi_0)$.
It is shown in \cite[$\S$4.3]{Lu-Mykola:deformation} that associated to $\pi_0$ one has  
the {\it oriented smoothing graph}
$\Gamma^+(\pi_0)$ of $pi_0$, whose vertex set is $[1, n]$ and whose
connected components are called {\it level sets in $[1, n]$ defined by $\pi_0$}. It is further
shown in \cite[$\S$4.3]{Lu-Mykola:deformation} that every $\bth \in \calS(\pi_0)$ is of the form
\begin{equation}\label{eq:bth-entries}
\bth = (0, \ldots, 0, -1, \bth_{j+1}, \ldots, \bth_{k-1}, -1, 0, \ldots, 0)^t \in \ZZ^n,
\end{equation}
for a unique pair $1 \leq j < k \leq n$ on the same level defined by $\pi_0$, with the two $-1$s 
as the positions $j$ and $k$, and that for every
$i \in [j+1, k-1]$ the integer $\bth_i$ is non-negative and depends only on the level of
$i$ and the level of $j$ and $k$ defined by $\pi_0$. Any  
 $-\bth_i$ appearing in \eqref{eq:bth-entries} for some $\bth \in \calS(\pi_0)$
is called a {\it Cartan integer associated to $\pi_0$} in \cite[$\S$4.3]{Lu-Mykola:deformation}.

Returning to the symmetric $\TT$-Poisson CGL extension
$R = (\bk[x_1, \ldots, x_n], \{\, , \, \})$, it is also shown in 
\cite[$\S$4.3]{Lu-Mykola:deformation} that for every $j \in [1, n]$, the level set $L(j)$ 
associated to $R$ defined 
in \autoref{nota-p-s} is contained in the level set\footnote{What is denoted as $L(j)$ here for the 
level set of $j \in [1, n]$ associated to $R$ is denoted as $L_\pi(j)$ in 
\cite{Lu-Mykola:deformation}, while the level set of $j$ defined by $\pi_0$ is denoted as $L(j)$ in 
\cite{Lu-Mykola:deformation}} of 
$j \in [1, n]$ defined by $\pi_0$. For $j \in {\rm ex}$, 
writing
\[
\bth^{(j, s(j))} = (0, \ldots, 0, -1, \bth^{(j, s(j))}_{j+1}, \ldots, \bth^{(j, s(j))}_{s(j)-1}, -1, 0, \ldots, 0)^t
\in \ZZ^n,
\]
then the non-negative integer 
$\bth^{(j, s(j))}_i$ for every $i \in [j+1, s(j)-1]$ depends only on the 
level sets $L(i)$ and $L(j)=L(k)$ associated to $R$, and  we write 
 \begin{equation}\label{eq:a-LL}
 a_{\sL(i), \sL(j)}:=-\bth^{(j, s(j))}_i, \hs i \in [j+1, s(k)-1].
 \end{equation}
While a {Cartan integer} associated to $\pi_0$, we also call $a_{\sL(i), \sL(j)}$ 
in \eqref{eq:a-LL} a
{\rm Cartan integer associated to $R$}.
 Consequently, for every $j \in {\rm ex}$, we have
 \begin{equation}\label{eq:bth-Cartan}
 \bth^{(j, s(j))} =
 (0, \,\ldots, \,0, \,-1, \, -a_{\sL(j+1), \sL(j)}, \, \ldots, \, 
 -a_{\sL(s(j)-1), \sL(j)}, \, -1, \, 0, \, \ldots, \, 0)^t \in \ZZ^n,
 \end{equation}
 where the two $-1$ entries are at positions  $j$ and $s(j)$. As a consequence (see
also \cite[$\S$4.6]{Lu-Mykola:deformation}, 
one has the following description of the matrix $M = ((M)_{i, j})_{i \in [1, n], j \in {\rm ex}}$:
 \begin{equation}\label{eq:M-jk}
M_{i, j}  = \begin{cases} 1, & \hs i = p(j)\neq -\infty, \\
-1, & \hs i = s(j), \\
a_{\sL(i), \sL(j)}, & \hs i < j < s(i) < s(j), \\
-a_{\sL(i), \sL(j)}, & \hs j < i < s(j) < s(i) \;\; (\mbox{including when} \; s(i) = +\infty),\\
0, & \hs \mbox{otherwise}.\end{cases}
\end{equation}


\subsection{Proper re-orderings of symmetric Poisson CGL extensions}\label{ss:reordering}
We continue to assume that 
$R = (\bk[x_1, \ldots, x_n], \{\, , \, \})$ is a symmetric Poisson CGL extension
as in \autoref{de:sym-CGL-intro},
and let the notation be as in $\S$\ref{ss:nota-sym}-$\S$\ref{ss:M-deg-csk}. 
As proved in \cite[$\S$6]{GY:Poi-CGL}, the symmetry property of $R$ gives rise to many other presentations
of the Poisson algebra $R$ as $\T$-Poisson CGL extensions, called {\it proper re-orderings of $R$}, which 
 we now recall. 

Following \cite[Definition 6.3]{GY:Poi-CGL}, let
$\Xi_n$ be the subset of $S_n$ consisting of all $\tau \in S_n$ such that for each $j \in [1, n]$,
$\tau([1, j]) \subset [1, n]$ is a sub-interval of $[1, n]$, i.e.
\[
\tau(j) = 1 + {\rm max} \;\tau([1, j-1]) \hs \mbox{or} \hs \tau(j) = -1+{\rm min}\;\tau([1, j-1]), \hs \forall \, j \in [2, n].
\]
For $\tau \in \Xi_n$, set 
\begin{align}
\label{eq:tau-plus}
\tau(+) &= \{j \in [2, n]: \, \tau(j) = 1 + {\rm max} \;\tau([1, j-1])\}=\{j \in [2, n]: \tau(j)>\tau(1)\}, \\
\label{eq:tau-minus}
\tau(-) &= \{j \in [2, n]: \, \tau(j) = -1+{\rm min}\;\tau([1, j-1])\} =\{j \in [2, n]: \tau(j)<\tau(1)\}.
\end{align}
For any $1 \leq j < k \leq n$, one then has
\[
 \{x_{\tau(j)}, \, x_{\tau(k)}\} =\begin{cases} -\chi_{\tau(j)}(h_{\tau(k)}) x_{\tau(j)}x_{\tau(k)}-
 \delta_{\tau(k)}(x_{\tau(j)}), & \hs k \in \tau(+),\\
 -\chi_{\tau(j)}(h_{\tau(k)}^\ast) x_{\tau(j)}x_{\tau(k)}-
 \delta_{\tau(k)}^\ast(x_{\tau(j)}), & \hs k \in \tau(-).
 \end{cases}
\]
Set  $\delta_{\tau, 1}=0$ and $h_{\tau, 1} = h_{\tau(1)}$, and for $k \in [2, n]$, set
\begin{align*}
&h_{\tau, k} = h_{\tau(k)}, \hs  \delta_{\tau, k} = \delta_{\tau(k)}, \hs
\mbox{if} \;\;\; k \in \tau(+),\\
&h_{\tau, k} = h_{\tau(k)}^\ast, \hs  \delta_{\tau, k} = \delta_{\tau(k)}^\ast, \hs
\mbox{if} \;\;\; k \in \tau(-).
\end{align*}
Note then that  $\lambda_{\tau, j} := \chi_{\tau(j)}(h_{\tau, j}) \in \bk^\times$ for all $j \in [1, n]$, 
where
$\lambda_{\tau, 1} = \lambda_{\tau(1)}$ and for $j \in [2, n]$,
\begin{equation}\label{eq:lam-tau-j}
\lambda_{\tau, j} = \lambda_{\tau(j)} \;\;\; \mbox{if} \;\;\; j \in \tau(+) \hs \mbox{and} \hs
\lambda_{\tau, j} = \lambda_{\tau(j)}^\ast \;\;\; \mbox{if} \;\;\; j \in \tau(-).
\end{equation}
 For each $\tau \in \Xi_n$, it now follows from the definitions 
 (as  proved in \cite[Proposition 6.4]{GY:Poi-CGL})
that the Poisson algebra $R$ is a $\TT$-Poisson CGL extension in the ordered set 
$(x_{\tau(1)}, \ldots, x_{\tau(n)})$ of polynomial generators, which we denote as
\begin{equation}\label{eq:R-tau}
R_\tau = (\bk[x_{\tau(1)}, \ldots, x_{\tau(n)}], \, \{\, , \, \})_{(\chi_{\tau(1)}, \ldots, \chi_{\tau(n)}, h_{\tau, 1}, \ldots, h_{\tau, n})}.
\end{equation}

\begin{definition-notation}\label{de:proper-reordering}
{\rm For $\tau \in \Xi_n$,
the $\TT$-Poisson CGL extension $R_\tau$ in \eqref{eq:R-tau} is called  the {\it proper re-ordering} of 
$R$ by $\tau$. Let
\[
{\bf y}^\prime_\tau = (y^\prime_{\tau, 1}, \ldots, y^\prime_{\tau, n})
\]
be the sequence of 
homogeneous Poisson prime elements
associated to $R_\tau$ by \autoref{thm:5.5}.
}
\end{definition-notation}

In view of the notation in \eqref{eq:R-tau}, we write $R_{\rm id}$,  ${\rm id}\in \Xi$ being
the identity element,
when we want to 
regard $R$ as a $\TT$-Poisson CGL extension in the original CGL generators $(x_1, \ldots, x_n)$.

Recall from $\S$\ref{ss:interval-y} the interval variables associated to $R$ as a symmetric $\TT$-Poisson CGL extension. 
Recall also that $p$ and $s$ are the respective predecessor and successor maps for $R_{\rm id}$.
The following description of ${\bf y}^\prime_\tau$ for $\tau \in \Xi_n$ is 
proved in \cite[Theorem 8.3]{GY:Poi-CGL}. 

\begin{proposition}\label{prop-y-tau}
For any $\tau \in \Xi_n$, one has $y_{\tau, 1}^\prime = x_{\tau(1)}$, and for $j \in [2, n]$,

(1) if $j \in \tau(+)$, then $y^\prime_{\tau, j} = y_{[p^m(\tau(j)),\, \tau(j)]}$, where
\[
m = {\rm max} \, \{m' \in \ZZ_{\geq 0}: p^{m'}(\tau(j)) \in \tau([1, j])\} = |L(\tau(j)) \cap \tau([1, j])|-1;
\]

(2) if $j \in \tau(-)$, then $y^\prime_{\tau, j} = y_{[\tau(j), \,s^m(\tau(j))]}$, where
\[
m = {\rm max} \, \{m' \in \ZZ_{\geq 0}: s^{m'}(\tau(j)) \in \tau([1, j])\}=|L(\tau(j)) \cap \tau([1, j])|-1.
\]
\end{proposition}

\subsection{The Goodearl-Yakimov theorem on  $(\bfy_\tau, M_\tau)$}\label{ss:y-tau-M-tau}
We continue with the notation from $\S$\ref{ss:reordering}. 
Let  $\calL$ be the collection of all level sets of $R_{\rm id}$, as recalled in $\S$\ref{ss:nota-sym}.
Let $\tau \in \Xi_n$.
By \cite[Corollary 8.6(b)]{GY:Poi-CGL}, the level sets of the $\TT$-Poisson CGL extension $R_\tau$ 
are precisely $\tau^{-1}(L)$, where $L \in \calL$. Using the disjoint unions
\[
[1, n] = \bigsqcup_{L \in {\calL}} L =  \bigsqcup_{L \in {\calL}}\tau^{-1}(L),
\]
the element $\tau_\bullet$ in the permutation group $S_n$ is defined in 
\cite[$\S$10.2]{GY:Poi-CGL} as the unique one 
such that
for each $L \in \calL$, $\tau_\bullet (L) = L$ and
\begin{equation}\label{eq:tau-dot-tau}
(\tau_\bullet \tau)|_{\tau^{-1}(L)}: \;\; \tau^{-1}(L) \longrightarrow L
\end{equation}
is the unique order preserving bijection from $\tau^{-1}(L)$ to $L$. In other words, if
\begin{equation}\label{eq:L-list}
L = \{k_1, k_2, \ldots, k_l\} \hs \mbox{and} \hs \tau^{-1}(L) = \{j_1, j_2, \ldots, j_l\}
\end{equation}
with $k_1 < \cdots < k_a$ and $j_1 < \cdots < j_a$, then  $(\tau_\bullet \tau)(j_a) = k_a$ for each $a \in [1, l]$.

\begin{notation}\label{nota:y-tau-0}
{\rm For $\tau \in \Xi_n$,  set\footnote{What are denoted as ${\bf y}_\tau$ and ${\bf q}_\tau$ 
here are respectively
denoted as $\tilde{{\bf y}}_\tau$ and ${\bf r}_\tau$ in \cite[$\S$11]{GY:Poi-CGL}.}
\begin{equation}\label{eq:by-tau}
{\bf y}_\tau =(y_{\tau, 1}, \, \ldots, \, y_{\tau, n}) = \left(y_{\tau, (\tau_\bullet \tau)^{-1}(1)}^\prime, \; y_{\tau, (\tau_\bullet \tau)^{-1}(2)}^\prime, \; \ldots, \, 
y_{\tau, (\tau_\bullet \tau)^{-1}(n)}^\prime\right).
\end{equation}
Let ${\bf q}_\tau \in {\rm Mat}_{n \times n}({\bf k})$ be the Poisson coefficient matrix  of the 
sequence ${\bf y}_\tau$, and let 
\begin{equation}\label{eq:chi-by-tau}
\chi_{{\bf y}_\tau} = (\chi_{y_{\tau, 1}}, \, \ldots, \, \chi_{y_{\tau, n}}) \in X(\TT)^n,
\end{equation}
where  $\chi_{y_{\tau, k}} \in X(\T)$, for $k \in [1, n]$, is the $\T$-weight of $y_{\tau, k}$.
Note that $\bfy = \bfy_{\rm id}$.
\hfill $\diamond$
}
\end{notation}

Recall that ${\rm ex} = \{j \in [1, n]: s(j) =+\infty\}$, and recall from \eqref{eq:Lambda} the diagonal
$\Lambda =(\Lambda_{j, k}) \in {\rm Mat}_{n \times {\rm ex}}({\bf k})$ with $\Lambda_{j,j} = \lambda_{s(j)}=
\lambda_{\sL(j)}$.
We now recall a main part of \cite[Theorem 11.1]{GY:Poi-CGL}.

\begin{theorem}\label{thm:GY-Mtau} \cite[Theorem 11.1]{GY:Poi-CGL}
Let $R$ be any symmetric Poisson CGL extension of length $n$ as in \autoref{de:sym-CGL-intro}.
For each $\tau \in \Xi_n$, there is a unique $M_\tau \in {\rm Mat}_{n \times {\rm ex}}(\ZZ)$ satisfying
\begin{equation}\label{eq:for-M-tau}
{\bf q}_\tau M_\tau = -\Lambda \hs \mbox{and} \hs \chi_{{\bf y}_\tau}M_{\tau} = 0.
\end{equation}
Let $M = M_{\rm id}$. Assume furthermore that\footnote{By \autoref{rk:pre-seed}, \eqref{eq:QQ-pos} is
equivalent to the assumption in \cite[(11.5)]{GY:Poi-CGL} that there exist
positive integers $\{d_{\sL(j)}: j \in {\rm ex}\}$ such that
$d_{\sL(j)} \lambda_{\sL(k)} = d_{\sL(k)} \lambda_{\sL(j)}$ for all $j, k \in {\rm ex}$.}

\begin{equation}\label{eq:QQ-pos}
\frac{\lambda_{\sL(j)}}{\lambda_{\sL(k)}} \in \QQ_{>0}, \hs \forall \;\; j, k \in {\rm ex},
\end{equation}
and assume that 
$R$ is  normal in the sense of \autoref{de:normal} 
(equivalently in the sense of \cite[$\S$9.2]{GY:Poi-CGL} by
\autoref{rk:GY-normal-1}). Then
$\{({\bf y}_\tau, M_\tau): \tau \in \Xi_n\}$ is a set of mutation equivalent  
$\TT$-Poisson seeds in
${\rm Frac}(R)$, and for every ${\rm inv}\subset [1, n]\backslash {\rm ex}$ one has
(see \autoref{de:cluster-A})
\begin{equation}\label{eq:A-R-inv}
\calA({\bf y}, M; {\rm inv}) = \calU({\bf y}, M; {\rm inv}).
\end{equation}
\end{theorem}


\begin{remark}\label{rk:GY-11.1}
{\rm In the terminology of \autoref{de:pre-seed}, without \eqref{eq:QQ-pos} the pair 
$(\bfy_\tau, M_\tau)$ for each $\tau \in \Xi_n$ is 
a $\TT$-Poisson pre-seed in ${\rm Frac}(R)$. By \autoref{rk:pre-seed}, the assumption in
\eqref{eq:QQ-pos} guarantees that $M_\tau$ is skew-symmetrizable,
so that $(\bfy_\tau, M_\tau)$ is a $\TT$-Poisson seed. 
\hfill $\diamond$
}
\end{remark}

As mentioned in the Introduction, the proof given in 
\cite{GY:Poi-CGL} of the existence and uniqueness of 
$M_\tau \in {\rm Mat}_{n \times {\rm ex}}(\ZZ)$ satisfying  \eqref{eq:for-M-tau} uses a rather involved 
induction process (see \cite[$\S$11.6]{GY:Poi-CGL} for detail), which, under the assumption that $R$ is normal, is also used to show that all the seeds 
$\{(\bfy_\tau, M_\tau):
\tau \in \Xi_n\}$ are mutation equivalent and that \eqref{eq:A-R-inv} holds.

In the next $\S$\ref{ss:explicit-M-tau},  by observing that the equations in \eqref{eq:for-M-tau} are equivalent to the GSV Equations  for the $\TT$-Poisson CGL extension $R_\tau$ 
for each $\tau \in \Xi_n$ (see \autoref{de:GSV-intro}), we give
a new proof of existence and uniqueness of $M_\tau$ as a consequence of \autoref{thm:M-1} applied to $R_\tau$.
By the same elementary 
linear algebra  arguments as that used in \autoref{:linear-alg-1}, we also 
give explicit formulas for $M_\tau$ as matrix products (\autoref{thm:M-3}) as well as
explicit description of the entries of $M_\tau$ in terms of Cartan integers 
associated to $R$  (\autoref{thm:M-tau-cases}).

\subsection{Explicit formulas for $M_\tau$}\label{ss:explicit-M-tau}
Continuing with the notation from $\S$\ref{ss:y-tau-M-tau}, 
we first prepare some more facts on the $\TT$-Poisson CGL extension $R_\tau$ for each $\tau \in \Xi_n$.

\begin{notation}\label{nota:p-s-tau}
{\rm 
For $\tau \in \Xi_n$, let 
\[
p_\tau: \;\; [1, n] \longrightarrow \{-\infty\} \cup [1, n-1]
\hs \mbox{and} \hs
s_\tau: \;\; [1, n] \longrightarrow [2, n] \cup \{+\infty\}
\]
be the respective predecessor and the successor maps
for $R_\tau$, 
and let 
\[
{\rm ex}_\tau = \{j \in [1, n]: s_\tau(j) \neq +\infty\}.
\]
Let
$E_\tau \in {\rm Mat}_{n \times n}(\ZZ)$ and
$F_\tau = (E_\tau)^{-1}$ be defined via \eqref{eq:E-Et} 
for the $\TT$-Poisson CGL extension $R_\tau$, i.e. (recall again $e_{+\infty} = +\infty$),
\begin{equation}\label{eq:E-tau-def}
E_\tau = (e_1 - e_{s_\tau(1)}, \; e_2 - e_{s_\tau(2)}, \; \ldots, \; e_n - e_{s_\tau(n)}),
\end{equation}
and let 
$\Lambda_\tau  \in {\rm Mat}_{n \times {\rm ex}_\tau}({\bf k})$ 
be the diagonal matrix defined via \eqref{eq:Lambda} for $R_\tau$, i.e., 
with $(j, j)$-entry $\lambda_{\tau, s_\tau(j)}$
 for $j \in {\rm ex}_\tau$. 
Let $\ep_\tau^\prime \in {\rm Mat}_{{\rm ex}_\tau \times {\rm ex}_\tau}(\ZZ)$ be the diagonal matrix 
whose $(j, j)$-entry $\ep_\tau^\prime(j)$, for
$j \in {\rm ex}_\tau$, is given by
\begin{equation}\label{eq:ep-tau-prime}
\ep_\tau^\prime(j) = \begin{cases} 1, & \hs \mbox{if} \;\; s_\tau(j) \in \tau(+),\\
-1, & \hs \mbox{if} \;\; s_\tau(j) \in \tau(-).
\end{cases}
\end{equation}
Set
$\tau (-\infty) = \tau_\bullet (-\infty) = -\infty$ and
$\tau (+\infty) = \tau_\bullet (+\infty) = +\infty$.
Let ${\rm id} \in \Xi_n \subset S_n$ be the identity element, so that $E = E_{\rm id}$ 
and $F = F_{\rm id}$ as given respectively in \eqref{eq:E-Et} and \eqref{eq:F}.
\hfill $\diamond$
}
\end{notation}

Recall from $\S$\ref{ss:nota-intro} that every  $\sigma \in S_n$ gives rise to the $n \times n$ 
matrix, also denoted by 
$\sigma$, via \eqref{eq:sigma-perm}, and that  $\sigma^{-1} = \sigma^t \in {\rm Mat}_{n \times n}(\ZZ)$.

\begin{lemma}\label{:p-s-E-Lambda-tau}
For any $\tau \in \Xi_n$, one has (see also \cite[Lemma 10.2]{GY:Poi-CGL})
\begin{align}\label{eq:ps-ex-tau}
&p_\tau = (\tau_\bullet \tau)^{-1} p (\tau_\bullet \tau), \hs
s_\tau = (\tau_\bullet \tau)^{-1} s (\tau_\bullet \tau), \hs
{\rm ex}_\tau =  (\tau_\bullet \tau)^{-1} ({\rm ex}),\\
\label{eq:E-Lambda-tau}
&E_\tau = (\tau_\bullet \tau)^{-1} E (\tau_\bullet \tau), \hs 
F_\tau = (\tau_\bullet \tau)^{-1} F (\tau_\bullet \tau), \hs 
\Lambda_\tau \ep_\tau^\prime = (\tau_\bullet \tau)^{-1} \Lambda (\tau_\bullet \tau)_{{\rm ex} \times {\rm ex}_\tau}.
\end{align}
Moreover, with the diagonal matrix $\overline{\Lambda} \in {\rm Mat}_{n \times n}(\bk)$ given in \eqref{eq:overline-Lambda}, one has
\begin{equation}\label{eq:tau-overline-Lambda}
\tau_\bullet^{-1} \,\overline{\Lambda} \,\tau_\bullet = \overline{\Lambda}.
\end{equation}
\end{lemma}

\begin{proof}
Let $L = \{k_1, \ldots, k_l\}$ be any level set associated to $R_{\rm id}$
with $k_1 < \cdots < k_a$, and let $\tau^{-1}(L) = \{j_1, \ldots, j_l\}$ with 
 $j_1 < \cdots < j_l$, so that
 $(\tau_\bullet \tau)(j_a) = k_a$ for each $a \in [1, l]$. 
 By definition,
\[
((\tau_\bullet \tau)^{-1} p (\tau_\bullet \tau))(j_1) = (\tau_\bullet \tau)^{-1} p(k_1))
= (\tau_\bullet \tau)^{-1} (-\infty) = -\infty = p_\tau(j_1),
\]
and for $a \in [2, l]$,
\[
((\tau_\bullet \tau)^{-1} p (\tau_\bullet \tau))(j_a) = (\tau_\bullet \tau)^{-1} p(k_a))
= (\tau_\bullet \tau)^{-1} (k_{a-1}) =j_{a-1} = p_\tau(j_a).
\]
This shows that $p_\tau = (\tau_\bullet \tau)^{-1} p (\tau_\bullet \tau)$. The formula for $s_\tau$ in 
\eqref{eq:ps-ex-tau} is proved similarly, and the formula for  ${\rm ex}_\tau$ in 
\eqref{eq:ps-ex-tau} follows from that for $s_\tau$.
The identities on $E_\tau$ and $F_\tau$  in \eqref{eq:E-Lambda-tau} also follow from their definitions
and the identify for 
$s_\tau$ in \eqref{eq:ps-ex-tau}.

To prove the identity on $\Lambda_\tau\ep_\tau^\prime$ in \eqref{eq:E-Lambda-tau}, fix 
$j \in {\rm ex}_\tau$ and let $\widetilde{\lambda}_j \in \bfk$ be  the $(j, j)$-entry of the diagonal matrix 
$\Lambda_\tau \ep_\tau^\prime \in {\rm Mat}_{n \times {\rm ex}_\tau}(\bfk)$. Let $k = (\tau_\bullet \tau)(j) \in 
{\rm ex}$, so that $s_\tau(j) = (\tau_\bullet \tau)^{-1}(s(k))$. 
By 
definition, 
$\widetilde{\lambda}_j = \ep_\tau^\prime(j) \lambda_{\tau, s_\tau(j)} = 
\ep_\tau^\prime(j) \lambda_{\tau, (\tau_\bullet \tau)^{-1}(s(k))}$.
As $s_\tau(j) > j \geq 1$, by
\eqref{eq:lam-tau-j}, one has
\begin{equation}\label{eq:lambda-j-prime}
\widetilde{\lambda}_j =
\begin{cases} \ep_\tau^\prime(j) \lambda_{\tau s_\tau(j)} = \lambda_{\tau_\bullet^{-1} (s(k))}
& \;\;\; s_\tau(j) \in \tau(+),\\
\ep_\tau^\prime(j) \lambda^\prime_{\tau s_\tau(j)} = -\lambda^\prime_{\tau_\bullet^{-1}(s(k))},
& \;\;\; s_\tau(j) \in \tau(-).\end{cases}
\end{equation}
Recall that $L(k) \subset [1, n]$ is the level set of $k$ associated to $R=R_{\rm id}$ (see \eqref{eq:Lk}), 
and that $\tau_\bullet (L(k)) = L(k)$.  
It follows from the definition of $\tau(\pm)$ that 
\begin{align}\label{eq:tau-jk-plus}
&s_\tau(j) \in \tau(+) \;\; \Longleftrightarrow \;\;
\tau s_\tau(j) > \tau(j) \;\; \Longleftrightarrow \;\;
\tau_\bullet^{-1}(s(k)) > \tau_\bullet^{-1}(k),\\
\label{eq:tau-jk-minus}
&s_\tau(j) \in \tau(-) \;\; \Longleftrightarrow \;\;
\tau s_\tau(j) < \tau(j) \;\; \Longleftrightarrow \;\;
\tau_\bullet^{-1}(s(k)) < \tau_\bullet^{-1}(k).
\end{align}
Consequently, if $s_\tau(j) \in \tau(+)$, then $\tau_\bullet^{-1}(s(k)) \neq {\rm min}(L(k))$ and, by
\eqref{eq:lam-L-1} and \eqref{eq:lam-L}, 
$\lambda_{\tau_\bullet^{-1} (s(k))} = \lambda_{\sL(k)} = \lambda_{s(k)}$; 
if $s_\tau(j) \in \tau(-)$, then $\tau_\bullet^{-1}(s(k)) \neq {\rm max}(L(k))$ and, again by
\eqref{eq:lam-L-1} and \eqref{eq:lam-L}, 
$-\lambda^\ast_{\tau_\bullet^{-1} (s(k))} =\lambda_{\sL(k)} = \lambda_{s(k)}$. 
It follows from \eqref{eq:lambda-j-prime} that
\[
\widetilde{\lambda}_j = \lambda_{s(k)}, \hs \forall \; j \in {\rm ex}_\tau \; \mbox{and}\; 
k = (\tau_\bullet \tau)(j) \in {\rm ex}.
\]
On the other hand, 
by definitions, 
$(\tau_\bullet \tau)^{-1} \Lambda (\tau_\bullet \tau)_{{\rm ex} \times {\rm ex}_\tau} \in 
{\rm Mat}_{n \times {\rm ex}_\tau}(\ZZ)$ is diagonal whose $(j, j)$-entry, for $j \in {\rm ex}_\tau$,  is $\lambda_{s(k)}$ with $k = (\tau_\bullet \tau)(j)$. 
Thus
$\Lambda_\tau \ep_\tau^\prime = (\tau_\bullet \tau)^{-1} \Lambda (\tau_\bullet \tau)_{{\rm ex} \times {\rm ex}_\tau}$.

Finally, recall that the 
$(j, j)$-entry of the diagonal matrix $\overline{\Lambda}$ is $\lambda_{L(j)}$ for each $j \in [1, n]$.
As $\tau_\bullet(L) = L$ for every level $L$, we have 
$\tau_\bullet^{-1} \,\overline{\Lambda} \,\tau_\bullet = \overline{\Lambda}$. Thus  \eqref{eq:tau-overline-Lambda} holds.
\end{proof}

For $\tau \in \Xi_n$ and  $k \in {\rm ex}$, set
\begin{equation}\label{eq:k-tau}
\kto = {\rm min}\{\tau_\bullet^{-1}(k), \;\tau_{\bullet}^{-1}(s(k))\} \hs \mbox{and} \hs
\ktt =  {\rm max}\{\tau_\bullet^{-1}(k),\; \tau_{\bullet}^{-1}(s(k))\}.
\end{equation}

\begin{lemma}\label{lm:K-tau-0}
For any $\tau \in \Xi_n$ and $i \in [1, n]\backslash {\rm ex}$, the $i^{\rm th}$-row of 
$E^{-1}\taubi E_{n \times {\rm ex}}$ is $0$.
\end{lemma}

\begin{proof}
Let $k \in {\rm ex}$. By the formulas for $E$ in \eqref{eq:E-Et} and $F = E^{-1}$ in \eqref{eq:F} we have
\[
(E^{-1}\taubi E)e_k = E^{-1}\taubi(e_k-e_{s(k)}) = F(e_{\taubi(k)}-e_{\taubi(s(k))})
=  \pm \sum_{l \in L(k) \cap [\kto, \ktt-1]} 
e_l.
\]
As $L(k) \cap [\kto, \ktt-1]\subset {\rm ex}$, 
the $i^{\rm th}$-row of 
$E^{-1}\taubi Ee_k$ is $0$ if $i \notin {\rm ex}$.
\end{proof}

\begin{notation}\label{nota:M-tau-prime}
{\rm 
For $\tau \in \Xi_n$, let 
$M_\tau^\prime \in {\rm Mat}_{n \times {\rm ex}_\tau}(\ZZ)$
be defined as in 
\autoref{thm:M-1} for the $\TT$-Poisson CGL extension $R_\tau$. 
Recall from \eqref{eq:ep-tau-prime} the definition of $\ep_\tau^\prime =\pm 1$ for $j \in {\rm ex}_\tau$.
\hfill $\diamond$
}
\end{notation}


\begin{theorem}\label{thm:M-3}
Let $R$ be any  symmetric $\TT$-Poisson CGL extension of length $n$ as in \autoref{de:sym-CGL-intro}.
For every $\tau \in \Xi_n$, the integer matrix 
\begin{equation}\label{eq:M-M-tau}
M_\tau = (\tau_\bullet \tau) (M_\tau^\prime \ep_\tau^\prime) 
((\tau_\bullet \tau)^{-1})_{{\rm ex}_\tau \times {\rm ex}} \in {\rm Mat}_{n \times {\rm ex}}(\ZZ)
\end{equation}
is the unique solution  to the linear equations 
\eqref{eq:for-M-tau}
in ${\rm Mat}_{n \times {\rm ex}}({\bf k})$, and explicitly one has
 \begin{align}\label{eq:MM-tau-0}
M_\tau & = E^t \tau_\bullet \bnu^{-1}  \tau_\bullet^t E \Lambda = 
E^t \tau_\bullet Q\, \tau_\bullet^{t} \, E_{n \times {\rm ex}} =
(E^{-1}\tau_\bullet^{-1} E)^t M (E^{-1}\tau_\bullet^{-1} E)_{{\rm ex} \times {\rm ex}},
\end{align}
where  the matrices $E, \bnu$,  $\Lambda$, and $Q = \bnu^{-1}\overline{\Lambda}$ are respectively given in 
\eqref{eq:E-Et}, \eqref{eq:bnu}, \eqref{eq:Lambda}, and \eqref{eq:def-Q-calH}, and $M = M_{\rm id}$. 
Here  ${\rm id}$ is again the identity element of $S_n$.
\end{theorem}

\begin{proof}
 Let $\tau \in \Xi_n$. 
Let\footnote{What is denoted as ${\bf q}_\tau^\prime$ 
here is denoted as ${\bf q}_\tau$ in \cite[$\S$11]{GY:Poi-CGL}.}  ${\bf q}_\tau^\prime$ be the Poisson coefficient matrix of $\bfy_\tau^\prime$ with respect to 
$\{\, , \, \}$. For $j \in [1, n]$, let $\chi_{y_{\tau, j}^\prime} \in X(\TT)$ be the 
$\TT$-character of $y_{\tau, j}^\prime$, and let
\[
\chi_{\bfy_\tau^\prime} =(\chi_{y_{\tau, 1}^\prime}, \ldots, \chi_{y_{\tau_, n}^\prime}),
\]
By the definition of $\bfy_\tau$ in terms of $\bfy_\tau^\prime$, we have
\[
{\bf q}_\tau = ((\tau_\bullet \tau)^{-1})^t {\bf q}_\tau^\prime (\tau_\bullet \tau)^{-1} =(\tau_\bullet \tau) {\bf q}_\tau^\prime (\tau_\bullet \tau)^{-1}  \hs \mbox{and} \hs
\chi_{{\bf y}_\tau} = \chi_{{\bf y}_\tau^\prime} (\tau_\bullet \tau)^{-1}.
\]
It follows that the equations in \eqref{eq:for-M-tau} on $M_\tau \in {\rm Mat}_{n \times {\rm ex}}(\bk)$ become
\[
{\bf q}_\tau^\prime (\tau_\bullet \tau)^{-1} M_\tau =-(\tau_\bullet \tau)^{-1} \Lambda \hs \mbox{and} \hs
\chi_{{\bf y}_\tau^\prime} (\tau_\bullet \tau)^{-1}M_\tau = 0,
\]
which, due to $\tau_\bullet\tau: {\rm ex}_\tau \to {\rm ex}$ being a bijection and the identity
$\Lambda_\tau \ep_\tau^\prime = (\tau_\bullet \tau)^{-1} 
\Lambda (\tau_\bullet \tau)_{{\rm ex} \times {\rm ex}_\tau}$ in \eqref{eq:E-Lambda-tau}, 
are in turn equivalent to
\[
{\bf q}_\tau^\prime (\tau_\bullet \tau)^{-1} M_\tau (\tau_\bullet \tau)_{{\rm ex} \times {\rm ex}_\tau}=-\Lambda_\tau \ep_\tau^\prime \hs \mbox{and} \hs
\chi_{{\bf y}_\tau^\prime} (\tau_\bullet \tau)^{-1}M_\tau (\tau_\bullet \tau)_{{\rm ex} \times {\rm ex}_\tau}= 0.
\]
By \autoref{thm:M-1} applied to the $\TT$-Poisson CGL extension $R_\tau$, we must have
\[
(\tau_\bullet \tau)^{-1} M_\tau (\tau_\bullet \tau)_{{\rm ex} \times {\rm ex}_\tau} = M_\tau^\prime 
\ep_\tau^\prime.
\]
Thus $M_\tau = (\tau_\bullet \tau) (M_\tau^\prime \ep_\tau^\prime) 
((\tau_\bullet \tau)^{-1})_{{\rm ex}_\tau \times {\rm ex}}$ is the unique solution to 
\eqref{eq:for-M-tau} in ${\rm Mat}_{n \times {\rm ex}}(\bk)$. 

To prove the first two identities in \eqref{eq:MM-tau-0}, 
recall from 
\eqref{eq:blam} the  matrix $\blam$. Let 
$\blam_\tau$ be the Poisson coefficient matrix of 
$(x_{\tau(1)}, \ldots, x_{\tau(n)})$ with respect to the log-canonical part of $\{\,, \, \}$. 
Then $\blam_\tau = \tau^{t} \blam \tau$. 
Applying  \autoref{le:bfq} to the $\TT$-Poisson CGl extension $R_\tau$, one has
\begin{equation}\label{eq:bfq-prime}
{\bf q}_\tau^\prime = F_\tau (\tau^{t} \blam \tau) F_\tau^t \hs \mbox{and} \hs
\chi_{{\bf y}^\prime_\tau} =  (\chi_{\tau(1)}, \ldots, \chi_{\tau(n)}) F_\tau^t = \chi_{\bf x} \tau F_\tau^t. 
\end{equation}
On the other hand, by the definition of $\bfy_\tau$ in \eqref{eq:by-tau} in terms of $\bfy_\tau^\prime$ one has
\begin{equation}\label{eq:bfq-pprime}
{\bf q}_\tau = ((\tau_\bullet \tau)^{-1})^t {\bf q}_\tau^\prime (\tau_\bullet \tau)^{-1} = 
(\tau_\bullet \tau) {\bf q}_\tau^\prime (\tau_\bullet \tau)^t \hs \mbox{and} \hs
\chi_{{\bf y}_\tau} = \chi_{{\bf y}_\tau^\prime} (\tau_\bullet \tau)^t.
\end{equation}
Combining \eqref{eq:bfq-prime} and \eqref{eq:bfq-pprime} and using \eqref{eq:E-Lambda-tau}, one gets
\begin{equation}\label{eq:q-chi-tau-1}
{\bf q}_\tau = F \tau_\bullet \blam (F\tau_\bullet)^t \hs \mbox{and} \hs \chi_{{\bf y}_\tau} = \chi_{\bf x} (F\tau_\bullet)^t.
\end{equation}
Recall now that ${\bf q} = F \blam F^t$ and $\chi_{\bfy} = \chi_{\bf x} F^t$. 
Using \eqref{eq:q-chi-tau-1} one sees that  \eqref{eq:for-M-tau} are equivalent to 
\[
{\bf q} (F \tau_\bullet F^{-1})^t M_\tau = - (F \tau_\bullet F^{-1})^{-1}\Lambda \hs \mbox{and}\hs   
\chi_{\bf y}(F \tau_\bullet F^{-1})^t M_\tau=0.
\]
By \autoref{:linear-alg-1} and the facts that $F = E^{-1}$ and $E \Lambda
= \overline{\Lambda}\, E_{n \times {\rm ex}}$, and by \eqref{eq:tau-overline-Lambda}, one gets
\[
M_\tau  = E^t \tau_\bullet \bnu^{-1}  \tau_\bullet^t E \Lambda = 
E^t \tau_\bullet \bnu^{-1}  \tau_\bullet^{t} \,\overline{\Lambda}\, E_{n \times {\rm ex}}
=E^t \tau_\bullet \,Q\, \tau_\bullet^{t} \, E_{n \times {\rm ex}} 
\]
To prove the last identity in \eqref{eq:MM-tau-0}, recall that 
$M=E^t Q E_{n \times {\rm ex}}$ from \eqref{eq:M-EQE}. Thus
\begin{align*}
(E^{-1}\tau_\bullet^{-1} E)^t M (E^{-1}\tau_\bullet^{-1} E)_{{\rm ex} \times {\rm ex}}
&=(E^{-1}\tau_\bullet^{-1} E)^t E^t Q\, E_{n \times {\rm ex}}(E^{-1}\tau_\bullet^{-1} E)_{{\rm ex} \times {\rm ex}} \\
&=
E^t \taub Q E_{n \times \rm ex} (E^{-1} \taubi E)_{{\rm ex} \times {\rm ex}}
\end{align*}
By \autoref{lm:K-tau-0}, $E_{n \times \rm ex} (E^{-1} \taubi E)_{{\rm ex} \times {\rm ex}}
= E (E^{-1} \taubi E)_{n \times {\rm ex}} = \taubi E_{n \times {\rm ex}}$. Thus
\[
(E^{-1}\tau_\bullet^{-1} E)^t M (E^{-1}\tau_\bullet^{-1} E)_{{\rm ex} \times {\rm ex}}=
E^t \taub Q \tau_\bullet^t E_{n \times {\rm ex}}.
\]
This finishes the proof of \autoref{thm:M-3}.
\end{proof}

\begin{remark}\label{rk:M-tau}
{\rm
Under 
the condition \eqref{eq:QQ-pos} in \autoref{thm:GY-Mtau}, i.e.,  
\begin{equation}\label{eq:QQ-again}
\frac{\lambda_{s(j)}}{\lambda_{s(k)}} \in \QQ_{>0}, \hs \forall \; j, k \in {\rm ex},
\end{equation}
the matrix $M_\tau$ for every $\tau \in \Xi_n$ is skew-symmetrizable by \autoref{rk:pre-seed}, so
\begin{equation}\label{eq:M-tau-prime-M-tau}
M_\tau^\prime \ep_\tau^\prime =
(\tau_\bullet \tau)^{-1} M_\tau
((\tau_\bullet \tau))_{{\rm ex} \times {\rm ex}_\tau } \in {\rm Mat}_{n \times {\rm ex}_\tau}(\ZZ)
\end{equation}
is also skew-symmetrizable. 
On the other hand, by \eqref{eq:lam-tau-j} and \eqref{eq:ep-tau-prime},  for every $j \in {\rm ex}_\tau$ one has
\[
\ep_\tau^\prime(j)\lambda_{\tau, s_\tau(j)} = \begin{cases}\ep_\tau^\prime(j)\lambda_{\tau s_\tau(j)} = 
\lambda_{\tau s_\tau(j)}, & \hs s_\tau(j) \in \tau(+),\\
\ep_\tau^\prime(j)\lambda^\ast_{\tau s_\tau(j)} = -\lambda^\ast_{\tau s_\tau(j)}, & \hs s_\tau(j) \in \tau(-).\end{cases}
\]
By the formula  in \eqref{eq:ps-ex-tau} for $s_\tau$, for every $j \in {\rm ex}_\tau$ one has 
As $\tau s_\tau(j) = \tau_\bullet^{-1} s(\tau_\bullet \tau(j))$ for $j \in {\rm ex}_\tau$, and as
$L(\tau s_\tau(j))=L(s(\tau_\bullet \tau(j))) = L(\tau_\bullet \tau(j))$, 
one has by \eqref{eq:Lam-max-min} that 
\[
\ep_\tau^\prime(j)\lambda_{\tau, s_\tau(j)} = \lambda_{\sL(s(\tau_\bullet \tau(j)))} = 
\lambda_{\sL(\tau_\bullet \tau(j))}, \hs j \in {\rm ex}_\tau.
\]
Thus \eqref{eq:QQ-pos} in \autoref{thm:GY-Mtau} is equivalent to 
\[
\frac{\ep_\tau^\prime(j)\lambda_{\tau, s_\tau(j)}}{\ep_\tau^\prime(j)\lambda_{\tau, s_\tau(k)}} \in \QQ_{>0}, 
\hs j, k \in {\rm ex}_\tau,
\]
for one, equivalently, for all, $\tau \in \Xi_n$.  We thus also know  that
 \eqref{eq:QQ-pos} in \autoref{thm:GY-Mtau} implies 
 that $M_\tau^\prime \ep_\tau^\prime$ is skew-symmetrizable for every $\tau \in \Xi_n$ by 
 applying \autoref{thm:M-1} directly to $R_\tau$.
Note also that 
in the notation of \autoref{de:seed-reordering}, we have
\[
({\bf y}_\tau, M_\tau) = ({\bf y}_\tau^\prime, M_\tau^\prime\ep_\tau^\prime)^{(\tau_\bullet\tau)^{-1}},
\]
i.e., $({\bf y}_\tau, M_\tau)$ is  the re-ordering of
$({\bf y}_\tau^\prime, M_\tau^\prime \ep_\tau^\prime)$ by $(\tau_\bullet \tau)^{-1}$.
\hfill $\diamond$
}
\end{remark}

Recall from \eqref{eq:a-LL} the non-negative Cartan integers $a_{\sL', \sL}$ associated to two distinct 
levels $L$ and $L'$ of $R = R_{\rm id}$ such that $|L|\geq 2$ and
$L' \cap [{\rm min}(L), {\rm max}(L)] \neq \emptyset$.
Recall from \eqref{eq:M-jk} the description of the entries of $M = M_{\rm id}$ in terms of the
Cartan integers. For $\tau \in \Xi_n$, we want to give a similar description of the entries of $M_\tau$.

For $\tau \in \Xi_n$, we first prove more properties of $\taubi$. Recall that 
$\tau_\bullet$ leaves invariant the level set $L(k)$ 
associated to $R = R_{\rm id}$ for every $k \in [1, n]$. 
We now show that $\tau_\bullet^{-1}$ has similar {\it interval property} on level sets of
$R_{\rm id}$
as $\tau$ does on $[1, n]$.  
For $k \in [1, n]$, recall from \eqref{eq:om-k} and \eqref{eq:op-k}
that $o_-(k)\geq 0$ and
$o_+(k) \geq 0$ are respectively the $p$-order and the $s$-order of $k$. 

 
\begin{lemma}\label{lm:taubi-00}
Let $\tau \in \Xi_n$ and $k \in [1, n]$. Let $j = (\taub\tau)^{-1}(k)$ and $L = L(k) =L(\tau(j))$. Then 
\[
\taubi(L \cap [1, k]) = L \cap \tau([1, j]).
\]

1) If $\taubi(k) > \taubi(k^\prime)$ for some $k' \in L \cap [1, k-1]$, then 
$\taubi(k) > \taubi(k^\prime)$ for all $k' \in L \cap [1, k-1]$, and in such a case $j \in \tau(+)$ and 
\[
\taubi(L \cap [1, k]) = L \cap [p^{o_-(k)} \taubi(k), \; \taubi(k)];
\]

2) If $\taubi(k) < \taubi(k^\prime)$ for some $k' \in L \cap [1, k-1]$, then  $\taubi(k) < \taubi(k^\prime)$ for all $k' \in L \cap [1, k-1]$, and in such a case $j \in \tau(-)$ and 
\[
\taubi(L \cap [1, k]) =L \cap [\taubi(k), \, s^{o_-(k)} \taubi(k)].
\]
\end{lemma}

\begin{proof}
By the definitions of $o_-(k)$ and $o_+(k)$,  the level set $L$ is given by
\begin{equation}\label{eq:Lk-increase}
L = \{p^{o_-(k)}(k) < \cdots < \;p(k) <k <s(k)< \cdots < s^{o_+(k)}(k)\}.
\end{equation}
Applying $(\taub \tau)^{-1}$ 
to both sides of \eqref{eq:Lk-increase} and using the fact that 
(see \eqref{eq:ps-ex-tau})
\[
(\taub \tau)^{-1} p^a(k)=p_\tau^a(j) \;\; \mbox{for}\;\;
a\in [0, o_-(k)] \hs \mbox{and}\hs  
(\taub \tau)^{-1} s^a(k)=s_\tau^a(j)  \;\; \mbox{for}\;\;a \in [0, o_+(k)]
\]
and the fact that 
$(\taub \tau)^{-1}: L \rightarrow  \tau^{-1}(L)$ is 
order-preserving,
one has  
\begin{equation}\label{eq:taui-L-1}
\tau^{-1}(L) =(\taub \tau)^{-1}(L) =  \{p_\tau^{o_-(k)}(j)< \cdots <p_\tau(j) < j
<s_\tau(j)< \cdots < s_\tau^{o_+(k)}(j)\}.
\end{equation}
 It then follows that 
\begin{align}\label{eq:taui-L-2}
\taubi(L \cap [1, k]) &= \taubi \{p^{o_-(k)}(k) < \cdots < \;p(k) <k\}=\tau \{p_\tau^{o_-(k)}(j)< \cdots <p_\tau(j) < j\}\\
\nonumber& = \tau(\tau^{-1}(L) \cap [1, j]) =L \cap \tau([1, j]).
\end{align}
Suppose that $\taubi(k) > \taubi(k^\prime)$ for some $k' \in L \cap [1, k-1]$, and let
$j' = (\taub \tau)^{-1}(k^\prime)$. Then 
\[
j = (\taub \tau)^{-1}(k) > (\taub \tau)^{-1}(k^\prime) = j^\prime \hs \mbox{and} \hs
\tau(j) = \taubi(k)> \taubi(k^\prime) =\tau(j').
\]
 Thus  $j \in \tau(+)$, and it follows from \eqref{eq:taui-L-2} that
$\taubi(k) = \tau(j) = {\rm max}(L \cap \tau([1, j]))$. As $\tau([1, j])$ is a sub-interval of
$[1, n]$, there exists $m \geq 0$ such that
\[
L \cap \tau([1, j]) = \{p^m \taubi(k) < \cdots < p \taubi(k) < \taubi(k)\}.
\]
Since $|L \cap \tau([1, j])| = |\taubi(L \cap [1, k])| = |L \cap [1, k]| = o_-(k)+1$, we have $m = o_-(k)$.
 This proves Statement 1). 
Statement 2)  is proved similarly. 
\end{proof}

\begin{remark}\label{rk:prop-11.5}
{\rm The arguments in the proof of \autoref{lm:taubi-00} are from \cite[Page 81]{GY:Poi-CGL} for proving
\cite[Proposition 11.5]{GY:Poi-CGL}, which we now recall: for $\tau \in \Xi_n$
recall from \eqref{eq:by-tau} the initial 
extended cluster $\bfy_\tau = (y_{\tau,1}, \ldots, y_{\tau, n})$, where $y_{\tau, k} = 
y_{\tau, (\taub \tau)^{-1}(k)}^\prime$ for $k \in [1, n]$. Let $j \in [1, n]$ and let
$m = |L(\tau(j)) \cap \tau([1, j])|$. By \autoref{prop-y-tau}, 
\[
y_{\tau, j}^\prime = \begin{cases} y_{[p^m(\tau(j)), \, \tau(j)]}, & \hs \mbox{if}\;\;\;
j \in \tau(+),\\
y_{[\tau(j), \, s^m (\tau(j))]}, & \hs \mbox{if}\;\;\;
j \in \tau(-).\end{cases}
\]
By \autoref{lm:taubi-00}, $m = o_-(k)$, where $k = (\taub\tau)(j)$. 
We thus have \cite[Proposition 11.5]{GY:Poi-CGL}, which says that
for every  $k \in [1, n]$, 
one has\footnote{The $o_+(k)$ in the formula 
for $\widetilde{Y}_\tau(e_k)$ in \cite[Proposition 11.5]{GY:Poi-CGL} for the case of $\taubi(k) \leq \tau(1)$ 
is a typo. It should be $o_-(k)$.}
\[
y_{\tau, k} = \begin{cases} y_{[p^{o_-(k)}(\taubi(k)), \; \taubi(k)]}, & \hs 
\mbox{if}\;\; (\taub\tau)^{-1}(k) \in \tau(+),\\
y_{[\taubi(k), \; s^{o_-(k)}(\taubi(k))]}, & \hs 
\mbox{if}\;\; (\taub\tau)^{-1}(k) \in \tau(-).\end{cases}
\]
\hfill $\diamond$
}
\end{remark}

For $\tau \in \Xi_n$ and  $k \in {\rm ex}$, recall the definitions of $\kto$ and $\ktt$ 
from \eqref{eq:k-tau}.

\begin{lemma}\label{lm:taubi-0}
Let $\tau \in \Xi_n$, and let $k \in {\rm ex}$ and $L = L(k) \subset [1, n]$. 
Then
\begin{equation}\label{eq:taubi-2}
\taubi(s(k))= \begin{cases} s({\rm max} (\taubi(L \cap [1, k]))),& \hs \mbox{if} \;\;\;
\taubi(k) < \taubi(s(k)),\\
p({\rm min} (\taubi(L \cap [1, k]))),& \hs \mbox{if} \;\;\;
\taubi(k) > \taubi(s(k)).
\end{cases}
\end{equation}
Moreover, if 
$L\cap [\kto, \ktt]$ contains at least three elements, then 
\begin{equation}\label{eq:L-kkkk}
L\cap [\kto, \ktt] = \taubi(L\cap [1, s(k)]).
\end{equation}
\end{lemma}

\begin{proof}
Assume first that $\taubi(k) < \taubi(s(k))$, 
Let $k' \in L \cap [1, k]$ be such that 
${\rm max}(\taubi(L \cap [1, k]) = \taubi(k')$. By \autoref{lm:taubi-00},
\begin{equation}\label{eq:taubi-sk}
\taubi(s(k)) = {\rm max} (\taubi(L \cap [1, s(k)]) > {\rm max}(\taubi(L \cap [1, k])) = \taubi(k^\prime).
\end{equation}
In particular, $\taubi(k^\prime) \in {\rm ex}$.   Let  
$k^{\prime \prime} \in L$ be such that 
$\taubi(k^{\prime \prime}) = s(\taubi(k^\prime))$. Then by \eqref{eq:taubi-sk},
\begin{equation}\label{eq:kkk}
\taubi(k^{\prime\prime})  \leq \taubi(s(k)).
\end{equation}
As $\taubi(k^{\prime\prime}) > \taubi(k^\prime)$, the definition of $k^\prime$ implies that
$k^{\prime\prime}\notin [1, k]$, so $k^{\prime\prime} \geq s(k)$. 
If $k^{\prime\prime} > s(k)$, then since $k^{\prime\prime} > k^\prime$ and 
$\taubi(k^{\prime\prime}) > \taubi(k^\prime)$, applying \autoref{lm:taubi-00} to $k^{\prime\prime}$ one gets 
$\taubi(k^{\prime\prime}) > \taubi(s(k))$, contradicting \eqref{eq:kkk}. Thus 
$k^{\prime\prime} = s(k)$. This proves the first case in \eqref{eq:taubi-2}. 
The second case in \eqref{eq:taubi-2} is proved similarly.

Assume now that $L\cap [\kto, \ktt]$ contains at least three elements,
and assume first that $\tau_\bullet^{-1}(k) < \tau_{\bullet}^{-1}(s(k))$, so that $\kto = \tau_\bullet^{-1}(k)$ and $\ktt = \tau_{\bullet}^{-1}(s(k))$. 
Consider
\begin{equation}\label{eq:taubi-s-sk}
\tau_\bullet^{-1}(L \cap [1, s(k)]) = \taubi(L \cap [1, k-1]) \sqcup \{\taubi(k), \taubi(s(k)\}.
\end{equation}
Suppose that  $k' \in L$  is such that 
\begin{equation}\label{eq:kkk-prime}
\taubi(k) < \taubi(k^\prime) < \taubi(s(k)).
\end{equation}
If $k^\prime > s(k)$, then it follows from $\taubi(k^\prime) < \taubi(s(k))$ and 
\autoref{lm:taubi-00} that 
$\taubi(k^\prime) < \taubi(k)$, a contradiction. As $k^\prime \notin \{k, s(k)\}$, 
we have $k^\prime \in L \cap [1, k-1]$. Thus by \eqref{eq:taubi-s-sk}.
\[
L \cap [\taubi(k), \, \taubi(s(k))] \subset \taubi(L \cap [1, s(k)]).
\]
The assumption on $L \cap [\kto, \ktt]$ implies that there exists
$k^\prime \in L$ such that \eqref{eq:kkk-prime} holds, and we have just shown that we must have
$k' \in [1, k-1]$.
By \autoref{lm:taubi-00} again,
\[
\taubi(k) < \taubi(k^{\prime\prime}) < \taubi(s(k))
\]
for all  $k^{\prime\prime}
\in L \cap [1, k-1]$. Thus $\taubi(L \cap [1, s(k)]) \subset L \cap [\taubi(k), \, \taubi(s(k))]$.
This proves \eqref{eq:L-kkkk} under the assumption that $\taubi(k) < \taubi(s(k))$. 
That \eqref{eq:L-kkkk} holds under the assumption that $\taubi(k) > \taubi(s(k))$ is proved similarly.
\end{proof}

For $\tau \in \Xi_n$, we can now describe the entries of $M_\tau$
in terms of the Cartan integers associated to $R$. For $k \in [1, n]$, we set 
\begin{equation}\label{eq:ep-tau-k}
\ep^\tau(k)=\begin{cases} 1, & \hs \mbox{if} \;\;\; 
\tau_\bullet^{-1}(k)< \tau_{\bullet}^{-1}(s(k)) \hs (\mbox{including when}\; s(k) = +\infty),\\
-1
& \hs \mbox{if} \;\;\; 
\tau_\bullet^{-1}(k)> \tau_{\bullet}^{-1}(s(k)).\end{cases}
\end{equation}
Note then that $\ep^\tau(k) 
=  \ep_\tau^\prime((\taub \tau)^{-1}(k))$ for $k \in {\rm ex}$.

\begin{theorem}\label{thm:M-tau-cases}
Let $R$ be any symmetric $\T$-Poisson CGL extension of length $n$,
and let
$\tau \in \Xi_n$. For $j \in [1, n]$, let (recall that $\taubi(+\infty) = +\infty$)
\[
\jto = {\rm min}\{\tau_\bullet^{-1}(j), \;\tau_{\bullet}^{-1}(s(j))\} \hs \mbox{and} \hs
\jtt =  {\rm max}\{\tau_\bullet^{-1}(j),\; \tau_{\bullet}^{-1}(s(j))\}.
\]
For $j \in [1, n]$ and $k \in {\rm ex}$, the $(j, k)$-entry $(M_\tau)_{j, k}$ of $M_\tau$
given as follows:

1) if $L(j)=L(k)$, then $(M_\tau)_{j, k} = 0$ except that 
\[
(M_\tau)_{s(k), k} = -\epsilon^\tau({k}) \hs \mbox{and} \hs 
(M_\tau)_{p(k), k}  = \epsilon^\tau({p(k)}) \; (\mbox{when}\; p(k)\neq -\infty); 
\]

2) if $L(j) \neq L(k)$, then  $(M_\tau)_{j, k} = 0$ except that
\begin{equation}\label{eq:m-tau-jk}
(M_\tau)_{j, k} = \begin{cases} \epsilon^\tau(j) \epsilon^\tau(k) a_{\sL(j), \sL(k)}, & 
\; \jto < \kto < \jtt<\ktt,\\
-\epsilon^\tau(j) \epsilon^\tau(k) a_{\sL(j), \sL(k)}, & \; \kto < \jto < \ktt
<\jtt \; (\mbox{including}\;\jtt = +\infty),\end{cases}.
\end{equation}
In particular, all the non-zero entries of $M_\tau$ are either $\pm 1$ or $\pm a$, where $a$ is
a negative Cartan integer associated to $R$.
\end{theorem}

\begin{proof}
Let $\tau \in \Xi_n$. We use the formula
$M_\tau = E^t \tau_\bullet Q\, \tau_\bullet^{t} \, E_{n \times {\rm ex}}
= E^t \tau_\bullet Q\, \tau_\bullet^{-1} \, E_{n \times {\rm ex}}$
in \eqref{eq:for-M-tau} to determine the entries of $M_\tau$. Fix $k \in {\rm ex}$. 
Then for every $j \in [1, n]$, we have (recall that $e_{+\infty} = 0$)
\begin{align*}
(M_\tau)_{j, k} &= e_j^t E^t \tau_\bullet Q\, \tau_\bullet^{-1} \, Ee_k =
(\taubi E e_j)^t Q (e_{\taubi(k)}-e_{\taubi(s(k))})\\
& =\ep^\tau(j)\ep^\tau(k) (e^t_{\jto}-e^t_{\jtt}) Q (e_{\kto}-e_{\ktt}),
\end{align*}
where note that $\jto = \taubi(j)$ 
and $\jtt = +\infty$ for $j \in [1, n]\backslash {\rm ex}$. 
Let $m \geq 1$ be such that $\ktt = s^m(\kto)$. Then the elements in 
$L(k) \cap [\kto, \ktt]$ are given in 
the increasing order as
$\kto < s(\kto) < \cdots < s^{m}(\kto)$,
and by \eqref{eq:QE-bth} we have
\begin{align*}
Q(e_{\kto}-e_{\ktt})&= Q(e_{\kto}-e_{s(\kto)} +  \cdots +
e_{s^{m-1}(\kto)} -e_{s^m(\kto)})\\
&= \bth^{(\kto, \,s(\kto))} + \cdots + \bth^{(s^{m-1}(\kto), \,s^m(\kto))}.
\end{align*}
By \eqref{eq:bth-Cartan}, the column vector $Q(e_{\kto}-e_{\ktt})$ is of the form
\begin{equation}\label{eq:Q-ee}
Q(e_{\kto}-e_{\ktt}) =  
(0, \,\ldots, \,0, \,-1, \, -a_{\sL(\kto+1), \, \sL(k)}, \, \ldots, \, 
 -a_{\sL(\ktt-1), \,\sL(k)}, \, -1, \, 0, \, \ldots, \, 0)^t,
\end{equation}
where the two $-1$ entries are at the positions $\kto$ and $\ktt$, 
and we set $a_{\sL(k), \, \sL(k)} = 2$. Let
\begin{align*}
U_k &= e_{\kto} + e_{\ktt} + Q(e_{\kto}-e_{\ktt})=
(0, \,\ldots, \,0,  \, -a_{\sL(\kto+1), \, \sL(k)}, \, \ldots, \, 
 -a_{\sL(\ktt-1), \,\sL(k)}, \ \, 0, \, \ldots, \, 0)^t.
\end{align*}
For $j \in [1, n]$, we then have
\[
(M_\tau)_{j, k} = -\ep^\tau(j)\ep^\tau(k) 
(e^t_{\jto}-e^t_{\jtt})(e_{\kto} + e_{\ktt})
+\ep^\tau(j)\ep^\tau(k) (e^t_{\jto}-e^t_{\jtt}) U_k.
\]

Suppose first that $L(j) \neq L(k)$. Then 
$(e^t_{\jto}-e^t_{\jtt})(e_{\kto} + e_{\ktt})=0$,  
and it follows from the formula for $U_k$ that $(M_\tau)_{j, k} = 0$ or otherwise as given in 
\eqref{eq:m-tau-jk}.

Suppose now that $L(j) = L(k)$ and 
$j \notin \{p(k), k, s(k)\}$ (when $p(k) \neq -\infty$). Then 
$\{\jto, \jtt\}
\cap \{\kto, \ktt\} = \emptyset$, so again 
$(e^t_{\jto}-e^t_{\jtt})(e_{\kto} + e_{\ktt})=0$.
Considering  $(e^t_{\jto}-e^t_{\jtt}) U_k$, it is clear that 
$(M_\tau)_{j, k} = 0$ except possibly when 
\[
a)\; \;\jto < \kto < \jtt < \ktt, 
\hs \mbox{or}\hs b) \;\;\kto < \jto < \ktt < \jtt.
\]
By \autoref{lm:taubi-0}, case a) would imply both $k<j$ and $j<k$, which is not possible. In case b), 
$\kto < \jto < \ktt$ would imply  $j < k$ so $j \in {\rm ex}$,
and it would then follow from $\jto < \ktt < \jtt$ that $k < j$, 
again not possible. Thus $(M_\tau)_{j, k} = 0$ for all $j \in L \backslash \{p(k), k, s(k)\}$.

By \eqref{eq:Q-ee}, $M_{k, k} = 0$. We now compute $(M_\tau)_{j, k}$ for $j = s(k)$ and $j = p(k)
\neq -\infty$ using
\begin{equation}\label{eq:M-tau-eeQ}
(M_\tau)_{j, k} =\ep^\tau(j)\ep^\tau(k) (e^t_{\jto}-e^t_{\jtt}) Q (e_{\kto}-e_{\ktt})
\end{equation}
and \eqref{eq:Q-ee}. If $s(k) \notin {\rm ex}$, then $\ep^\tau(s(k)) = 1$, and 
\[
(M_\tau)_{s(k), k} =\ep^\tau(k) e^t_{\taubi(s(k))} Q (e_{\kto}-e_{\ktt})=
-\ep^\tau(k).
\]
Assume now that $s(k) \in {\rm ex}$. 
Suppose first that $\epsilon^\tau(k) = 1$. Then by \autoref{lm:taubi-00}, 
\[
\tau_\bullet^{-1}(k) < \tau_\bullet^{-1}(s(k)) < \tau_\bullet^{-1}(s^2(k))\hs \mbox{or}\hs
 \tau_\bullet^{-1}(s^2(k)) <\tau_\bullet^{-1}(k) < \tau_\bullet^{-1}(s(k)),
\]
where in the first case $\ep^\tau(s(k)) = 1$, so by \eqref{eq:M-tau-eeQ}
\[
(M_\tau)_{s(k), k} = (e^t_{\taubi(s(k))}-e^t_{s^2(\taubi(k))})Q(e_{\taubi(k)}-e_{\taubi(s(k))})=-1-0=-1 = -\ep^\tau(k),
\]
and in the second case $\ep^\tau(s(k)) = -1$, so by \eqref{eq:M-tau-eeQ}
\[
(M_\tau)_{s(k), k}= -(e^t_{\taubi(s^2(k))}-e^t_{(\taubi(s(k))})Q(e_{\taubi(k)}-e_{\taubi(s(k))})=-(0+1) =-1 = -\ep^\tau(k).
\]
Suppose that 
$\epsilon^\tau(k) = -1$. Then by
\autoref{lm:taubi-00}, 
\[
\tau_\bullet^{-1}(s^2(k)) < \tau_\bullet^{-1}(s(k)) < \tau_\bullet^{-1}(k) \hs \mbox{or} \hs 
\tau_\bullet^{-1}(s(k)) <\tau_\bullet^{-1}(k) < \tau_\bullet^{-1}(s^2(k)),
\]
where in the first case $\ep^\tau(s(k)) = -1$, so by \eqref{eq:M-tau-eeQ}
\[
(M_\tau)_{s(k), k} = (e^t_{\taubi(s^2(k))}-e^t_{\taubi(s(k))})Q(e_{\taubi(s(k))}-e_{\taubi(k)})=0-(-1)=1 = -\ep^\tau(k),
\]
and in the second case $\ep^\tau(s(k)) = 1$, so 
\[
(M_\tau)_{s(k), k}= -(e^t_{\taubi(s(k))}-e^t_{\taubi(s^2(k))})Q(e_{\taubi(s(k))}-e_{\taubi(k)})=-(-1-0)=1 = -\ep^\tau(k).
\]
Turning to $(M_\tau)_{p(k), k}$ when $p(k) \neq -\infty$,
assume first that $\ep^\tau(p(k)) = 1$. Then
by \autoref{lm:taubi-00}, 
\[
\taubi(p(k)) < \taubi(k) < \taubi(s(k)) \hs \mbox{or}\hs 
\taubi(s(k)) < \taubi(p(k)) < \taubi(k),
\]
where in the first case  $\ep^\tau(k) = 1$, so by \eqref{eq:M-tau-eeQ}
\[
(M_\tau)_{p(k), k} = (e^t_{\taubi(p(k))}-e^t_{\taubi(k)})Q(e_{\taubi(k)}-e_{\taubi(s(k))}) = 0-(-1)=1=\ep^\tau(p(k)),
\]
and in the second case we $\ep^\tau(k) =-1$, so by \eqref{eq:M-tau-eeQ}
\[
(M_\tau)_{p(k), k} = -(e^t_{\taubi(p(k))}-e^t_{\taubi(k)})Q(e_{\taubi(s(k))}-e_{\taubi(k)}) = -(-2+1) = 1
=\ep^\tau(p(k)).
\]
Suppose now that $\ep^\tau(p(k)) = -1$. Then
by \autoref{lm:taubi-00}, 
\[
\taubi(k) < \taubi(p(k)) < \taubi(s(k)) \hs \mbox{or}\hs 
\taubi(s(k)) < \taubi(k) < \taubi(p(k)),
\]
where in the first case $\ep^\tau(k) = 1$, so by \eqref{eq:M-tau-eeQ}
\[
(M_\tau)_{p(k), k} =  -(e^t_{\taubi(k)}-e^t_{\taubi(p(k))})Q(e_{\taubi(k)}-e_{\taubi(s(k))}) = -(-1+2)=-1
=\ep^\tau(p(k)),
\]
and in the second case  $\ep^\tau(k) = -1$, so by \eqref{eq:M-tau-eeQ},
\[
(M_\tau)_{p(k), k} =  (e^t_{\taubi(k)}-e^t_{\taubi(p(k))})Q(e_{\taubi(s(k))}-e_{\taubi(k)})=-1-0= -1
=\ep^\tau(p(k)).
\]
This finishes the proof of \autoref{thm:M-tau-cases}.
\end{proof}

\subsection{The mutation matrices $M_\tau^\prime \ep_\tau^\prime$}\label{ss:seeds-prime}
Let again $R$ be a length $n$ symmetric $\TT$-Poisson CGL extension.
For applications, such as in the Lie theoretical examples to be presented in $\S$\ref{s:BFZ}, we sometimes want to consider, for $\tau \in \Xi_n$,  the Goodearl-Yakimov initial seeds 
$({\bf y}_\tau^\prime, M_\tau^\prime\ep_\tau^\prime)$
associated to the proper re-ordering $R_\tau$ instead of its re-ordering $(\by_\tau, M_\tau)$
by the permutation $(\tau_\bullet \tau)^{-1}$. While the extended cluster ${\bf y}_\tau^\prime$ is 
described in 
\autoref{prop-y-tau} and \autoref{rk:prop-11.5}, by \autoref{thm:M-3} we have
\begin{equation}\label{eq:M-tau-prime-product}
M_\tau^\prime \ep_\tau^\prime=(\taub\tau)^{-1} M_\tau (\taub \tau)_{{\rm ex} \times {\rm ex}_\tau}
= (\tau E_\tau)^t \,Q (\tau E_\tau)_{n \times {\rm ex}_\tau}.
\end{equation}
For an explicit description of the entries of $M_\tau^\prime\ep_\tau^\prime$, define, for 
$j \in [1, n]$
\[
j_\tau^{[1]} = {\rm min}\{\tau(j), \, \tau s_\tau(j)\} \hs \mbox{and} \hs
j_\tau^{[2]} = {\rm max}\{\tau(j), \, \tau s_\tau(j)\},
\]
where again $\tau(+\infty) = +\infty$. Extend the definition of $\ep_\tau^\prime$ in \eqref{eq:ep-tau-prime}
by 
\begin{equation}\label{eq:ep-tau-extend}
\ep_\tau^\prime(j) = \begin{cases} 1, & \hs \mbox{if} \;\; s_\tau(j) \in \tau(+) \;\; \mbox{or}\;\; 
s_\tau(j) = +\infty,\\
-1, & \hs \mbox{if} \;\; s_\tau(j) \in \tau(-).\end{cases}
\end{equation}
We  now have the 
following direct 
 consequence of \autoref{thm:M-tau-cases}.

\begin{corollary}\label{cor:M-tau-prime}
 Let $R$ be any length $n$ symmetric $\TT$-Poisson CGL extension.
For any $\tau \in \Xi_n$, $j \in  [1, n]$ and $k \in {\rm ex}_\tau$, 
 the $(j, k)$-entry $(M_\tau^\prime \ep_\tau^\prime)_{j, k}$ of
$M_\tau^\prime \ep_\tau^\prime \in {\rm Mat}_{n \times {\rm ex}_\tau}(\ZZ)$ is given as follows: 

1) if $L(\tau(j)) = L(\tau(k))$, then $(M_\tau^\prime \ep_\tau^\prime)_{j, k}=0$ except that
\[
(M_\tau^\prime \ep_\tau^\prime)_{p_\tau(k), k} = \ep_\tau^\prime(p_\tau(k)) \;\;(\mbox{when}\;\; p_\tau(k) \neq -\infty) \hs 
\mbox{and} \hs (M_\tau^\prime \ep_\tau^\prime)_{s_\tau(k), k} = -\ep_\tau^\prime(k);
\]

2) if $L(\tau(j)) \neq L(\tau(k))$, then $(M_\tau^\prime \ep_\tau^\prime)_{j, k}=0$ except that 
\[
(M_\tau^\prime \ep_\tau^\prime)_{j, k} = \begin{cases} \ep_\tau^\prime(j) \ep_\tau^\prime(k) a_{\sL(\tau(j)), \sL(\tau(k))}, 
& \; j_\tau^{[1]} < k_\tau^{[1]} < j_\tau^{[2]}
<k_\tau^{[2]},\\
-\ep_\tau^\prime(j) \ep_\tau^\prime(k) a_{\sL(\tau(j)), \sL(\tau(k))}, & \; k_\tau^{[1]} < j_\tau^{[1]} < k_\tau^{[2]}
<j_\tau^{[2]} \;\; (\mbox{including}\;\; 
s_\tau(j) = +\infty).\end{cases}
\]
\end{corollary}

To further analyze the cases in 2) of \autoref{cor:M-tau-prime}, 
for $\tau \in \Xi_n$, recalling  from \eqref{eq:tau-plus} and \eqref{eq:tau-minus} the
sets $\tau(+) \subset [2, n]$ and $\tau(-)\subset [2, n]$, we define
$\ep_\tau: [2, n] \to \{1, -1\}$ by
\begin{equation}\label{eq:ep-tau}
\ep_\tau(j) = \begin{cases} 1, & \hs j \in \tau(+),\\
-1, & \hs j \in \tau(-).\end{cases}
\end{equation}

\begin{theorem}\label{thm:M-tau-prime} Let $R$ be any length $n$ symmetric $\TT$-Poisson CGL extension.
For any $\tau \in \Xi_n$, $j \in  [1, n]$ and $k \in {\rm ex}_\tau$, including when 
$s_\tau(j) = +\infty$,
the $(j, k)$-entry $(M_\tau^\prime \ep_\tau^\prime)_{j, k}$ of
$M_\tau^\prime \ep_\tau^\prime \in {\rm Mat}_{n \times {\rm ex}_\tau}(\ZZ)$ is given as follows: 
\[
(M_\tau^\prime\ep_\tau^\prime)_{j, k} = \begin{cases} \ep_\tau(k), & s_\tau(j) = k, \\
-\ep_\tau(j), &   j = s_\tau(k), \\
\ep_\tau(k) a_{\sL(\tau(j)), \sL(\tau(k))}, & j < k < s_\tau(j) < s_\tau(k)\;\; \mbox{and}\;\;
\ep_\tau(k) =  \ep_\tau(s_\tau(j)), \\
&\mbox{or}\;\; j < k < s_\tau(k) < s_\tau(j) \; \;\mbox{and}\;\; \ep_\tau(k) = -\ep_\tau(s_\tau(k)), \\
-\ep_\tau(j) a_{\sL(\tau(j)), \sL(\tau(k))}, &  k < j < s_\tau(k) < s_\tau(j)
\;\; \mbox{and} \;\;  \ep_\tau({j})=\ep_\tau({s_\tau(k)}),\\
&\mbox{or}\;\;k < j < s_\tau(j) < s_\tau(k)\;\; \mbox{and}\;\; \ep_\tau({j})=-\ep_\tau({s_\tau(j)}), \\
0, &  \mbox{otherwise}.\end{cases}
\]
\end{theorem}

\begin{proof}
Let $j \in [1, n]$ and $k\in {\rm ex}_\tau$. Assume first that $L(\tau(j)) = L(\tau(k))$. 
 If $j \notin \{p_\tau(k), s_\tau(k)\}$ then $(M_\tau^\prime\ep_\tau^\prime)_{j, k} = 0$ by \autoref{cor:M-tau-prime}.
If $j = p_\tau(k)$, i.e., if $s_\tau(j) = k$, then by \autoref{cor:M-tau-prime}, 
$(M_\tau^\prime\ep_\tau^\prime)_{j, k} = \ep_\tau^\prime(j) = \ep_\tau(k)$. If $j = s_\tau(k)$, by \autoref{cor:M-tau-prime} again,
$(M_\tau^\prime\ep_\tau^\prime)_{j, k} = -\ep_\tau^\prime(k) = -\ep_\tau(j)$.

Assume now that $L(\tau(j)) \neq L(\tau(k))$. We say that we are in 

Case a) if $j_\tau^{[1]} < k_\tau^{[1]} < j_\tau^{[2]}<k_\tau^{[2]}$;

Case b) if  $k_\tau^{[1]} < j_\tau^{[1]} < k_\tau^{[2]} <j_\tau^{[2]}$  (including $j_\tau^{[2]}= +\infty$);

Case c) otherwise. 

\noindent
On the other hand, we have either $j < k < s_\tau(k)$ or $k < j < s_\tau(j)$, including when
$s_\tau(j) =+\infty$, which lead to the following 
six mutually exclusive cases:
\begin{align*}
&(1) \;  j < s_\tau(j) < k < s_\tau(k), \hs  (2) \;j < k < s_\tau(j) < s_\tau(k), \hs 
(3)\;j < k < s_\tau(k) < s_\tau(j),\\
&(4) \;  k < s_\tau(k) < j < s_\tau(j), \hs  (5) \;k < j < s_\tau(k) < s_\tau(j), \hs 
(6)\;k < j < s_\tau(j) < s_\tau(k).
\end{align*}
We further examine each case as follows. 

Assume (1). Then $j_\tau^{[1]} < j_\tau^{[2]} < \tau(k)$ if $\ep_\tau(k) = 1$, and
 $\tau(k) < j_\tau^{[1]} < j_\tau^{[2]}$ if $\ep_\tau(k) = -1$, and both cases lead to Case c) 
whether  $\ep_\tau(s_\tau(k)) = 1$ or $\ep_\tau(s_\tau(k)) = -1$. Thus
$(M_\tau^\prime\ep_\tau^\prime)_{j, k}= 0$.

Assume (2).  If $\ep_\tau(k) = -\ep_\tau(s_\tau(j))$, then 
\[
\tau(k) > \tau(j) > \tau(s_\tau(j)) \hs \mbox{or}\hs \tau(k) < \tau(j) < \tau(s_\tau(j)),
\]
and both cases lead to Case c), so $(M_\tau^\prime\ep_\tau^\prime)_{j, k} = 0$. If $\ep_\tau(k) = \ep_\tau(s_\tau(j))=\epsilon$, then
\begin{equation}\label{eq:tau-jk-2}
 \tau(j) < \tau(k) < \tau(s_\tau(j))\; \;\mbox{if} \;\;\epsilon = 1, \hs\mbox{and} \hs
\tau(j) > \tau(k) > \tau(s_\tau(j)) \; \;\mbox{if} \;\;\epsilon = -1,
\end{equation}
and in both cases of \eqref{eq:tau-jk-2} we are in Case a) if $\ep_\tau(s_\tau(k)) = 1$, which gives
\[
(M_\tau^\prime\ep_\tau^\prime)_{j,k} = \ep_{\tau}(s_\tau(j))\ep_{\tau}(s_\tau(k))a_{\sL(\tau(j)), \sL(\tau(k))}=
\ep_\tau(k)a_{\sL(\tau(j)), \sL(\tau(k))},
\]
and we are in Case b) if 
$\ep_\tau(s_\tau(k)) = -1$, which again gives
\[
(M_\tau^\prime\ep_\tau^\prime)_{j,k} = -\ep_{\tau}(s_\tau(j))\ep_{\tau}(s_\tau(k))a_{\sL(\tau(j)), \sL(\tau(k))}=
\ep_\tau(k)a_{\sL(\tau(j)), \sL(\tau(k))}.
\]

Assume (3), possibly with $s_\tau(j)=+\infty$.  If $\ep_\tau(k) = \ep_\tau(s_\tau(k))$, then 
\[
\tau(j) < \tau(k) < \tau(s_\tau(k)) \hs \mbox{or}\hs \tau(s_\tau(k)) < \tau(k) < \tau(j),
\]
and both cases lead to Case c), so $(M_\tau^\prime\ep_\tau^\prime)_{j, k} = 0$. If $\ep_\tau(k) = -\ep_\tau(s_\tau(k))=\epsilon$, then
\begin{equation}\label{eq:tau-jk-3}
 \tau(k) > \tau(j) > \tau(s_\tau(k))\; \;\mbox{if} \;\;\epsilon = 1, \hs\mbox{and} \hs
\tau(k) < \tau(j) < \tau(s_\tau(k)) \; \;\mbox{if} \;\;\epsilon = -1,
\end{equation}
and in both cases of \eqref{eq:tau-jk-3} we are in Case a) if $\ep_\tau(s_\tau(j)) = -1$, which gives
\[
(M_\tau^\prime\ep_\tau^\prime)_{j,k} = \ep_{\tau}(s_\tau(j))\ep_{\tau}(s_\tau(k))a_{\sL(\tau(j)), \sL(\tau(k))}=
\ep_\tau(k)a_{\sL(\tau(j)), \sL(\tau(k))},
\]
and we are in Case b) if 
$\ep_\tau(s_\tau(j)) = 1$, including when $s_\tau(j) = +\infty$, which gives
\[
(M_\tau^\prime\ep_\tau^\prime)_{j,k} = -\ep_{\tau}(s_\tau(j))\ep_{\tau}(s_\tau(k))a_{\sL(\tau(j)), \sL(\tau(k))}=
\ep_\tau(k)a_{\sL(\tau(j)), \sL(\tau(k))}.
\]
Cases (4)-(6) are analyzed similarly, and one proves that $\widetilde{m}_{j, k}$ is as given in all the cases.
\end{proof}

\section{BFZ mutation matrices associated to signed words}\label{s:BFZ}
\subsection{Symmetric Poisson CGL extensions from generalized Cartan matrices}\label{ss:R-A-i}
Let $A = (a_{i, i'})_{i, i' \in [1, r]}$ be a symmetrizable generalized Cartan matrix with  a
fixed choice of a symmetrizer $(d_i)_{i \in [1, r]}$. Let $\{\alpha_1, \ldots, \alpha_r\}$ be a set of simple roots in the
root system associated to $A$, and let $\TTA$ be the split complex torus whose character lattice is the root lattice 
$\calQ=\sum_{i=1}^r \ZZ \alpha_i$. Let $\t_\sA$ be the Lie algebra of $\TT_\sA$, so that $\t_\sA^* = \sum_{i=1}^r \CC \alpha_i$.
Let $\langle \,, \, \rangle_\sA$ be the unique symmetric bilinear form  on $\t_\sA^*$
such that 
\[
\langle \alpha_i, \, \alpha_{i'}\rangle_\sA = d_ia_{i, i'}, \hs i, i'  \in [1, n].
\]
For $i \in [1, r]$, let $s_i$ be the reflection operator on $\t_\sA^*$ defined by $\alpha_i$.
Given any sequence 
\[
{\bf i} = (i_1, \ldots, i_n) \in [1, r]^n,
\]
setting $\beta_j = s_{i_1}\cdots s_{i_{j-1}}\alpha_{i_j}\in \calQ$ for $j \in [1, n]$,
one then has the log-canonical Poisson structure 
\[
\pi_0^{(A, \bfi)} =- \sum_{j<k} \langle \beta_j, \, \beta_k\rangle_\sA x_jx_k \frac{\partial}{\partial x_j}
\wedge \frac{\partial}{\partial x_j}
\]
on $\CC^n$. Let $\TT_\sA$ act on $\CC^n$ such that 
 $x_j$ has $\TT_\sA$-weight $\beta_j$ for $j \in [1, n]$. It is shown in 
\cite[$\S$6.2]{Lu-Mykola:deformation} that  there is a uniquely defined 
$\TTA$-invariant algebraic Poisson structure 
$\pi^{(A, {\bf i})}$ on $\CC^n$ with $\pi^{(A, {\bf i})}_0$ as its log-canonical term,
and  that the induced Poisson algebra 
\begin{equation}\label{eq:R-A-i}
R^{(A, \bfi)}= (\CC[x_1, \ldots, x_n], \, \{\,, \, \}_{\pi^{(A, {\bf i})}})
\end{equation}
is a normal
(\autoref{de:normal}) symmetric $\TTA$-Poisson CGL extension. Moreover, 
\[
\lambda_j = \langle \alpha_{i_j}, \, \alpha_{i_j}\rangle_\sA = 2d_{i_j}, \hs j \in [1, n],
\]
where $\lambda_j$ for $j \in [1, n]$ is defined in \eqref{eq:lambda-j} for $R^{(A, \bfi)}$. 
Thus the scalar condition in \eqref{eq:QQ-pos} is satisfied. By \cite[Theorem 11.1]{GY:Poi-CGL}
(see \autoref{thm:GY-Mtau}), one  has the family $\{(\bfy_\tau, M_\tau): \tau \in \Xi_n\}$
of mutation equivalent $\TTA$-Poisson seeds in $\CC(x_1, \ldots, x_n)$, defining
 a cluster structure on
$\CC[x_1, \ldots, x_n]$ compatible with both $\{\,, \, \}_{\pi^{(A, {\bf i})}}$ and the $\TTA$-action.

We remark that 
the Poisson structure $\pi^{(A, \bfi)}$ on $\CC^n$ is shown in 
\cite{Lu-Mykola:deformation} to be the
 unique {\it maximal normalized admissible algebraic deformation} of the log-canonical Poisson
structure $\pi_0^{(A, \bfi)}$. 
When $A$ is of finite type and if $G$ is a simply connected
complex semi-simple Lie group of the same Cartan-Dynkin type as $A$, 
the Poisson structure $\pi^{(A, {\bf i})}$ coincides with the 
{\it standard Poisson structure} on 
the Bott-Samelson cell of $G$ associated to ${\bf i}$ in Bott-Samelson coordinates (see
\cite{EL:BS} and \cite[$\S$6.3]{Lu-Mykola:deformation} for details).

Our main goal in $\S$\ref{s:BFZ} is to compare the mutation matrices $M_\tau$ associated to
the symmetric CGL extensions $R^{(A, \bfi)}$ for $\tau \in \Xi_n$, or rather
their permutations $M_\tau^\prime\ep_\tau^\prime$ as in $\S$\ref{ss:seeds-prime}, with the
matrices, refereed to as {\it BFZ mutation matrices} in the literature
\cite{BFZ:III, BZ:quantum, Shen-Weng:dBS, Qin:dual, CQW:i-boxes} that are associated to $A$ and 
{\it signed words} \cite{Qin:dual, CQW:i-boxes}, or {\it double words} \cite{BZ:quantum}, in
$[1, r]$. The fact that these two collections of mutation matrices are the same is proved in 
\autoref{thm:same-BFZ}. We also explain in $\S$\ref{ss:ensemble} how  the $n \times n$ \textit{nondegenerate cluster ensemble matrix} associated to $(A, \bfi)$ introduced in \cite{Wil:ensemble}
can be written as a matrix product.

\subsection{The matrices $\widehat{M}_\tau(\bfi)$ associated to a generalized Cartan matrix $A$}\label{ss:hat-M-tau}
In this section, we fix a symmetrizable generalized Cartan matrix $A = (a_{i, i'})_{i, i' \in [1, r]}$,
and let
\[
\bfi = (i_1, \ldots, i_n)
\]
be any sequence in $[1, r]$.
For each $\tau \in \Xi_n$, we then have the $\TT_\sA$-Poisson CGL extension 
$R^{(A, \bfi)}_\tau$, defined as the proper re-ordering of the symmetric 
$\TT_\sA$-Poisson CGL extension $R^{(A, \bfi)}$ by $\tau$
as in $\S$\ref{ss:reordering}, and the mutation matrix $M_\tau^\prime\ep_\tau^\prime
\in {\rm Mat}_{n \times {\rm ex}_\tau}(\ZZ)$ considered in 
$\S$\ref{ss:seeds-prime}. 

For $j \in [1, n]$, let again $L(j) \subset [1, n]$ be the level set of $j$ associated to $R^{(A, \bfi)}$
(see \autoref{nota-p-s}). It is shown in \cite[$\S$6.2]{Lu-Mykola:deformation} 
that $L(j) = \{j' \in [1, n]: i_{j'} = i_j\}$, and that 
\begin{equation}\label{eq:a-LL-Lie}
a_{\sL(j), \sL(k)} = a_{i_j, i_k}, \hs j, k \in [1, n],
\end{equation}
where 
$a_{\sL(j), \sL(k)}$ is defined in \eqref{eq:a-LL}. Moreover, the matrix $Q$ in \eqref{eq:M-tau-prime-product}
is given by
\begin{equation}\label{eq:Q-bfi}
Q =\left(\begin{array}{ccccc} 1 & a_{i_1, i_2} & \cdots & a_{i_1, i_{n-1}}  & a_{i_1,i_n}\\
0 & 1 &\cdots & a_{i_2, i_{n-1}}  & a_{i_2, i_n}\\
\cdots & \cdots & \cdots & \cdots & \cdots \\
0 & 0 & \cdots & 1 & a_{i_{n-1}, i_n}\\
0 & 0 & \cdots & 0 & 1\end{array}\right).
\end{equation}

\begin{notation}\label{nota:hat-M-tau}
{\rm
Fix a symmetrizable generalized Cartan matrix $A = (a_{i, i'})_{i, i' \in[1, r]}$. For any
sequence $\bfi = (i_1, \ldots, i_n)$ and any $\tau \in \Xi_n$, introduce the {\it square} matrix
\begin{equation}\label{eq:hat-M-tau}
\widehat{M}_\tau(\bfi) = (\tau E_\tau)^t \,Q \,\tau E_\tau \in {\rm Mat}_{n \times n}(\ZZ),
\end{equation}
where $E_\tau$ is defined as in \eqref{eq:E-tau-def} for  
$R^{(A, \bfi)}_\tau$, and $Q$ is given in \eqref{eq:Q-bfi}. 
\hfill $\diamond$
}
\end{notation}

Note that, 
although not indicated in the notation, the matrix $E_\tau$ in \eqref{eq:hat-M-tau}
depends on both $\tau$ and $\bfi$, and the matrix $Q$ in \eqref{eq:hat-M-tau} depends on $(A, \bfi)$.
By \eqref{eq:M-tau-prime-product}, 
\[
M_\tau^\prime\ep_\tau^\prime = (\widehat{M}_\tau(\bfi))_{n \times {\rm ex}_\tau}.
\]
In the following \autoref{thm:hat-M-tau}, we make use of the special form of
$Q$ in \eqref{eq:Q-bfi} to extend \autoref{thm:M-tau-prime} on $M_\tau^\prime\ep_\tau^\prime$ to an explicit description, in terms of $(\bfi, \tau)$ and the Cartan integers in $A$, of 
all the entries of full matrix
$\widehat{M}_\tau(\bfi)$.

\begin{lemma}\label{lm:s-tau-i}
For any $\bfi = (i_1, \ldots, i_n) \in [1, r]^n$ and $\tau \in \Xi_n$, the successor map
$s_\tau$ for the $\TTA$-Poisson CGL extension $R_\tau^{(A, \bfi)}$ 
is given, for $j \in [1, n]$, by 
\[
s_\tau(j) =  \begin{cases} {\rm min}\{j' \in [j+1, n]: i_{\tau(j')} = i_{\tau(j)}\}, & \hs 
\{j' \in [j+1, n]: i_{\tau(j')} = i_{\tau(j)}\}\neq \emptyset,\\
+\infty, & \hs \mbox{otherwise}.\end{cases}
\]
\end{lemma}

\begin{proof}
The statement for $\tau = {\rm id}$ is proved in \cite[$\S$6.2]{Lu-Mykola:deformation}. Let 
$\tau \in \Xi_n$ be arbitrary. Let $L$ be any level set associated to $R^{(A,\bfi)}$, i.e., 
$L = \{j\in [1, n]: i_j = i_0\}$ for some $i_0 \in \{i_1, \ldots, i_n\}$. 
Then 
\[
\tau^{-1}(L) = \{j \in [1, n]: \;\tau(j) \in L\} = \{j \in [1, n]: \; i_{\tau(j)} = i_0\}.
\]
List the elements in $\tau^{-1}(L)$ in the increasing order as 
$\tau^{-1}(L) = (j_1, j_2, \ldots, j_l\}$. By the definition of $s_\tau$, 
one has $s_\tau(j_a) = j_{a+1}$ for $a \in [1, l-1]$ and $s_\tau(j_l) = +\infty$.
Thus $s_\tau$ is as described.
\end{proof}

For $\tau \in \Xi_n$, recall from \eqref{eq:ep-tau} the definition of the function $\ep_\tau:
[2, n] \to \{1, -1\}$.

\begin{theorem}\label{thm:hat-M-tau}
Let $A = (a_{i, i'})_{i, i' \in [1, r]}$ be any symmetrizable generalized Cartan matrix.
For any sequence $\bfi = (i_1, \ldots, i_n)$ in $[1, n]$ and  any $\tau \in \Xi_n$, writing
$\widehat{M}_\tau(\bfi) = (\whm_{j, k})_{j, k \in [1, n]}$,  one has, for all $j, k \in [1, n]$ and including when
$s_\tau(j) =+\infty$ or $s_\tau(k) = +\infty$,
\[
\whm_{j, k} = \begin{cases} \ep_\tau(k), & s_\tau(j) = k, \\
-\ep_\tau(j), &   j = s_\tau(k), \\
1, & j = k \;\; \mbox{and}\;\; s_\tau(k) =+\infty,\\
\ep_\tau(k) a_{i_{\tau(j)}, i_{\tau(k)}}, & j < k < s_\tau(j) < s_\tau(k)\;\; \mbox{and}\;\;
\ep_\tau(k) =  \ep_\tau(s_\tau(j)), \\
&\mbox{or}\;\; j < k < s_\tau(k) < s_\tau(j) \; \;\mbox{and}\;\; \ep_\tau(k) = -\ep_\tau(s_\tau(k)), \\
-\ep_\tau(j) a_{i_{\tau(j)}, i_{\tau(k)}}, &  k < j < s_\tau(k) < s_\tau(j)
\;\; \mbox{and} \;\;  \ep_\tau({j})=\ep_\tau({s_\tau(k)}),\\
&\mbox{or}\;\;k < j < s_\tau(j) < s_\tau(k)\;\; \mbox{and}\;\; \ep_\tau({j})=-\ep_\tau({s_\tau(j)}), \\
0, &  \mbox{otherwise}.\end{cases}
\]
\end{theorem}

\begin{proof}
By \autoref{thm:M-tau-prime}, we only need to assume  $j \in [1, n]$ and $k \in [1, n]\backslash
{\rm ex}_\tau$ and prove that $\whm_{j, k}$ is as described, i.e., 
\[
\whm_{j, k} = \begin{cases} \ep_\tau(k), & s_\tau(j) = k, \\
1, & j = k \in [1, n]\backslash {\rm ex}_\tau,\\
\ep_\tau(k) a_{i_{\tau(j)}, i_{\tau(k)}}, & j < k < s_\tau(j)\;\; \mbox{and}\;\;
\ep_\tau(k) =  \ep_\tau(s_\tau(j)), \\
-\ep_\tau(j) a_{i_{\tau(j)}, i_{\tau(k)}},&k < j < s_\tau(j)\;\; \mbox{and}\;\; \ep_\tau({j})=-\ep_\tau({s_\tau(j)}), \\
0, &  \mbox{otherwise}.\end{cases}
\]
By \autoref{lm:s-tau-i}, we have $i_{\tau(j)} =i_{\tau(s_\tau(j))}$ when $s_\tau(j) \neq +\infty$, By \eqref{eq:Q-bfi}, we have
\begin{equation}\label{eq:mh-jk-0}
\whm_{j, k} 
= (e_{\tau(j)}^t-e_{\tau(s_\tau(j))}^t) \left(e_{\tau(k)} + \sum_{l=1}^{\tau(k)-1} a_{i_l, i_{\tau(k)}} e_l\right).
\end{equation}
Assume first that $L(\tau(j)) = L(\tau(k))$, so that $i_{\tau(j)} =i_{\tau(s_\tau(j))}=i_{\tau(k)}$.
If $j = k$, then $\whm_{j, k} = 1$ by \eqref{eq:mh-jk-0}. If $j \neq k$, then since 
$s_\tau(k) =+\infty$, we have 
\[
(i)\;\; j < s_\tau(j) <k, \hs \mbox{or}\hs (ii)\;\; s_\tau(j) = k.
\]
In (i), both $\tau(j)$ and $\tau(s_\tau(j))$ are less than $\tau(k)$ if $\ep_\tau(k) = 1$, and
both $\tau(j)$ and $\tau(s_\tau(j))$ are bigger than $\tau(k)$ if $\ep_\tau(k) = 1$, so
$\whm_{j, k} = 0$ by \eqref{eq:mh-jk-0}; In (ii), it follows from $a_{i_{\tau(k)}, i_{\tau(k)}} = 2$
and \eqref{eq:mh-jk-0} that $\whm_{j, k} = \ep_\tau(k)$. 
When $L(\tau(j)) \neq L(\tau(k))$, using $s_\tau(k) = +\infty$ and $\ep_\tau(s_\tau(k)) = 1$,
the same arguments used in the proof of 
\autoref{thm:M-tau-prime} show that $\whm_{j, k}$ is as described. 
\end{proof}

Continuing with the notation as above,  set 
${\rm fr} = [1,n] \backslash \text{ex}$ and 
${\rm fr}_\tau = [1,n] \backslash \text{ex}_\tau$  for $\tau \in \Xi$.

\begin{proposition}\label{prop:decomp}
For the symmetric Poisson CGL extension $R^{(A,\mathbf{i})}$ in (\ref{eq:R-A-i}) and for any 
$\tau \in \Xi_n$, the identity
\[
 (\tau E_\tau)^t Q(\tau E_\tau) =  \frac{1}{2}(\tau E_\tau )^t (Q-\overline{\Lambda}^{-1} Q^t \overline{\Lambda}) (\tau E_\tau) + \frac{1}{2}(\tau E_\tau )^t (Q+\overline{\Lambda}^{-1} Q^t \overline{\Lambda}) (\tau E_\tau)
\]
is a decomposition of 
$\widehat{M}_\tau(\bfi)= (\tau E_\tau)^t Q(\tau E_\tau)$ into its skew-symmetrizable part and symmetrizable part,  with
$\tau^{-1} \overline{\Lambda} \tau = {\rm diag}(2d_{i_{\tau(j)}})_{j \in [1, n]}$
as a left skew-symmetrizer, resp. symmetrizer. 
 Furthermore, the symmetrizable part of $\widehat{M}_\tau(\bfi)$ is supported on ${\rm fr}_\tau \times {\rm fr}_\tau$.
\end{proposition}
\begin{proof}
    The first statement follows from $\overline{\Lambda}(\tau E_{\tau}) = (\tau E_{\tau}) (\tau^{-1} \overline{\Lambda} \tau)$.
    The second one follows from
    \begin{equation}\label{eq:QQF}
     Q+\overline{\Lambda}^{-1} Q^t \overline{\Lambda}  = (F_{{\rm fr} \times n})^t A(\bfi)  F_{{\rm fr} \times n} = ((F_{\tau})_{{\rm fr}_\tau \times n} \tau^{-1})^t A(\bfi)  ((F_{\tau})_{{\rm fr}_\tau \times n} \tau^{-1}),
    \end{equation}
    where $A(\bfi) := A_{\text{supp}(\bfi) \times\text{supp}(\bfi)}$ 
    by identifying ${\rm fr}  = \text{supp}(\bfi) = {\rm fr}_\tau$. Note that the 
    first equality in \eqref{eq:QQF} is a special case of 
    the second one, which holds  due to the identity 
    $F_{{\rm fr} \times n}\tau= (F_\tau)_{{\rm fr}_\tau \times n}$, which is in turn
    proved by comparing the rows of the matrices on both sides.
\end{proof}

\begin{remark}\label{rk:why-skew-symmetric}
{\rm
In the setting of \autoref{prop:decomp}, 
the fact that the symmetrizable part of $\widehat{M}_\tau(\bfi)$ is supported at its ${\rm fr}_\tau \times {\rm fr}_\tau$
sub-matrix gives another explanation of the skew-symmetrizablility of its sub-matrix $M_\tau'\ep'_\tau$ (see
\autoref{rk:M-tau}). 
\hfill $\diamond$
}
\end{remark}

\subsection{Signed words and admissible triples}\label{ss:Srn-Trn}
Fix integers $r \geq 1$ and $n \geq 2$, and set
\begin{align*}
\Srn &= \{\bida  = (\ida_1, \ldots, \ida_n) : \;
\ida_j \in[-r, -1] \sqcup [1, r] \; \mbox{for every}\; j \in [1, n]\},\\
\Trn & = \{({\bf i} = (i_1, \ldots, i_n), \tau, \ep_1): \; 
i_j\in [1, r] \; \mbox{for every}\; j \in [1, n], \; \tau \in \Xi_n, \;
\ep_1= \pm 1\}.
\end{align*}
Any $\bida \in \Srn$ is called 
a {\it signed word} \cite{Qin:dual, CQW:i-boxes}, or a {\it double word} \cite{BZ:quantum}, in 
$[1, r]$ of length $n$.
We call any $(\bfi, \tau, \ep_1) \in \Trn$ an {\it $(r, n)$-admissible triple} or simply an 
{\it admissible triple}. 
In this section, we establish a bijection between $\Srn$ and $\Trn$.

 \begin{notation}\label{nota:bida}
{\rm 
For $\bida =  (\ida_1, \ldots, \ida_n) \in \Srn$, 
if
\[
\{a \in [1, n]: \ida_a \in [1, r]\} = \{a_1, \ldots, a_m\}  \hs \mbox{and} \hs
\{b \in [1, n]: \ida_b \in [-r, -1]\} = \{b_1, \ldots, b_{n-m}\} 
\]
with  $a_1 < \cdots < a_m$ and  $b_1 < \cdots < b_{n-m}$, we set
\begin{align}\label{eq:bida-plus}
\bida_+ &= (\ida_{a_1}, \;\ldots, \;\ida_{a_m}) \hs \mbox{and} \hs
\left(\bida_+\right)^{-1} = \left(\ida_{a_m}, \;\ldots, \;\ida_{a_1}\right),\\
\label{eq:bida-minus}
\bida_- &= (\ida_{b_1}, \;\ldots, \;\ida_{b_{n-m}}) \hs \mbox{and} \hs 
 -\bida_- = \left(-\ida_{b_1}, \;\ldots, \;-\ida_{b_{n-m}}\right),
\end{align}
and we denote by $\tau_{\bida}$ the element in  the permutation group $S_n$ given by
\begin{equation}\label{eq:tau-bida} 
(\tau_{\bida}(a_1), \, \ldots, \, \tau_{\bida}(a_m), \, \tau_{\bida}(b_1), \, \ldots, 
\tau_{\bida}(b_{n-m})) = (m,  \ldots, 1, \, m+1, \, \ldots, n),
\end{equation}
i.e., $\tau_{\bida}(a_t) = m+1-t$ for $t \in [1, m]$ and $\tau_{\bida}(b_t) = m+t$ for $t \in [1, n-m]$.
Also set 
\begin{equation}\label{eq:Pos-Neg}
{\rm Pos}(\bida) = \{a_1, \ldots,a_m\} \hs \mbox{and} \hs 
{\rm Neg}(\bida) = \{b_1, \ldots, b_{n-m}\}.
\end{equation}
\hfill $\diamond$
}
\end{notation}

For $\tau \in \Xi_n$, recall again from \eqref{eq:tau-plus} and \eqref{eq:tau-minus} the
sets $\tau(+) \subset [2, n]$ and $\tau(-)\subset [2, n]$.

\begin{lemma}\label{lm:tau-bida}
For any $\bida\in \Srn$, one has $\tau_{\bida} \in \Xi_n$, and 
\begin{equation}\label{eq:tau1}
\tau_{\bida}(1) = m \;\;\; \mbox{if}\;\;\;\ep_1 = 1, \hs and \hs
\tau_{\bida}(1) = m+1 \;\;\; \mbox{if}\;\;\;\ep_1 = -1,
\end{equation}
where $m = |{\rm Pos}(\bida)|$ and  $\ep_1 = {\rm sign}(\ida_1)$. Moreover, 
\begin{equation}\label{eq:bida-Pos-Neg}
\tau_{\bida}(-) = [2, n] \cap {\rm Pos}(\bida) \hs \mbox{and} \hs 
\tau_{\bida}(+) = [2, n] \cap {\rm Neg}(\bida).
\end{equation}
\end{lemma}

\begin{proof} 
With the notation as in \eqref{eq:Pos-Neg}, and by
\eqref{eq:bida-plus} and \eqref{eq:bida-minus}, 
$a_1 = 1$ if $\ep_1 = 1$ and $b_1 = 1$ if
$\ep_1 = -1$. Thus \eqref{eq:tau1} holds by \eqref{eq:tau-bida}. 
Let $c \in [2, n]$. 
If $c \in {\rm Pos}(\bida)$,  then $\tau_{\bida}(c) \leq m$  by \eqref{eq:tau-bida} and 
\[
\tau_{\bida}([1, c] \cap {\rm Pos}(\bida)) =[\tau_{\bida}(c), \;m], \hs 
\tau_{\bida}([1, c] \cap {\rm Neg}(\bida)) = [m+1, \; \tau_{\bida}({\rm max}([1, c] \cap {\rm Neg}(\bida))].
\]
If $c \in {\rm Neg}(\bida)$, then  $\tau_{\bida}(c) \geq m+1$ by \eqref{eq:tau-bida} and 
\[
\tau_{\bida}([1, c] \cap {\rm Pos}(\bida)) =[\tau_{\bida}({\rm max}([1, c] \cap {\rm Pos}(\bida)), \;m],\hs
\tau_{\bida}([1, c] \cap {\rm Neg}(\bida)) = [m+1, \; \tau_{\bida}(c)].
\]
In both cases, $\tau_{\bida}([1, c])$ is a sub-interval of $[1, n]$, and 
$c \in \tau_{\bida}(-)$ if $c \in {\rm Pos}(\bida)$ and 
$c \in \tau_{\bida}(+)$ if $c \in {\rm Neg}(\bida)$.
\end{proof}

\begin{lemma}\label{lm:calS-calT}
Given $(\bfi = (i_1, \ldots, i_n), \tau, \ep_1) \in \Trn$, define 
\begin{equation}\label{eq:S-i}
\calS(\bfi, \tau, \ep_1) = (\ep_1 i_{\tau(1)}, \; \ep_2 i_{\tau(2)}, \, \ldots, \, 
\ep_ni_{\tau(n)})\in \Srn,
\end{equation}
where for $j \in [2, n]$, $\ep_j = -1$ if $j \in \tau(+)$ and $\ep_j = 1$ if $j \in \tau(-)$.
Then the map 
\[
\calS:\;\; \Trn \longrightarrow \Srn, \;\; (\bfi, \tau, \ep_1) \longmapsto \calS(\bfi, \tau, \ep_1)
\]
is bijective, and its inverse is given by 
\begin{equation}\label{eq:T-bida}
\calT:\;\; \Srn \longrightarrow \Trn, \;\;\bida \longrightarrow  \calT(\bida) :=(\bfi, \tau, \ep_1), 
\end{equation}
where 
$\bfi = ((\bida_+)^{-1}, \; -\bida_-)$, $\tau = \tau_{\bida}$, and 
 $\ep_1 = {\rm sign}(\ida_1)$.
\end{lemma}

\begin{proof}
Let $(\bfi = (i_1, \ldots, i_n), \tau, \ep_1) \in \Trn$ and let $\bida = \calS(\bfi, \tau, \ep_1)
=(\ida_1, \ldots, \ida_n)$ be given as in \eqref{eq:S-i}. We first show that 
$\calT(\bida) = (\bfi, \tau, \ep_1)$. Note that one has $\ep_1 = {\rm sign}(\ida_1)$ by definitions.
We need to prove that $\bfi = ((\bida_+)^{-1}, \; -\bida_-)$ and $\tau = \tau_{\bida}$.

List the elements in $\tau(-)\subset [2, n]$ and $\tau(+)\subset [2, n]$ in 
increasing order as
\[
\tau(-) = \{l_1, \ldots, l_{m'}\} \hs \mbox{and} \hs 
\tau(+) = \{p_1, \ldots, p_{n-m'-1}\}.
\]
Then $1\leq \tau(l_{m'}) < \cdots < \tau(l_1) < \tau(1) < \tau(p_1) < \cdots < \tau(p_{n-{m'}-1})\leq n$. Thus  
\[
(\tau(l_{m'}), \;\ldots, \;\tau(l_1), \; \tau(1), \; \tau(p_1). \; \ldots,\; \tau(p_{n-{m'}-1})) =
(1, \,\ldots,\, {m'},\, {m'}+1,\, {m'}+2,  \,\ldots,\, n).
\]
If $\ep_1 = 1$, then 
\begin{align*}
\bida_+ & =  (i_{\tau(1)}, \, i_{\tau(l_1)}, \, \ldots, \, i_{\tau(l_{m'})}) = 
(i_{{m'}+1}, \, i_{{m'}}, \, \ldots, \, i_1),\\
\bida_-& = (-i_{\tau(p_1)}, \, \ldots, \, -i_{\tau(p_{n-{m'}-1})}) = (-i_{{m'}+2}, \, \ldots, \, -i_n),
\end{align*}
so by \eqref{eq:tau-bida}, $\tau = \tau_{\bida}$ and $\bfi = (i_1, \ldots, i_{m'}, i_{{m'}+1}, i_{{m'}+2}, \ldots, i_n) =
((\bida_+)^{-1}, \; -\bida_-)$. 
If $\ep_1 = -1$, then 
\begin{align*}
\bida_+ & =  (i_{\tau(l_1)}, \, \ldots, \, i_{\tau(l_{m'})}) = 
(i_{{m'}}, \, \ldots, \, i_1),\\
\bida_-& = (-i_{\tau(1)}, \, -i_{\tau(p_1)}, \, \ldots, \, -i_{\tau(p_{n-{m'}-1})}) 
= (-i_{{m'}+1}, \, -i_{{m'}+2}, \, \ldots, \, -i_n),
\end{align*}
so again by \eqref{eq:tau-bida}, $\tau = \tau_{\bida}$ and 
$\bfi = (i_1, \ldots, i_{m'}, i_{{m'}+1}, i_{{m'}+2}, \ldots, i_n) =
((\bida_+)^{-1}, \; -\bida_-)$. 
We thus proved that
$\calT(\calS((\bfi, \tau, \ep_1))) = (\bfi, \tau, \ep_1)$ for all 
$(\bfi, \tau, \ep_1)\in \Trn$.

Let now  $\bida =(\ida_1, \ldots, \ida_n) \in \Srn$ be arbitrary, and 
we prove that $\bida =\calS(\calT(\bida))$.
Let $\calT(\bida) = (\bfi, \tau, \ep_1)$, and write $\bfi = (i_1, \ldots, i_n)$.
Let
\[
{\rm Pos}(\bida) = \{a_1 < \cdots < a_m\} \hs \mbox{and} \hs 
{\rm Neg}(\bida) = \{b_1 < \cdots < b_{n-m}\}.
\]
It then follows from  
$\bfi = ((\bida_+)^{-1}, -\bida_-)$ that 
\[
(i_1,\, \ldots, \, i_m, \, i_{m+1}, \,  \ldots, i_n) = (\ida_{a_m},\, \ldots, \,
\ida_{a_1}, \,-\ida_{b_1}, \,\ldots, \,-\ida_{b_{n-m}}).
\]
By the definition of $\tau =\tau_{\bida}$ in \eqref{eq:tau-bida},  one then has 
\[
(i_{\tau(a_m)},\, \ldots, \, i_{\tau(a_1)}, \, i_{\tau(b_1)}, \,  \ldots, i_{\tau(b_{n-m})}) = (\ida_{a_m},\, \ldots, \,
\ida_{a_1}, \,-\ida_{b_1}, \,\ldots, \,-\ida_{b_{n-m}}).
\]
Setting $\ep_j = {\rm sign}(\ida_j)$ for $j \in [1, n]$, one then has
$\ida_j = \ep_j i_{\tau(j)}$ for every $j \in [1, n]$. Thus
\[
\bida =  (\ep_1 i_{\tau(1)}, \; \ep_2 i_{\tau(2)}, \, \ldots, \, 
\ep_ni_{\tau(n)}) = \calS((\bfi, \tau, \ep_1)) = \calS (\calT(\bida)).
\]
 We have thus finished proving that the maps $\calS$ and $\calT$ are inverses of each other.
\end{proof}

\begin{example}\label{ex:bida-0}
{\rm
For any $\bfi =(i_1, \ldots, i_n)$, we  have 
\[
\calS(\bfi, {\rm id}, -1) = -\bfi=(-i_1, \ldots, -i_n), \hs \mbox{and}\hs
\calS(\bfi, w_0, 1) = \bfi^{-}=(i_n, \ldots, i_1),
\]
where $w_0$ is the longest element in $S_n$.
\hfill $\diamond$
}
\end{example}

\begin{example}\label{ex:bida-1}
{\rm
Let  $r = 3$ and $n = 14$. For easy visualization, we present (signed) words 
 and eleents in $S_n$ in two-line notation. 

1) For 
$\bida =\left(\begin{array}{rrrrrrrrrrrrrr} 
\textcolor{blue}{1} & \textcolor{blue}{2} & \textcolor{blue}{3} & \textcolor{blue}{4} & \textcolor{blue}{5} 
& \textcolor{blue}{6} & \textcolor{blue}{7} & \textcolor{blue}{8} & \textcolor{blue}{9} & \textcolor{blue}{10} 
& \textcolor{blue}{11} & \textcolor{blue}{12} & \textcolor{blue}{13} & \textcolor{blue}{14}\\
1 & -2 & -2 & {2} & {3} & {3} &-3 
& {1} &{2}  & -1& {3}& -2 & {1} & 
{1}\end{array}\right)$, 
we have
\begin{align*}
\tau_{\bida}& = \left(\begin{array}{cccccccccccccc} 
\textcolor{blue}{1} & \textcolor{blue}{2} & \textcolor{blue}{3} & \textcolor{blue}{4} & \textcolor{blue}{5} 
& \textcolor{blue}{6} & \textcolor{blue}{7} & \textcolor{blue}{8} & \textcolor{blue}{9} & \textcolor{blue}{10} 
& \textcolor{blue}{11} & \textcolor{blue}{12} & \textcolor{blue}{13} & \textcolor{blue}{14}\\
\textcolor{red}{9} & 10 & 11& \textcolor{red}{8} & \textcolor{red}{7} & \textcolor{red}{6} & 
12 & \textcolor{red}{5} & \textcolor{red}{4} & 13 & \textcolor{red}{3} & 14 
& \textcolor{red}{2} & \textcolor{red}{1} \end{array}\right),\\
\bfi& =\left(\begin{array}{cccccccccccccc} 
\textcolor{blue}{1} & \textcolor{blue}{2} & \textcolor{blue}{3} & \textcolor{blue}{4} & \textcolor{blue}{5} 
& \textcolor{blue}{6} & \textcolor{blue}{7} & \textcolor{blue}{8} & \textcolor{blue}{9} & \textcolor{blue}{10} 
& \textcolor{blue}{11} & \textcolor{blue}{12} & \textcolor{blue}{13} & \textcolor{blue}{14}\\
1 & 1 & 3 & 2 & 1 & 3 & 3 & 2 & 1 & 2 & 2 & 3 & 1 & 2 \end{array}\right), \hs \mbox{and}\;\;\ep_1 = 1;
\end{align*}

2) For 
$\bida =\left(\begin{array}{rrrrrrrrrrrrrr} 
\textcolor{blue}{1} & \textcolor{blue}{2} & \textcolor{blue}{3} & \textcolor{blue}{4} & \textcolor{blue}{5} 
& \textcolor{blue}{6} & \textcolor{blue}{7} & \textcolor{blue}{8} & \textcolor{blue}{9} & \textcolor{blue}{10} 
& \textcolor{blue}{11} & \textcolor{blue}{12} & \textcolor{blue}{13} & \textcolor{blue}{14}\\
-1 & -2 & -2 & {2} & {3} & {3} &-3 
& {1} &{2}  & -1& {3}& -2 & {1} & 
{1}\end{array}\right)$, 
we have
\begin{align*}
\tau_{\bida} &= \left(\begin{array}{cccccccccccccc} 
\textcolor{blue}{1} & \textcolor{blue}{2} & \textcolor{blue}{3} & \textcolor{blue}{4} & \textcolor{blue}{5} 
& \textcolor{blue}{6} & \textcolor{blue}{7} & \textcolor{blue}{8} & \textcolor{blue}{9} & \textcolor{blue}{10} 
& \textcolor{blue}{11} & \textcolor{blue}{12} & \textcolor{blue}{13} & \textcolor{blue}{14}\\
{9} & 10 & 11& \textcolor{red}{8} & \textcolor{red}{7} & \textcolor{red}{6} & 
12 & \textcolor{red}{5} & \textcolor{red}{4} & 13 & \textcolor{red}{3} & 14 
& \textcolor{red}{2} & \textcolor{red}{1} \end{array}\right),\\
\bfi& =\left(\begin{array}{cccccccccccccc} 
\textcolor{blue}{1} & \textcolor{blue}{2} & \textcolor{blue}{3} & \textcolor{blue}{4} & \textcolor{blue}{5} 
& \textcolor{blue}{6} & \textcolor{blue}{7} & \textcolor{blue}{8} & \textcolor{blue}{9} & \textcolor{blue}{10} 
& \textcolor{blue}{11} & \textcolor{blue}{12} & \textcolor{blue}{13} & \textcolor{blue}{14}\\
1 & 1 & 3 & 2 & 1 & 3 & 3 & 2 & 1 & 2 & 2 & 3 & 1 & 2 \end{array}\right), \hs \mbox{and}\;\;\ep_1 =- 1;
\end{align*}

\noindent
Note that the two $\bida$s  differ only with $\ep_1 = {\rm sign}(\ida_1)$
and have the same $\tau_{\bida} \in \Xi_{14}$ and $\bfi$.
\hfill $\diamond$
}
\end{example}

\subsection{Goodearl-Yakimov mutation matrices and BFZ mutation matrices}\label{ss:GY-BFZ}
In this section, we fix a symmetrizable generalized Cartan matrix $A = (a_{i, i'})_{i, i' \in [1, r]}$.

\begin{notation}\label{nota:hat-M-bida}
{\rm
For a signed word $\bida = (\ida_1, \ldots, \ida_n)$ in $[1, r]$, we set
\[
\widehat{M}(\bida) = \widehat{M}_\tau(\bfi) \in {\rm Mat}_{n \times n}(\ZZ),
\]
where $(\bfi, \tau, \ep_1) =\calT(\bida)$ is defined in \autoref{lm:calS-calT} and
$\widehat{M}_\tau(\bfi) \in {\rm Mat}_{n \times n}(\ZZ)$ is given in \eqref{eq:hat-M-tau}.
\hfill $\diamond$
}
\end{notation}

\begin{remark}\label{rk:hat-M-bida}
{\rm
The matrix $\widehat{M}(\bida)$ depend only on $(\bfi, \tau)$ in
$\calT(\bfi, \tau, \ep_1)$ and not on $\ep_1 = {\rm sign}(\ida_1)$. The two examples
of $\bida$ in \autoref{ex:bida-1} thus give the same matrix $\widetilde{M}(\bida)$.
\hfill $\diamond$
}
\end{remark}

\begin{example}\label{ex:bida-01}
{\rm
If $\bida = (\ida_1, \ldots, \ida_n)$ is positive, i.e., if $\ida_j \in [1, r]$ for every $j \in [1, n]$, then 
\[
\widehat{M}(\bida) = \widehat{M}_{w_0}((\bida)^{-1}),
\]
where $w_0$ is again the longest element in $S_n$ and  $(\bida)^{-1} = (\ida_n, \ldots, \ida_1)$.
If $\bida = (\ida_1, \ldots, \ida_n)$ is negative, i.e., if
$\ida_j \in [-r, -1]$ for every $j \in [1, n]$, then 
\[
\widehat{M}(\bida) = \widehat{M}_{{\rm id}}(-\bida),
\]
where ${\rm id}$ is the identity element of $S_n$ and $-\bida = (-\ida_1, \ldots, -\ida_n)$.
\hfill $\diamond$
}
\end{example}

To express $\widehat{M}(\bida)$ directly using the pair $(A, \bida)$, we set up some more notation.

\begin{notation}\label{nota:hat-B-bida}
{\rm 
Let $\bida = (\ida_1, \ldots, \ida_n) \in \Srn$.
For $j \in [1, n]$,  let $\ep_j = {\rm sign}(\ida_j) =\pm 1$, and let
\[
j[1] = \begin{cases} {\rm min}\{j' \in [j+1, n]: |\ida_{j'}| = |\ida_j|\}, & \hs
\mbox{if}\;\; \{j' \in [j+1, n]: |\ida_{j'}| = \ida_j\}\neq \emptyset,\\
+\infty, & \hs \mbox{otherwise}.\end{cases}
\]
Let $E_{\bida}$ be the $n \times n$ lower triangular matrix whose $j^{\rm th}$ column
is (setting again $e_{+\infty} = 0$)
\[
E({\bida})e_j = e_j-e_{j[1]}, \hs j \in [1, n].
\]
With again ${\rm Pos}(\bida)= \{j \in [1, n]: \ep_j = 1\}$
and ${\rm Neg}(\bida)=\{j \in [1, n]: \ep_j = -1\}$  (see \eqref{eq:Pos-Neg}), 
let $\calP(\bida)$ be the set of all pairs $(j, k)$ with $j, k \in [1, n]$ such that
\[
 j \in {\rm Pos}(\bida) \; \mbox{and} \; k \in {\rm Neg}(\bida),
\hs\mbox{or}\;\; j, k \in {\rm Pos}(\bida) \; \mbox{and} \; j > k, 
\hs\mbox{or}\;\; j, k \in {\rm Neg}(\bida) \; \mbox{and} \; j < k,
\]
and let  $Q({\bida}) \in {\rm Mat}_{n \times n}(\ZZ)$  with $(j, k)$-entry, for all $j, k\in [1, n]$, given by
\[
Q({\bida})_{j, k} = \begin{cases} 1, & j = k,\\
a_{|\ida_j|, |\ida_k|}, & (j, k) \in \calP(\bida),\\
0, & \mbox{otherwise}.\end{cases}
\]
\hfill $\diamond$
}
\end{notation}

\begin{theorem}\label{thm:hat-M-bida}
Fix a symmetrizable generalized Cartan matrix $A = (a_{i, i'})_{i, i' \in [1, r]}$.
For any signed word $\bida = (\ida_1, \ldots, \ida_n)$, the  matrix 
$\widehat{M}(\bida) \in  {\rm Mat}_{n \times n}(\ZZ)$ is given by
\begin{equation}\label{eq:hat-M-bida-product}
\widehat{M}(\bida) =E(\bida)^t \,Q({\bida)} E({\bida}).
\end{equation}
Writing $\widehat{M}(\bida) = (\whm_{j, k})_{j, k \in [1, n]}$, for all $j, k \in [1, n]$, including when $j[1] = +\infty$ or $k[1]=+\infty$, one has
\begin{equation}\label{eq:whb-entries}
\whm_{j, k} = \begin{cases} -\ep_k, & j[1] = k, \\
\ep_j, &  j = k[1], \\
1, & j = k \in [1, n] \;\; \mbox{and}\;\;k[1]=+\infty,\\
-\ep_k a_{|\ida_j|, |\ida_k|}, &  j < k < j[1] < k[1] \;\; \mbox{and}\;\;
\ep_k =  \ep_{j[1]},\\
&\mbox{or}\;\;j < k < k[1] < j[1]\;\; \mbox{and}\;\; \ep_k = -\ep_{k[1]}, \\
\ep_j a_{|\ida_j|, |\ida_k|}, & k < j < k[1] < j[1]\;\; \mbox{and}\;\;
\ep_{j}=\ep_{k[1]},\\ 
&\mbox{or}\;\;k < j < j[1] < k[1] \;\;\mbox{and}\;\; \ep_{j}=-\ep_{j[1]},\\
0, &  \mbox{otherwise}.\end{cases}
\end{equation}
\end{theorem}

\begin{proof}
Write $\bfi = (i_1, \ldots, i_n)$ so that 
$\bida = (\ep_1 i_{\tau(1)}, \ldots, \ep_n i_{\tau(n)})$.
In the notation used in \autoref{thm:hat-M-tau} for the pair $(\bfi, \tau)$,  we have
$\ep_\tau(j) = -\ep_j$ for $j \in [2, n]$, and 
\[
i_{\tau(j)} = |\ida_j| \hs \mbox{and}\hs s_\tau(j) = j[1], \hs j \in [1, n].
\]
The explicit entry-wise description of $\widehat{M}(\bida)$ is thus a direct 
translation of \autoref{thm:hat-M-tau} applied to  $(\bfi, \tau)$. 
To prove the product formula \eqref{eq:hat-M-bida-product} for $\widehat{M}(\bida)$, we note first that 
$E_\tau = E(\bida)$ by definitions. With $Q$ given in \eqref{eq:Q-bfi} and for $j, k \in [1, n]$, the 
$(j, k)$-entry for $\tau^tQ \tau=\tau^{-1}Q \tau$ is then 
\[
(\tau^tQ \tau)_{j, k}  = \begin{cases} 1, & \tau(j) = \tau(k), \\
a_{i_{\tau(j)}, i_{\tau(k)}} = a_{|\ida_j|, |\ida_k|} , & \tau(j) < \tau(k), \\
0. & \mbox{otherwise}\end{cases}
\]
By the definition of $\tau = \tau_{\bida}$ in \eqref{eq:tau-bida}, for all $j, k \in [1, n]$ one has
$\tau(j) < \tau(k)$ if and only if $(j, k) \in \calP(\bida)$. Thus $\tau^t Q \tau = Q(\bida)$.
It follows that
\[
\widehat{M}(\bida) = (\tau E_\tau)^t \,Q \,\tau E_\tau = E_\tau^t \tau^t Q \tau E_\tau =
E(\bida)^t \,Q({\bida)} E({\bida}).
\]
\end{proof}

Recall that we have fixed a symmetrizable generalized Cartan matrix $A = (a_{i, i'})_{i, i' \in [1, r]}$.
For a signed word $\bida =(\ida_1, \ldots, \ida_n)$ in $[1, r]$, set
\[
{\rm ex}(\bida) = \{j \in [1, n]: j[1]\neq +\infty\}.
\]

\begin{definition}\label{nota:tilde-M-bida}
{\rm
For a signed word $\bida = (\ida_1, \ldots, \ida_n)$ in $[1, r]$, set
\[
\widetilde{M}(\bida) = (\widehat{M}_\tau(\bfi))_{n \times {\rm ex}(\bida)} \in 
{\rm Mat}_{n \times {\rm ex}(\bida)}(\ZZ),
\]
and we call $\widetilde{M}(\bida)$ the {\it Goodearl-Yakimov mutation matrix 
associated to the pair $(A, \bida)$}. 
\hfill $\diamond$
}
\end{definition}

\begin{remark}\label{rk:tilde-M-bida}
{\rm
Under the correspondence
$\bida \mapsto \calT(\bida) = (\bfi, \tau, \ep_1)$ and
in the notation of $\S$\ref{ss:seeds-prime}, 
the matrix $\widetilde{M}(\bida)$ is thus the Goodearl-Yakimov mutation matrix  
$M_\tau^\prime \ep_\tau^\prime \in {\rm Mat}_{n \times {\rm ex}_\tau}(\ZZ)$
associated to the $\TT_\sA$-Poisson CGL extension $R^{(A, \bfi)}_\tau$.
\hfill $\diamond$
}
\end{remark}

Associated to the pair $(A, \bida)$, and based on \cite{BFZ:III}, 
 A. Berenstein and A. Zelevinsky introduced in \cite[(8.7)]{BZ:quantum}
a matrix
\[
\widetilde{B}(\bida) \in {\rm Mat}_{n \times {\rm ex}(\bida)}(\ZZ),
\]
which we will call the {\it BFZ mutation matrix} associated to $(A, \bida)$. The same matrix
was also introduced\footnote{The matrix  
in \cite{BFZ:III, CQW:i-boxes} is the negative of that defined in \cite{BZ:quantum}.
See also \cite[Remark 2.4]{BFZ:III} and \cite[Remark 8.8]{BZ:quantum} for
alternative descriptions of the entries of $\widetilde{B}(\bida)$).}  in  \cite[(3)]{CQW:i-boxes}
(see also \cite{Shen-Weng:dBS, Qin:dual}).

\begin{theorem}\label{thm:same-BFZ}
For any symmetrizable generalized Cartan matrix $A = (a_{i, i'})_{i, i' \in [1, r]}$, and for 
any signed word $\bida =(\ida_1, \ldots, \ida_n)$ in $[1, r]$, one has
\[
\widetilde{B}(\bida) = \widetilde{M}(\bida)  = E(\bida)^t Q(\bida) E(\bida)_{n \times {\rm ex}(\bida)}.
\]
\end{theorem}

\begin{proof}
The identity $\widetilde{B}(\bida) = \widetilde{M}(\bida)$ follows directly from  the entry-wise description of 
$\widetilde{B}(\bida)$ in \cite[(8.7)]{BZ:quantum} and that of $\widetilde{M}(\bida)$ in
\autoref{thm:hat-M-bida}. The product matrix formula for 
$\widetilde{B}(\bida) = \widetilde{M}(\bida)$ follows from the definition of
$\widetilde{M}(\bida)$ as a sub-matrix of $\widehat{M}(\bida)$ and the matrix product formula for 
$\widehat{M}(\bida)$ in \autoref{thm:hat-M-bida}.
\end{proof}

\subsection{The nondegenerate cluster ensemble as a matrix product}\label{ss:ensemble}
 

Let $G$ be the Kac-Peterson group associated to a symmetrizable generalized Cartan matrix $A$ of size $r$
whose derived subgroup is generated by co-root subgroups, let $G_{\text{Ad}}$ be the quotient of $G$ by a discrete subgroup of the center of $G$, and let $W$ be the Weyl group of $G$. For their detailed construction, we refer to \cite[Appendix A.1]{Shen-Weng:dBS}.
Let $\bida$ be a 
double reduced word  of $(u,v) \in W \times W$, and let  $\bida_{\text{ext}} = [-\tilde{r},\ldots,-1] \sqcup \bida$, 
 where $\tilde{r} = r+\text{corank}(A)$.
In \cite[Proposition 3.28]{Wil:ensemble}, H. Williams computed the exponent 
matrix\footnote{Our definition is the transpose of the matrix in loc. cit. and \cite[Proposition 3.43, 3.44]{Shen-Weng:dBS}.}  \[
\widehat{B}(\bida_{\text{ext}}) \in {\rm Mat}_{(n+\tilde{r}) \times (n + \tilde{r})}(\ZZ)
\]
for 
the monomial change of coordinates between the 
co-weight parametrization of the double Bruhat cell $G^{u,v}_\text{Ad}$ in $G_{\rm Ad}$ 
and generalized Chamber Ansatz on the duble Bruhat cell $G^{u,v}$ in $G$. Also given  in 
\cite[Proposition 3.28]{Wil:ensemble} is 
a decomposition
\begin{equation}\label{eq:decomp}
\widehat{B}(\bida_{\text{ext}}) = \widehat{B}^-(\bida_{\text{ext}}) + \widehat{B}^+(\bida_{\text{ext}})  
\in {\rm Mat}_{(n+\tilde{r}) \times (n + \tilde{r})}(\ZZ)
\end{equation}
of $\widehat{B}(\bida_{\text{ext}})$ into a skew-symmetrizable part and a symmetrizable part. 
 It is easy to see from \cite[Proposition 3.28]{Wil:ensemble} that the sub-matrix $\widehat{B}(\bida) := \widehat{B}(\bida_{\text{ext}})_{\bida \times \bida} \in {\rm Mat}_{n \times n}(\ZZ)$ of
$\widehat{B}(\bida_{\text{ext}})$, related to the reduced double Bruhat cells $L_{\rm Ad}^{u,v}$ and $L^{u, v}$, 
takes the explicit form 
\[
	\begin{aligned}
		\widehat{B}(\bida)
		&= \frac{1}{2}a_{|\ida_j|,|\ida_k|}
		\Big(
			\{j[1],k[1]=+\infty\}
			+\ep_{k}\{j[1]=k\}
			-\ep_{j}\{k[1]=j\} \\
		&\qquad
			+\ep_{k}\{j<k<j[1]\}
			-\ep_{{k[1]}}\{j<k[1]<j[1]\}\{k[1]<+\infty\} \\
		&\qquad
			-\ep_{j}\{k<j<k[1]\}
			+\ep_{{j[1]}}\{k<j[1]<k[1]\}\{j[1]<+\infty\}
		\Big),
	\end{aligned}
\]
where $\{P\}$ is the Boolean function of statement $P$. We  call $\widehat{B}(\bida)$ the {\it nondegenerate cluster ensemble matrix} associated to $(A, \bida)$ and extend its definition to arbitrary signed word $\bida$. The matrix $\widehat{B}(\bida)$ also appears in \cite[Proposition 7.4]{RW:Gr} in the context of Grassmannians.

\begin{theorem}\label{thm:same-ensemble}
For any symmetrizable generalized Cartan matrix $A = (a_{i, i'})_{i, i' \in [1, r]}$, and for 
any signed word $\bida =(\ida_1, \ldots, \ida_n)$ in $[1, r]$, one has
\[
\widehat{B}(\bida) = \widehat{M}(\bida)  = E(\bida)^t Q(\bida) E(\bida).
\]
In particular, $\det(\widehat{B}(\bida))=1$.
\end{theorem}
\begin{proof}

This follows from a case by case comparison with \autoref{thm:hat-M-bida}. 
\end{proof}


\begin{remark}\label{rk:decomp}
 {\rm
Setting $\widehat{B}^-(\bida)$ and $\widehat{B}^+(\bida)$ to be the respective sub-matrices 
of $\widehat{B}^-(\bida_{\rm ext})$ and $\widehat{B}^+(\bida_{\rm ext})$ in (\ref{eq:decomp}) 
corresponding to $\bida \times \bida$, one has the 
decomposition 
\[
\widehat{B}(\bida) =\widehat{B}^-(\bida)+\widehat{B}^+(\bida)
\]
of $\widehat{B}(\bida)$ into a skew-symmetrizabble part and a symmetrizable part.
As $\widehat{B}(\bida) = \widehat{M}(\bida)$, 
\autoref{prop:decomp}  now gives matrix product formulas for both $\widehat{B}^-(\bida)$
and $\widehat{B}^+(\bida)$. 
\hfill $\diamond$
}
\end{remark}

\bibliographystyle{alpha}
\bibliography{LLM-Mut-Matrix}
\end{document}